\documentclass[12pt]{article}
\usepackage{amscd}     % nice arrows
\usepackage{amsmath}   % math stuff
\usepackage{amssymb}   % special symbols for math
\usepackage{amsthm}    % gives us the \newtheorem
\usepackage{hyperref}  % creates hyper referencing
\usepackage{bookmark}  % creates automatic bookmarks in the pdf by sections
\usepackage{bbm}
\usepackage{mathtools} % all sorts of math widgets
\usepackage{epstopdf}  % enables usage of png and eps in the same file
\usepackage{verbatim}  % if you want latex to read it exactly as you write it use verbatim environment
\usepackage{graphicx}  % standard for graphicx includes
\usepackage{color}     % enables font coloring
\usepackage[skip=2pt,font={small, it}]{caption} % enables caption customizations

\usepackage{algorithm} % Algorithms environment
\usepackage{algorithmicx}
\usepackage{multirow}
\usepackage{adjustbox}
\usepackage{caption}
\usepackage{subcaption}
\usepackage{algpseudocode}
\usepackage[normalem]{ulem}
\usepackage[title]{appendix}

\usepackage[useregional]{datetime2}
\usepackage[UKenglish]{babel}
\usepackage{xpatch}

\algrenewcommand\alglinenumber[1]{{\sffamily\footnotesize#1}}
\makeatletter
\xpatchcmd{\algorithmic}{\itemsep\z@}{\itemsep=0.5ex}{}{}
\makeatother

\usepackage{geometry}
\newtheorem{definition}{Definition}

\newtheorem{Theorem}{Theorem}[section]
\newtheorem{lemma}[Theorem]{Lemma}

\newtheorem{prop}[Theorem]{Proposition}
\theoremstyle{remark}

\newcommand{\RR}{\mathbb{R}}
\newcommand{\CC}{\mathbb{C}}
\newcommand{\ZZ}{\mathbb{Z}}

\newcommand{\OO}{\mathcal{O}}

\newcommand{\EE}{\mathbb{E}}
\newcommand{\wtilde}[1]{\widetilde{#1}}
\newcommand{\eps}{\varepsilon}

\newcommand{\dbtilde}[1]{\stackrel{\approx}{#1}}

\DeclareMathOperator*{\argmin}{\arg\!\min}

\DeclareMathOperator{\Ima}{Im}

\graphicspath{{./Figures/}}
\begin{document}
\title{ Rank-one Multi-Reference Factor Analysis}
\author{ Yariv Aizenbud${^{1,\#}}$~~Boris Landa${^{1,\#}}$~~Yoel Shkolnisky${^1}$\\
\small{${^1}$School of Mathematical Sciences, Tel Aviv University, Israel}\\
\small{${^\#}$These authors contributed equally}
}

\maketitle
\begin{abstract}
In recent years, there is a growing need for processing methods aimed at extracting useful information from large datasets. In many cases the challenge is to discover a low-dimensional structure in the data, often concealed by the existence of nuisance parameters and noise.
Motivated by such challenges, we consider the problem of estimating a signal from its scaled, cyclically-shifted and noisy observations. 
We focus on the particularly challenging regime of low signal-to-noise ratio (SNR), where different observations cannot be shift-aligned. We show that an accurate estimation of the signal from its noisy observations is possible, and derive a procedure which is proved to consistently estimate the signal. 
The asymptotic sample complexity (the number of observations required to recover the signal) of the procedure is $1/\operatorname{SNR}^4$. Additionally, we propose a procedure which is experimentally shown to improve the sample complexity by a factor equal to the signal's length. Finally, we present numerical experiments which demonstrate the performance of our algorithms, and corroborate our theoretical findings. 
\end{abstract}

\section{Introduction}
% Looking for low dim structure in large noisy data is interesting
Due to recent improvements in acquisition, storage, and processing capabilities, there is a growing need for techniques aimed at extracting useful information from large datasets.
It is commonplace to encounter large datasets of scientific observations, which are corrupted by noise and some deformation (e.g. translations, rotations, etc...)~\cite{965114, diamond1992multiple,rosen2016certifiably}.
In many cases, in these large datasets there is a hidden low-dimensional structure which is masked by the deformations and noise.

More formally, we present the following model for the observation $y \in \mathbb{C}^L$:
\begin{align}\label{eq:factor model}
\begin{aligned}
y &= \mathcal{R}\left\{x\right\} + \eta, \\
x &= \sum_{i=1}^r a_i \theta_i,
\end{aligned}
\end{align}
where $ \mathcal{R} $ is some random deformation operator, $\theta_i \in \mathbb{C}^L, \|\theta_i\| = 1$ for $i=1,\ldots,r$ are unknown deterministic orthonormal signals, $a_i\sim \mathcal{CN}(0,\lambda_i)$ is a complex-valued random scale factor with variance $ \lambda $, and $\eta\sim \mathcal{CN}(0, \sigma^2 I_L)$ is a complex-valued noise vector, with $ \mathcal{CN} $ being the circularly-symmetric complex normal distribution \cite{gallager2008circularly} (intuitively, $ a \sim \mathcal{CN}(0,\lambda) $ is equivalent to $ \operatorname{Im}(a) \sim \mathcal{N}(0,\lambda/2) $ and $ \operatorname{Re}(a) \sim \mathcal{N}(0,\lambda/2) $). We assume that the noise variance $\sigma^2$ is known. Given observations $y_i$ from the model \eqref{eq:factor model}, our goal is to estimate the signal $\theta$ and its strength~$\lambda$.

In this work, we consider a special prototype of the model \eqref{eq:factor model}. First, we take the random deformation $\mathcal{R}$ to be the cyclic shift operator $R_s\left\{\cdot\right\}$, that is, for any signal $\theta \in \mathbb{C}^L$ and any $s\in \left\{0,1,\ldots,L-1\right\}$, we define $\mathcal{R}_s$ by
\begin{equation}\label{eq:R_s def}
\mathcal{R}_s\left\{\theta\right\} [\ell] = \theta\left[\operatorname{mod}{(\ell-s,L)}\right],
\end{equation}
where $s$ is drawn from the uniform distribution over $\mathbb{Z}_L$ (i.e., $s$ is drawn from the random variable $S$ satisfying $Pr(S=s)=1/L, \; s\in \left\{0,1,\ldots,L-1\right\}$). We name this estimation problem \textit{Multi-Reference Factor Analysis} (MRFA), since it can be considered as factor analysis~\cite{child1990essentials} under the cyclic shift $\mathcal{R}_s$. In what follows, we drop the modulus by $L$ from all vector indices, as all vectors are considered as periodic. Second, we consider $x$ in~\eqref{eq:factor model} to be a rank one signal ($ r=1 $). Formally, in the above notation, we consider the model:
\begin{align}\label{eq:MRFA_model_def}
\begin{aligned}
y &= \mathcal{R}_s\left\{x\right\} + \eta, \\
x &= a \theta.
\end{aligned}
\end{align}
We name the problem of estimating $\theta$ and $\lambda$ from observations generated from the model~\eqref{eq:MRFA_model_def} \textit{rank-one MRFA}. Specifically, given $N$ independent observations $y_1,y_2,\ldots,y_N$ from the model~\eqref{eq:MRFA_model_def}, where $y_i =  R_{s_i}\left\{a_i\theta\right\} + \eta_i $, our goal is to recover $\theta$ and $\lambda$. We note that the random shift $s$ in~\eqref{eq:MRFA_model_def} is a nuisance parameter, and estimating its realizations $\left\{ s_i\right\}_{i=1}^N$ is of no interest in our model. Note that in the model~\eqref{eq:MRFA_model_def}, the signal $\theta$ may be estimated only up to an arbitrary cyclic shift and a product with a complex number of modulus one (global phase).
Even though our analysis is focused on the case of $a\sim \mathcal{CN}(0,\lambda)$, all of our statements can be easily adapted to the more general setting where $a$ admits an arbitrary distribution (either real or complex) with $\mathbb{E} \vert a \vert^2 = \lambda$ and a bounded fourth moment.
The algorithms derived in this work are also suitable for solving the rank-one MRFA problem in this more general setting.

% Connection to MRA
Note that if the random factor $a\sim \mathcal{CN}(0,\lambda)$ in~\eqref{eq:MRFA_model_def} is replaced by a constant, then the model~\eqref{eq:MRFA_model_def} reduces to Multi-Reference Alignment (MRA)~\cite{abbe2017sample, bandeira2014multireference, bendory2017bispectrum, perry2017sample}, which has recently drawn much interest. Both the MRA model and the model~\eqref{eq:MRFA_model_def} provide a simplified model for various problems in science and engineering, particularly in areas such as communications, radar, image processing, and structural biology~\cite{965114, diamond1992multiple, robinson2007optimal, rosen2016certifiably, theobald2012optimal, zwart2003fast}.%(noting that a cyclic shift of a signal is the one-dimensional analogue of an in-plane rotation of an image).

% Methods to solve the MRA problem
Having mentioned that our problem is a generalization of MRA, it is worthwhile to discuss the latter's possible solutions for different levels of noise. When the noise variance $\sigma^2$ is small, it is known that MRA can be solved by first estimating the relative shift between any two observations (taking the shift that maximizes the correlation between the two observations), then, aligning all of the observations, and finally averaging the aligned observations~\cite{bendory2017bispectrum}. Such a procedure is known to result in a sample complexity (the number of samples required for a prescribed error in estimating $\theta$) of $N\propto \sigma^2$. However, this approach fails when the noise variance $\sigma^2$ is large, as pairwise correlations become meaningless, and thus an alignment of the observations cannot be achieved.
In the high noise regime, several algorithms were proposed in~\cite{abbe2017multireference, bendory2017bispectrum,chen2018spectral,perry2017sample}, and were demonstrated to achieve the optimal asymptotic sample complexity of $N\propto \sigma^6$~\cite{perry2017sample} in the case of uniformly-distributed cyclic shifts, and of $N\propto \sigma^4$~\cite{abbe2017multireference} in the case of an arbitrary aperiodic distribution of the shifts. Indeed, such methods do not attempt to estimate the shifts of the observations (a task doomed to fail when the noise variance is large ~\cite{perry2017sample}), but rather use all  observations to estimate a sufficient number of shift-invariant statistics, from which the underlying signal can be recovered. It is shown~\cite{bendory2017bispectrum} that for the general case of the MRA problem, it is possible to recover the underlying signal using only the first three shift-invariant statistics: the mean (first order), the power spectrum (second order), and the bispectrum (third order) \cite{nikias1987bispectrum}, where the latter governs the achieved sample complexity.
The distinction between the two noise level regimes, high and low, applies analogously for samples generated by the model~\eqref{eq:MRFA_model_def}. In the low noise regime, $\theta$ in~\eqref{eq:MRFA_model_def} can be estimated by aligning the observations using correlations, followed by calculating the rank-one factorization of the aligned samples. Therefore, when studying the model~\eqref{eq:MRFA_model_def}, the regime of interest is the one of large noise variance.

% EM for MRA
A different approach for MRA, or more generally, for estimating model parameters in the presence of nuisance parameters is Maximum-Likelihood Estimation (MLE) via Expectation-Maximization (EM)~\cite{dempster1977maximum}, which marginalizes over the nuisance parameters.

Typically, and in particular in the context of MRA, EM suffers from two major shortcomings. The first is the lack of theoretical convergence guarantees (except in some special cases~\cite{wu1983convergence}), and the second is the phenomenon of extremely slow convergence -- resulting in very long running time, particularly when the noise variance is large~\cite{bendory2017bispectrum}. In contrast, estimators based on invariant statistics typically enjoy rigorous error bounds on one hand, and faster running times on the other (as they are essentially single-pass algorithms -- going over every observation only once).

% How to solve MRFA
Going back to our model of rank-one MRFA~\eqref{eq:MRFA_model_def}, and motivated by the above discussion, we consider the following question:
Can we accurately estimate $\lambda$ and $\theta$ of ~\eqref{eq:MRFA_model_def} in the regime of large $\sigma$ and large $N$?
We answer this question affirmatively, and moreover, propose an algorithm which, under mild conditions, is guaranteed to recover $\lambda$ and $\theta$ (up to the ambiguities of global phase and cyclic shift) when $N\rightarrow\infty$. The asymptotic sample complexity of our proposed algorithm is $N\propto \sigma^8$ (for large $\sigma$). As a by-product, we develop new concentration results for non-i.i.d. sub-exponential random vectors (and their sample covariance matrices), when the vectors admit a certain underlying structure (see Appendix~\ref{sec:appendix_thmStep1}).

% disscusssion on asymptotic complexity and connection to MRA's asymp. complexity
While the optimal asymptotic sample complexity of MRA is $ N\propto \sigma^6 $, the asymptotic complexity of our algorithm for MRFA is $ N\propto \sigma^8 $. The reason for the different rates is that the rank-one MRFA is fundamentally different from MRA, as the third-order invariant statistic, the \textit{bispectrum} (given by certain triplet correlations in the Fourier domain -- see Section~\ref{sec:moments}), which is used to solve MRA, vanishes under the setting of~\eqref{eq:MRFA_model_def}. Therefore, we resort to a fourth-order shift-invariant statistic, known as the \textit{trispectrum} (given by certain quadruplet correlations in the Fourier domain), and hence the dependence on $\sigma^8$ instead of $\sigma^6$. Since the bispectrum vanishes entirely, we believe it is unlikely that a rate better than $\sigma^8$ can be achieved (see~\cite{abbe2018estimation} for this type of argument in the case of MRA).

% more on sample complexity and SNR
Note that the previously-mentioned sample complexities are oblivious to the signal's length~$L$, as they assume~$L$ is a fixed constant. In practice, the length of the signal affects the sample complexity, and thus, we do not neglect it in our analysis. We analyze the sample complexity of our algorithms in terms of the signal's length~$L$, and in terms of the signal-to-noise ratio (SNR), defined as
\begin{equation}\label{eq:SNR_def}
\operatorname{SNR} = \frac{\mathbb{E}\left\Vert x \right\Vert^2}{\mathbb{E}\left\Vert \eta \right\Vert^2} = \frac{\lambda}{L\sigma^2}.
\end{equation}
We present in this work two algorithm that solve the rank-one MRFA problem. We show that the sample complexity of the first algorithm satisfies $N(L,\operatorname{SNR})= \OO ( L/\operatorname{SNR}^4)$ (in the regime of large $\sigma$), where $N$ is the number of required measurements for estimating $\theta$ and $\lambda$ to a given accuracy. We observed numerically that the achieved sample complexity of this algorithm is actually better by a factor of $L$, satisfying $N(L,\operatorname{SNR}) \propto 1/\operatorname{SNR}^4$. The second algorithm is a variant of the first one, which works better in practice at the cost of weaker theoretical bounds. It is observed to provide a further improvement by a factor of $L$, resulting in a sample complexity of $N\propto 1/( L\cdot \operatorname{SNR}^4 )$.

% Descrittion of numerical results
We demonstrate all our algorithms numerically, and show the agreement between the theory and the experimental results. We also compare our algorithms with the EM algorithm (designed for the rank-one MRFA problem) and show that it provides the same sample complexity as our methods (as a function of $\sigma$) with a marginal gain in the estimation error, while suffering from extremely slow running times and lack of theoretical guarantees.

% Paper structure
The paper is organized as follows.
Section~\ref{sec:method} describes our algorithms for recovering $\lambda$ and $\theta$ of~\eqref{eq:MRFA_model_def}, and provides the main theoretical guarantees for their performance.
Section~\ref{sec:moments} connects our approach with that of shift-invariant statistics. Section~\ref{sec:experimental} presents the numerical experiments supporting our theoretical derivations. Finally, Section~\ref{sec:summary} provides some concluding remarks and some possible future research directions.

\section{Method description and main results}\label{sec:method}
In this section, we describe our methods for estimating the model parameters $\lambda$ and $\theta$ of~\eqref{eq:MRFA_model_def}, and provide their theoretical error bounds and sample complexities.

Instead of working with the model~\eqref{eq:MRFA_model_def} directly, we consider an equivalent, more convenient formulation in the Fourier domain, where cyclic shifts are replaced by modulations. Let $F\in\mathbb{C}^{L\times L}$ be the unitary Discrete Fourier Transform (DFT) matrix
\begin{equation}\label{eq:DFT_def}
F[\ell,k] = \frac{1}{\sqrt{L}}\omega^{\ell k}, \qquad \omega = e^{-\imath 2\pi/L}, \qquad \ell,k = 0,\ldots,L-1,
\end{equation}
and denote the Fourier transforms of the quantities in~\eqref{eq:MRFA_model_def} by
\begin{equation}\label{eq:FT_params_def}
\hat{y} = F y, \qquad \hat{x} = F x, \qquad \hat{\theta} = F \theta, \qquad \hat{\eta} = F \eta.
\end{equation}
Then, a formulation equivalent to~\eqref{eq:MRFA_model_def}  in the Fourier domain is
\begin{align}
\begin{aligned}
\hat{y}[k] &= \omega^{s k} \hat{x}[k] + \hat{\eta}[k], \\ \hat{x}[k] &= a\hat{\theta}[k],
\end{aligned} \label{eq:MRFA_model_def_Fourier}
\end{align}
for $k=0,\ldots,L-1$, where $\hat{\eta}\sim \mathcal{CN}(0,\sigma^2 I_L)$ and $\Vert \hat{\theta} \Vert=1$ (since $F$ is unitary).
In what follows, we consider the problem of estimating $\lambda$ and $\hat{\theta}$ from the observations $\hat{y}_1,\ldots,\hat{y}_N$ ($\hat{y}_i=F y_i$), recalling that $\theta$ can always be obtained from $\hat{\theta}$ by the inverse Fourier transform, i.e.
\begin{equation}
\theta = F^{-1} \hat{\theta} =F^* \hat{\theta}.
\end{equation}
In the rest of this section, we describe our methods for estimating the signal parameters $\hat{\theta}$ and $\lambda$. We start by describing how to estimate the signal's strength $\lambda$ together with the magnitudes of $\hat{\theta}$ (i.e. $\vert \hat{\theta} \vert$) from the power spectrum of the observations. Then, we detail two methods for estimating the phases of $\hat{\theta}$ (i.e. $\operatorname{arg}\{\hat{\theta}\}$). The first method is simpler to implement and enjoys faster running times. The second method is able to exploit more information from the statistics calculated from the observations and results in lower estimation errors. We show that both methods are statistically consistent in estimating $\lambda$ and $\theta$ (or, equivalently, $ \hat{\theta} $) as $N\rightarrow \infty$, up to the inherent ambiguities for $\theta$ -- global phase and cyclic shift. We also show that in the low SNR regime (large $\sigma$), the first method admits a sample complexity of $N = \OO ( L/\operatorname{SNR}^4)$. Later, in Section~\ref{sec:experimental}, we demonstrate numerically  that the first method actually achieves a better sample complexity of $N = \OO ( 1/\operatorname{SNR}^4)$. The second method further improves upon this rate by a factor of $L$, with $N = \OO (1/(L \cdot \operatorname{SNR}^4))$.

\subsection{Estimating $\lambda$ and the magnitudes of $\hat{\theta}$} \label{subsec:magnitude est}
Consider the power spectrum of the random signal $x$ from~\eqref{eq:MRFA_model_def}, given by
\begin{equation}
p_x[k] = \mathbb{E}\left\vert \hat{x}[k] \right\vert^2 = \lambda \vert \hat{\theta}[k] \vert^2, \qquad k=0,\ldots,L-1, \label{eq:p_x def}
\end{equation}
and note that since $\Vert \hat{\theta} \Vert=1$, $ p_x $ encodes both $\lambda$ and the magnitudes
of $\hat{\theta}$ (i.e. $\vert \hat{\theta} \vert$) via
\begin{equation}\label{eq:power spectrum relations}
{\lambda} = \sum_{k=0}^{L-1} {p}_x[k], \qquad \vert \hat{\theta}[k] \vert = \sqrt{{p}_x[k]/{\lambda}} .
\end{equation}
The power spectrum $p_x$, and correspondingly $\lambda$, can be estimated from $\left\{ \hat{y}_i\right\}_{i=1}^N$ by
\begin{equation} \label{eq:power spectrum estimate}
\widetilde{p}_x[k] = \frac{1}{N}\sum_{i=1}^N \vert \hat{y}_i[k] \vert^2 - \sigma^2,  \qquad \wtilde{\lambda} = \sum_{k=0}^{L-1} \left|\widetilde{p}_x[k]\right|,
\end{equation}
which are consistent estimators (as $N\rightarrow\infty$) for $p_x$ and $\lambda$, respectively. Furthermore, $\widetilde{p}_x$ satisfies~\cite{wasserman2013all} % example 7.10
\begin{equation}
\widetilde{p}_x[k] = {p}_x[k] + \OO (\frac{\sigma^2}{\sqrt{N}}), \label{eq:power spectrum est error}
\end{equation}
where we discarded lower-order terms of $\sigma$ (since we are interested in the regime of $\sigma \to \infty$), and therefore
\begin{equation}\label{eq:lambda_est_error}
\wtilde{\lambda} = \lambda + \OO(\frac{L\sigma^2}{\sqrt{N}}).
\end{equation}
Hence, the sample complexity of estimating $\lambda$ and the magnitudes of $\hat{\theta}$ from $\wtilde{\lambda}$ and $\wtilde{p}_x$ is $N = \OO ( 1/\operatorname{SNR}^2 )$ in the regime of large noise variance.

\subsection{Estimating the phases of $\hat{\theta}$}
Next, we proceed to estimate the phases of the signal $\hat{\theta}$, that is, the vector $\operatorname{arg}\left\{ \hat{\theta}\right\}$.
Consider the vectors $u^{(m)}\in\mathbb{C}^L$, for $m=0,\ldots,L-1$, defined by
\begin{equation}\label{eq:um and zm def}
u^{(m)}[k] = \hat{\theta}[k]\hat{\theta}^*[k+m], \qquad k=0,\ldots,L-1,
\end{equation}
were $ \hat{\theta}^* $ is  the complex conjugate of $\hat{\theta}$. Essentially, each vector $u^{(m)}$ consists of the products between the elements of $\hat{\theta}$ with stride $m$, and thus encodes the phases of $\hat{\theta}$ through the relation
\begin{equation}\label{eq:um_arg}
\operatorname{arg}\left\{ u^{(m)}[k] \right\} = \operatorname{arg}\{ \hat{\theta}[k] \} - \operatorname{arg}\left\{ \hat{\theta}[k+m] \right\},
\end{equation}
up to an integer multiple of $2\pi$.
That is, the vectors $u^{(m)}$, $m=0,\ldots,L-1$, describe all pairwise differences between the phases of the elements of $\hat{\theta}$. We mention that $u^{(0)}$ does not provide any useful phase information (as it is equivalent to the power spectrum), and will play no role in what follows.

Before we describe how to extract the phases of $\hat{\theta}$ from the vectors $u^{(m)}$, we present a method for estimating $u^{(m)}$ from the observations $\hat{y}_1,\ldots,\hat{y}_N$. We define the stride-$m$ products of the elements of $\hat{y}$ of~\eqref{eq:MRFA_model_def_Fourier} by
\begin{equation}
z^{(m)}[k] = \hat{y}[k]\hat{y}^*[k+m], \label{eq:z^m def}
\end{equation}
and observe that by~\eqref{eq:MRFA_model_def_Fourier} we have
\begin{align}\label{eq:zm_expression}
z^{(m)}[k] = |a|^2 \omega^{-s m} u^{(m)}[k] + \epsilon^{(m)}[k],
\end{align}
where $\epsilon^{(m)}$ is a noise term given by
\begin{equation}
\epsilon^{(m)}[k] =  a \omega^{s k} \hat{\theta}[k] \hat{\eta}^*[k+m] + a^*\omega^{-s(k+m)}\hat{\theta}^*[k+m] \hat{\eta}[k] +  \hat{\eta}[k]\hat{\eta}^*[k+m]. \label{eq:epsilon error def}
\end{equation}
Note that if we had no noise, i.e. $\sigma=0$ (and hence $\epsilon^{(m)}=0$), then different realizations of $z^{(m)}$ would be equal to $u^{(m)}$ up to constant factors (as $|a|^2 \omega^{-s m}$ is independent of the frequency index $k$). We define the covariance matrix of $z^{(m)}$ as
\begin{equation}
C_z^{(m)} = \mathbb{E}\left[ z^{(m)} \left( z^{(m)}\right)^* \right], \label{eq:z cov def}
\end{equation}
and observe that in the case of no noise ($\sigma=0$), $C_z^{(m)}$ would be of rank one, with its leading eigenvector equal to $u^{(m)}$ (again -- up to a constant factor of known magnitude). We therefore proceed by computing the sample covariance matrices of $z^{(m)}$
\begin{equation}
\wtilde{C}_z^{(m)} = \frac{1}{N} \sum_{i=1}^N z^{(m)}_i \left( z^{(m)}\right)^*, \qquad m=1,\ldots,L-1, \label{eq:z sample cov def}
\end{equation}
and correct for the effect of the noise (see Appendix~\ref{sec:covariance of z_m}) by modifying the main diagonal of each $\wtilde{C}_z^{(m)}$ according to
\begin{equation}
\widetilde{C}_z^{(m)}[k,k] \gets \widetilde{C}_z^{(m)}[k,k] - \sigma^2 \left( \wtilde{p}_x[k] + \wtilde{p}_x[k+m]\right) -\sigma^4,
 \label{eq:z sample cov bias correction}
\end{equation}
where $\wtilde{p}_x$ is the estimate of the power spectrum of $x$ from~\eqref{eq:power spectrum estimate}.
Essentially,~\eqref{eq:z sample cov bias correction} corrects for the bias in $  \widetilde{C}_z^{(m)} $ caused by the noise term $\epsilon^{(m)}$. In Section~\ref{sec:moments} below, we relate the sequence of matrices $\{ C_z^{(m)}\}$, $m=0\ldots L-1$, of~\eqref{eq:z cov def} with the trispectrum of $y$ (for the definition of the trispectrum see also Section~\ref{sec:moments}).

Once $\widetilde{C}_z^{(m)}$ has been computed (for any $ m $), we take its leading normalized eigenvector $\wtilde{u}^{(m)}$ (corresponding to the largest eigenvalue of $\widetilde{C}_z^{(m)}$) as an estimate for $u^{(m)}$.
Note that the subtraction of $\sigma^4$ in~\eqref{eq:z sample cov bias correction} has no effect on the eigenvectors of $\widetilde{C}_z^{(m)}$, and therefore has no effect on our estimate for $u^{(m)}$.

Note that by the definition of $u^{(m)}$ (see also~\eqref{eq:um_arg}), it satisfies
\begin{equation}\label{eq:umphasesum}
\sum_{k=0}^{L-1} \arg\{u^{(m)}[k]\} = 0,
\end{equation}
whereas $ \wtilde{u}^{(m)} $ does not necessarily satisfy~\eqref{eq:umphasesum}, since $ \wtilde{u}^{(m)} $ estimates $u^{(m)}$ up to an arbitrary phase. We can easily update $ \wtilde{u}^{(m)} $ to satisfy $ \sum_{k=0}^{L-1} \arg\{\wtilde{u}^{(m)}[k]\} = 0 $ by
\begin{equation}\label{eq:umupdate}
\wtilde{u}^{(m)}  \gets \wtilde{u}^{(m)} \cdot \exp\left\{-\frac{\imath}{L} \sum_{k=0}^{L-1} \arg\{\wtilde{u}^{(m)}[k]\}\right\}.
\end{equation}
Note that after the update~\eqref{eq:umupdate}, $ \wtilde{u}^{(m)} $ is unique up to a multiplication by $ e^{\imath 2\pi j/L} $ for $j \in\{0,\ldots, L-1\}$.
One approach to try and solve this ambiguity is to multiply $ \wtilde{u}^{(m)} $ by a phase, such that $ \arg\{\wtilde{u}^{(m)}[0]\} = 0 $. Unfortunately, this will violate requirement $ \sum_{k=0}^{L-1} \arg\{\wtilde{u}^{(m)}[k]\} = 0 $.
We fix this ambiguity while satisfying~\eqref{eq:umphasesum} by choosing $j_{min}$ such that $ \operatorname{arg}\{\wtilde{u}^{(m)}[0] e^{\imath 2\pi j_{min}/L}\} $ is minimal, and then we update $ \wtilde{u}^{(m)} $ again by
\begin{equation}\label{eq:umupdate2}
\wtilde{u}^{(m)}  \gets \wtilde{u}^{(m)} \cdot e^{\imath 2\pi j_{min}/L}.
\end{equation}

Next, we characterize the accuracy of the estimate $\wtilde{u}^{(m)}$ in the low SNR regime. For the sake of clarity and simplicity of notation in the proof, the result is stated in Lemma~\ref{lem:u_1 sigma_series_bound} and  Theorem~\ref{thm: step 1 A4 err} for the case of $ m=1 $.

\begin{lemma}\label{lem:u_1 sigma_series_bound}
	Let $u^{(1)}$ be given by~\eqref{eq:um and zm def}, let $\wtilde{u}^{(1)}$ be the leading eigenvector of $\wtilde{C}_z^{(1)}$ of~\eqref{eq:z sample cov bias correction}, and define $\gamma_1 = L \left\Vert u^{(1)} \right\Vert^2$ and $ \delta_1 = \min_{k\in\{0,\ldots,L-1\}}|u^{(1)}[k]| $. If $\gamma_1,\delta_1>0$ and $N>2L$ is large enough, then the error in the estimate $ \wtilde{u}^{(1)} $ of $ \frac{u^{(1)}}{\left\Vert u^{(1)}\right\Vert} $ can be bounded by
	\begin{equation*}
	\left\Vert \wtilde{u}^{(1)} - e^{-\imath 2\pi s_1 /L } \frac{u^{(1)}}{\left\Vert u^{(1)} \right\Vert} \right\Vert \leq \|A_0\| + \sigma^2 \|A_2\|+ \sigma^3 \|A_3\|+ \sigma^4 \|A_4\|,
	\end{equation*}
	where $ A_0, A_2, A_3 $ and $ A_4 $ are independent of $ \sigma $, for some $s_1\in \{0,\ldots,L-1\}$.
\end{lemma}
The proof is provided in Appendix~\ref{subsec:proof_lem_u_1 sigma_series_bound}.
\begin{Theorem}
	\label{thm: step 1 A4 err}
	Let $u^{(1)}$ be given by~\eqref{eq:um and zm def}, let $\wtilde{u}^{(1)}$ be the leading eigenvector of $\wtilde{C}_z^{(1)}$ of~\eqref{eq:z sample cov bias correction}, and define $\gamma_1 = L \left\Vert u^{(1)} \right\Vert^2$ and $ \delta_1 = \min_{k\in\{0,\ldots,L-1\}}|u^{(1)}[k]| $. Then, if $\gamma_1,\delta_1>0$ and $N>2L$ is large enough, there exist constants $C_1,C_2,C_3,c_1,c_2 > 0$ (independent of $u^{(1)},N,L$ and $\sigma$), such that for $ A_4 $ from Lemma~\ref{lem:u_1 sigma_series_bound}
	\begin{equation}\label{eq:step 1 A4 err bound}
	\|A_4\| \leq \frac{C_3}{\gamma_1\delta_1^2} \sqrt{\frac{L^3}{\lambda^4 N}},
	\end{equation}
	with probability at least
	\begin{equation}
	1 - C_1 L e^{-c_1 {L}^{1/2}} - C_2 N e^{-c_2 N^{1/4}}. \label{eq:step 1 prob}
	\end{equation}
\end{Theorem}
The proof is provided in Appendix~\ref{subsec:proof_step1 _A4 err}.
Note that the probability in~\eqref{eq:step 1 prob} tends to $1$ rapidly as $N,L\rightarrow \infty$ (dominated by an exponential rate). 

In essence, Lemma~\ref{lem:u_1 sigma_series_bound} states that the accuracy of the estimate $\wtilde{u}^{(m)}$ can be bounded by a degree-4 polynomial in  $ \sigma $. In the large noise regime ($ \sigma \gg 1 $), the dominant part of the error in the estimate of $\wtilde{u}^{(m)}$ is $\|A_4\|$, and thus we have Theorem~\ref{thm: step 1 A4 err} that bounds $\| A_4\| $ with high probability. We argue in Appendix~\ref{sec:non-asymp_u} that with high probability, in the regime of large $ N $ and large $ \sigma $
\begin{equation} \label{eq:utilde_err_propto}
\left\Vert \wtilde{u}^{(m)} - e^{-\imath 2\pi s_1 /L } \frac{u^{(m)}}{\left\Vert u^{(m)} \right\Vert} \right\Vert_2 \propto \sqrt{\frac{1}{ N L \cdot\operatorname{SNR}^4}},
\end{equation}
where SNR is defined in~\eqref{eq:SNR_def} and $ \gamma_m $ is defined analogously to $ \gamma_1 $ by $\gamma_m = L \left\Vert u^{(m)} \right\Vert^2$.
In other words, in the regime of low SNR, the number of samples $ N $ required to estimate $ u^{(m)} $ by $\wtilde{u}^{(m)}$ to a constant prescribed error is
\begin{equation}
N = \OO \left( \frac{1}{ L \cdot\operatorname{SNR}^4} \right).
\end{equation}

Next, we show how to estimate the phases of $ \hat{\theta} $ using the estimates $ \wtilde{u}^{(m)}  $.

\subsubsection{Phase estimation by frequency marching}
Note that $u^{(1)}$ of~\eqref{eq:um and zm def} is sufficient to recover the phases of $\hat{\theta}$, since (up to an integer multiple of~$2\pi$)
\begin{equation}
\operatorname{arg}\left\{ \hat{\theta}[k+1] \right\} = \operatorname{arg}\left\{ \hat{\theta}[k] \right\} - \operatorname{arg}\left\{ u^{(1)}[k] \right\}. \label{eq:frequency marching}
\end{equation}
Therefore, the phases of $\hat{\theta}$ can be obtained by first initializing the phase of $\hat{\theta}[0]$ to zero (recalling that we have an ambiguity of a global phase and a cyclic shift, so we can initialize the phase of $\hat{\theta}[0]$ arbitrarily), followed by applying~\eqref{eq:frequency marching} repeatedly for $k=0,\ldots,L-2$. We refer to this procedure as \textit{frequency marching} (FM) (see~\cite{bendory2017bispectrum} for a similar technique baring the same name). Using $\widetilde{u}^{(1)}$ (the leading eigenvector of $\widetilde{C}^{(1)}_z$ of~\eqref{eq:z sample cov bias correction}) as an estimate of $u^{(1)}$,  the explicit formula for the estimated signal $\wtilde{\theta}$ (combining the amplitude given by~\eqref{eq:power spectrum relations} with the phases given by~\eqref{eq:frequency marching}) is given by
\begin{equation}
\widetilde{\theta} = \widetilde{\theta}_m \odot \widetilde{\theta}_p
\label{eq:frequency marching theta est formula}
\end{equation}
where $\odot$ denotes element-wise multiplication (the Hadamard product), $ \wtilde{\theta}_m $ is the magnitude part (given by~\eqref{eq:power spectrum relations})
\begin{equation}
\widetilde{\theta}_m[k] = \sqrt{{\widetilde{p}_x[k]}/{\widetilde{\lambda}}} ~~~k\in\{0,\ldots, L-1\},
\label{eq:frequency marching theta est formula_amp}
\end{equation}
and  $ \wtilde{\theta}_p $ is the phase part (given by~\eqref{eq:frequency marching})
\begin{equation}
\widetilde{\theta}_p[k] =
\begin{dcases}
1, & k=0, \\
\exp\left\{ -\imath \sum_{\ell=0}^{k-1} \operatorname{arg}\left\{ \widetilde{u}^{(1)}[\ell] \right\} \right\}, & k>0.
\end{dcases} \label{eq:frequency marching theta est formula_phase}
\end{equation}

The algorithm for estimating the signal parameters $\lambda$ and $\theta$ using frequency marching is summarized in Algorithm~\ref{alg:MRFA_formal}.

Algorithm~\ref{alg:MRFA_formal} clearly does not exploit all the available information. While the method presented in the next section exploits more complicated statistics of the observations and provides lower estimation errors, Algorithm~\ref{alg:MRFA_formal} is simpler to implement and enjoys faster running times (then the algorithm of the next section). For Algorithm~\ref{alg:MRFA_formal} (unlike the method of the next section) we show that in the low SNR regime it admits a sample complexity of $N = \OO ( L/\operatorname{SNR}^4)$.

\begin{algorithm}[H]
	\caption{Rank-one MRFA by frequency marching (MRFA-FM)
	}
	\label{alg:MRFA_formal}
		\begin{algorithmic}[1]
		\Statex {\bfseries Inputs:} Observations $y_i \in \CC^L$, $i=1,\ldots, N$, from~\eqref{eq:MRFA_model_def}, and the noise variance $\sigma^2$.
		\Statex {\bfseries Outputs:} Estimates $ \widetilde{\theta}$ and $\widetilde{\lambda} $ for the signal parameters ${\theta}$ and ${\lambda}$.
		\State $ \hat{y}_i \gets F y_i, \quad i=1,\ldots,N$ \Comment{Fourier transform of the observations}		
		\State $\widetilde{p}_x[k] \leftarrow  \frac{1}{N}\sum_{i=1}^{N}|\hat{y}_i[k]|^2  - \sigma^2, \quad k=0,\ldots,L-1$ \Comment{Power spectrum estimation}		
		\State $\widetilde{\lambda} \gets \sum_{k=0}^{L-1} \left|\widetilde{p}_x[k]\right|$ \Comment{Estimation of $\lambda$}			
		\State $z^{(1)}_i[k] \gets \hat{y}_i[k]\hat{y}^*_i[k+1], \quad i=1,\ldots,N, \quad k=0,\ldots,L-1$ \Comment{Stride-$1$ products}
		\State $\widetilde{C}_z^{(1)}\leftarrow \frac{1}{N} \sum_{i=1}^N z_i^{(1)} \left( z_i^{(1)} \right)^*$ \label{algstep:Cz_algFM}\Comment{Sample covariance of $z^{(1)}$}		
		\State $\widetilde{C}_z^{(1)}[k,k]\leftarrow \widetilde{C}_z^{(1)}[k,k] - \sigma^2 \left( \wtilde{p}_x[k]+ \wtilde{p}_x[k+1]\right), \quad k=0,\ldots,L-1$ \Comment{Bias correction}
		\State $\wtilde{u}^{(1)} \gets $ the first eigenvector of $\widetilde{C}_z^{(1)}$ \Comment{Estimate for $u^{(1)}$}\label{algstep:tilde_u_algFM}
		\State $\wtilde{u}^{(1)}  \gets \wtilde{u}^{(1)} \cdot \exp\left\{-\imath/L \sum_{k=0}^{L-1} \arg\{\wtilde{u}^{(1)}[k]\}\right\}$ \Comment{$ u^{(1)} $ phase correction~\eqref{eq:umupdate}}
		\State $\wtilde{u}^{(1)}  \gets \wtilde{u}^{(1)} \cdot e^{\imath 2\pi j_{min}/L}$ \Comment{$ u^{(1)} $ phase correction~\eqref{eq:umupdate2}}
		\State $\wtilde{\theta}_m[k]\gets \sqrt{\wtilde{p}_x[k]/\wtilde{\lambda}}, \quad k=0,\ldots,L-1$ \label{algstep:algFM_fixPower} \Comment{Magnitude estimation}
		\State $\wtilde{\theta}_p[0]\gets 1$
		\For {$k=0,\ldots,L-2$} \label{algstep:FM}
		\State $\wtilde{\theta}_p[k+1]\gets \wtilde{\theta}_p[k] \cdot \exp\left\{-\imath\cdot \operatorname{arg}\left\{ \widetilde{u}^{(1)}[k]\right\} \right\}$\label{algstep:freqMarch} \Comment{Frequency marching }
		\EndFor
		\State $\wtilde{\theta}\gets \wtilde{\theta}_p \odot \wtilde{\theta}_m$ \Comment Combine magnitude and phase estimates for $\hat{\theta}$
		\State $\wtilde{\theta} \gets {F}^{*} \wtilde{\theta}$ \Comment{Inverse Fourier transform}\label{algstep:final theta_tilde}
		\State\Return $\wtilde{\theta},\wtilde{\lambda}$
	\end{algorithmic}
\end{algorithm}

The next theorem establishes that Algorithm~\ref{alg:MRFA_formal} results in a consistent estimate for $\theta$ as $N\rightarrow \infty$, up to the ambiguities of a phase factor and a cyclic shift.

\begin{Theorem}[Consistency of Algorithm~\ref{alg:MRFA_formal}] \label{thm:alg 1 consistency}
Let $\theta$ be from the model~\eqref{eq:MRFA_model_def} and suppose that $| \hat{\theta}[k] |>0$ for all $k=0,\ldots,L-1$. Let  $\wtilde{\theta}$ be from Algorithm~\ref{alg:MRFA_formal}. Then, there exists an integer ${s}\in\mathbb{Z}_L$ and a constant $\alpha\in\CC$ satisfying $\left\vert \alpha \right\vert = 1$, such that
\begin{equation}\label{eq:est err lim}
\lim_{N\rightarrow \infty} \left\Vert \theta - \alpha \mathcal{R}_{{s}} \{ \wtilde{\theta} \} \right\Vert_2 \underset{\textbf{a.s.}}{=} 0,
\end{equation}
where $\mathcal{R}_{{s}} \left\{ \cdot \right\}$ is a cyclic shift by ${s}$.
\end{Theorem}
The proof is provided in Appendix~\ref{sec:proof of consistency of alg 1}.

Additionally, we have the following theorem, which characterizes the  asymptotic estimation error of Algorithm~\ref{alg:MRFA_formal} in the low SNR regime.

\begin{Theorem}[Asymptotic error of Algorithm~\ref{alg:MRFA_formal}] \label{thm:alg 1 asym err}
	Let $\theta$ be from the model~\eqref{eq:MRFA_model_def} and $\wtilde{\theta}$ be from Algorithm~\ref{alg:MRFA_formal}. Define $\gamma_1 = L \left\Vert u^{(1)} \right\Vert^2$ and $ \delta_1 = \min_{k\in\{0,\ldots,L-1\}}|u^{(1)}[k]| $ where $u^{(1)}$ is given by~\eqref{eq:um and zm def}. If $\delta_1>0$ and $N>2L$ is large enough, there exists an integer $s\in\mathbb{Z}_L$, and a constant $\alpha\in\CC$ satisfying $\left\vert \alpha\right\vert=1$, such that
	\begin{equation}\label{eq:err_b_i_bound}
	\left\Vert \theta - \alpha \mathcal{R}_{{s}} \left\{ \wtilde{\theta} \right\} \right\Vert \leq b_0  + \sigma^2 b_2 +\sigma^3 b_3 +\sigma^4 b_4,
	\end{equation}
	where $ b_0,b_2,b_3,b_4 $ are independent of $ \sigma $.
	Additionally, there exist constants $C_1,C_2,C_4,c_1,c_2 > 0$ (independent of $\theta,N,L$ and $\sigma$), such that
	\begin{equation}\label{eq:alg 1 asym err bound}
	b_4 \leq \frac{\sigma^4 C_4}{\gamma_1 \delta_1^4} \sqrt{\frac{L^5}{\lambda^4 N}},
	\end{equation}
	with probability at least
	\begin{equation}
	1 - C_1 L e^{-c_1 {L}^{1/2}} - C_2 N e^{-c_2 N^{1/4}}. \label{eq:alg 1 prob}
	\end{equation}
\end{Theorem}

The proof is provided in Appendix~\ref{sec:proof of alg 1 asym err}.
By bounding $b_{0},~b_{2}$, and $b_{3}$ in \eqref{eq:err_b_i_bound} along the lines of Appendix~\ref{sec:non-asymp_u}, we get from Theorem~\ref{thm:alg 1 asym err} that
for large $N$ and $\sigma$, with high probability we have
\begin{equation}
\left\Vert \theta - \alpha \mathcal{R}_{{s}} \left\{ \wtilde{\theta} \right\} \right\Vert_2 \propto \sqrt{\frac{L}{ N \cdot\operatorname{SNR}^4}},
\end{equation}
or, that in the regime of low SNR, the number of samples $ N $ required by Algorithm~\ref{alg:MRFA_formal} to estimate $ \theta $ to a constant prescribed error (the sample complexity) is
\begin{equation}\label{eq:N_asymp_L_SNR}
N = \OO \left( \frac{L}{\operatorname{SNR}^4} \right).
\end{equation}

In Section~\ref{sec:experimental}, we demonstrate numerically that the sample complexity achieved by Algorithm~\ref{alg:MRFA_formal}  is actually better than~\eqref{eq:N_asymp_L_SNR} by a factor of $L$, i.e. $N\propto 1/\operatorname{SNR}^4$.

\subsubsection{Phase estimation by alternating minimization}
Next, we propose a method for estimating the phases of $\hat{\theta}$ using all vectors $\wtilde{u}^{(m)}$, $m=1,\ldots,L-1$, simultaneously. To that end, the main observation is that the vector $u^{(m)}$ of~\eqref{eq:um and zm def} is the $m$'th diagonal of $ \hat{\theta} \hat{\theta^{*}} $. Therefore, we can use all vectors $u^{(m)}$ to construct $ \hat{\theta} \hat{\theta^{*}} $, followed by estimating $ \hat\theta $ from the leading eigenvector of $ \hat{\theta} \hat{\theta^{*}} $. In practice, as we only have access to the estimated vectors $\wtilde{u}^{(m)}$, which estimate ${u}^{(m)}$ up to unknown phase factors, we can only approximate the matrix $\hat{\theta} \hat{\theta}^*$ up to an element-wise product with a circulant matrix (of the unknown phases). Therefore, we propose to estimate $\hat{\theta}$ and the unknown phase factors in $\wtilde{u}^{(m)}$ simultaneously. As we try to obtain the phases of $\hat{\theta}$, we will use only the phases of $\wtilde{u}^{(m)}$ throughout this procedure.

We start by forming the $L \times L$ matrix $\wtilde{C}_x$ as
\begin{equation} \label{eq:C_x def}
\wtilde{C}_x[k_1,k_2] =
\begin{dcases}
1 &, k_1=k_2, \\
\frac{\wtilde{u}^{(k_2-k_1)\operatorname{mod} L}[k_1]}{\left\vert \wtilde{u}^{(k_2-k_1)\operatorname{mod} L}[k_1] \right\vert } &, k_1 \neq k_2.	
\end{dcases}
\end{equation}
That is, $\wtilde{C}_x$ has ones on its main diagonal, and the phases of the vector $ \wtilde{u}^{(m)}$ on its $m$'th diagonal (with circulant wrapping for each diagonal). Note that by the structure of $\wtilde{C}_x$ in~\eqref{eq:C_x def} and by the definition of $\wtilde{u}^{(m)}$, the matrix $\wtilde{C}_x$ is self-adjoint.
Thus, if we have no noise, i.e $\sigma=0$, it follows that $\wtilde{C}_x$ can be written as the element-wise product of a rank-one matrix, and a circulant matrix of phases (since $\wtilde{u}^{(m)}$ estimates ${u}^{(m)}$ up to an unknown constant phase factor). Specifically, if $\sigma = 0$ we have
\begin{equation}
\wtilde{C}_x = \hat{\theta}_p \hat{\theta}_p^* \odot \operatorname{Circul}\left\{ \alpha \right\}, \label{eq:C_x observation no noise}
\end{equation}
where $\hat{\theta}_p\in\mathbb{C}^L$ is the phase part of $\hat{\theta}$, i.e. $\hat{\theta}_p[k] = \hat\theta[k]/\left\vert\hat\theta[k]\right\vert$, $\odot$ denotes element-wise multiplication (the Hadamard product), and $\operatorname{Circul}\left\{ \alpha \right\}$ is a circulant matrix constructed from a vector of phases $\alpha\in\mathbb{C}^L$, i.e.
\begin{equation}
\operatorname{Circul}\left\{ \alpha \right\} [k1,k2] = \alpha[\left(k_2 - k_1\right)\operatorname{mod} L ].
\end{equation}

Motivated by the observation in~\eqref{eq:C_x observation no noise} for the case of no noise ($\sigma=0$), we propose to recover $ q $ by solving
\begin{equation}
\min_{q,\alpha\in\mathbb{C}^L,\; \left\vert\alpha[k]\right\vert=1} \left\Vert q q^* \odot \operatorname{Circul}\left\{ \alpha \right\} - \wtilde{C}_x \right\Vert_F. \label{eq:am opt problem}
\end{equation}
Now, since~\eqref{eq:am opt problem} is non-convex, we propose to solve it by alternating minimization, where we alternate between solving for $q$ and solving for $\alpha$. When $\alpha$ is held fixed, the solution for $q$ is given simply by the singular value decomposition (SVD) of the matrix $\wtilde{C}_x \odot \operatorname{Circul}\left\{ \alpha^* \right\}$. That is,
\begin{equation}
q = s_1 v_1, \label{eq:am q step}
\end{equation}
where $s_1$ is the largest singular value of $\wtilde{C}_x \odot \operatorname{Circul}\left\{ \alpha^* \right\}$, and $v_1$ is its corresponding singular vector (left and right singular vectors are equal up to a sign since the matrix $\wtilde{C}_x \odot \operatorname{Circul}\left\{ \alpha^* \right\}$ is self-adjoint). When $q$ is held fixed, the optimization problem for $\alpha$ reduces to a least-squares problem with absolute-value constraints, whose solution is given explicitly by
\begin{equation}
\alpha[k] = \exp\left\{ \imath \cdot \operatorname{arg}\left\{ \sum_{\ell=0}^{L-1} q^*[\ell]q[\ell+k] \wtilde{C}_x[\ell,\ell+k] \right\} \right\}, \qquad k=0,\ldots,L-1, \label{eq:am alpha step}
\end{equation}
where all matrix and vector index assignments in~\eqref{eq:am alpha step} are modulo $L$.
We alternate between computing the solutions for $\alpha$ and for $q$ iteratively until the objective in~\eqref{eq:am opt problem} reaches saturation.
We summarize the algorithm for estimating $ \theta $ and $ \lambda $ of~\eqref{eq:MRFA_model_def} using alternating minimization in Algorithm~\ref{alg:MRFA}.

\begin{algorithm}[H]
	\caption{Rank-one MRFA by alternating minimization (MRFA-AM) }
	\label{alg:MRFA}
	\begin{algorithmic}[1]
		\Statex {\bfseries Inputs:} Observations $y_i \in \CC^L$, $i=1,\ldots, N$, from~\eqref{eq:MRFA_model_def}, the noise variance $\sigma^2$, and the number of iterations $\tau$.
		\Statex {\bfseries Outputs:} Estimates $ \widetilde{\theta}$ and $\widetilde{\lambda} $ for the signal parameters ${\theta}$ and ${\lambda}$.
		\State $ \hat{y}_i \gets F y_i, \quad i=1,\ldots,N$ \Comment{Fourier transform of the observations}		
		\State $\widetilde{p}_x[k] \leftarrow  \frac{1}{N}\sum_{i=1}^{N}|\hat{y}_i[k]|^2  - \sigma^2, \quad k=0,\ldots,L-1$ \Comment{Power spectrum estimation}		
		\State $\widetilde{\lambda} \gets \sum_{k=0}^{L-1} \left|\widetilde{p}_x[k]\right|$ \Comment{Estimation of $\lambda$}			
		\State $\wtilde{\theta}_m[k]\gets \sqrt{\wtilde{p}_x[k]/\wtilde{\lambda}}, \quad k=0,\ldots,L-1$ \Comment{Magnitudes estimation} \label{algstep:algAM_fixPower}	
		\For {$m=1,\ldots,L-1$}
		\State $z^{(m)}_i[k] \gets \hat{y}_i[k]\hat{y}^*_i[k+m], \quad i=1,\ldots,N, \quad k=0,\ldots,L-1$ \Comment{Stride-$m$ products}
		\State  $C_z^{(m)}\leftarrow \frac{1}{N} \sum_{i=1}^N z_i^{(m)} \left( z_i^{(m)} \right)^*$ \label{algstep:Cz_algAM} \Comment{Sample covariance of $z^{(m)}$}		
		\State  $\widetilde{C}_z^{(m)}[k,k]\leftarrow \widetilde{C}_z^{(m)}[k,k] - \sigma^2 \left(\wtilde{p}_x[k] + \wtilde{p}_x[k+m]\right), \quad k=0,\ldots,L-1$ \Comment{Bias correction}
		\State  $\wtilde{u}^{(m)} \gets $ the first eigenvector of $\widetilde{C}_z^{(m)}$ \Comment{Estimation of $u^{(m)}$}	
		\EndFor
		\State Form the matrix $\widetilde{C}_x$ of~\eqref{eq:C_x def}\label{algstep:form_Cx_algAM}
		\State $\wtilde{\alpha}_0[k] \gets \exp \left\{\imath 2\pi {\phi}[k] \right\}, \quad {\phi}[k]\sim U[0,1),\quad k=0,\ldots,L-1$ \Comment{Initialize $\alpha$}
		\For {$t=1,\ldots,\tau$} \Comment{Estimate phases by alternating minimization}\label{algstep:alternating for start}
		\State $\wtilde{q}_{t}\gets \argmin_{q\in\mathbb{C}^{L}}{\left\Vert {q}{q}^* - \widetilde{C}_x \odot \operatorname{Circul}\left\{\wtilde{\alpha}_{t-1}^* \right\} \right\Vert_F}$ \Comment See~\eqref{eq:am q step}\label{algstep:altmin_step1}
		\State $\wtilde{\alpha}_{t}\gets\argmin_{\alpha\in\mathbb{C}^{L},\; \left\vert \alpha[k]\right\vert =1}{\left\Vert \wtilde{q}_{t}\wtilde{q}_{t}^* \odot \operatorname{Circul}\left\{\alpha \right\} - \widetilde{C}_x \right\Vert_F}$  \Comment See~\eqref{eq:am alpha step}		
		\EndFor \label{algstep:alternating for end}
		\State $\wtilde{\theta}_p[k]\gets {\wtilde{q}_\tau[k]} /  {\left\vert \wtilde{q}_\tau[k] \right\vert}  $	\Comment Estimation of the phases \label{algstep:algAM_fixPhases}
		\State $\wtilde{\theta}\gets \wtilde{\theta}_p \odot \wtilde{\theta}_m$	\Comment Combine magnitude and phase estimates for $\hat{\theta}$ \label{algstep:est theta hat}
		\State $\wtilde{\theta} \gets {F}^{*} \wtilde{\theta}$ \Comment{Inverse Fourier transform} \label{algstep:invFT_algAM}
		\State\Return $\wtilde{\theta},\wtilde{\lambda}$
	\end{algorithmic}
\end{algorithm}

The following theorem states that even a single iteration of the alternating minimization in Algorithm~\ref{alg:MRFA} lines~\ref{algstep:alternating for start}-\ref{algstep:alternating for end} (i.e. $\tau=1$) results in a consistent estimate for $\theta$ as $N\rightarrow \infty$ (up to the inherent ambiguities of phase and cyclic shift).

\begin{Theorem}[Consistency of Algorithm~\ref{alg:MRFA}] \label{thm:alg 2 consistency}
Let $\theta$ be from the model~\eqref{eq:MRFA_model_def}, $\wtilde{\theta}$ be from Algorithm~\ref{alg:MRFA} with $\tau=1$, and suppose that $| \hat{\theta}[k] |>0$ for all $k=0,\ldots,L-1$. Then, there exists an integer ${s}\in\mathbb{Z}^L$ and a constant $\alpha\in\CC$ satisfying $\left\vert \alpha \right\vert = 1$, such that
\begin{equation}
\lim_{N\rightarrow \infty} \left\Vert \theta - \alpha \mathcal{R}_{{s}} \left\{ \wtilde{\theta} \right\} \right\Vert_2 \underset{\textbf{a.s.}}{=} 0,
\end{equation}
where $\mathcal{R}_{{s}} \left\{ \cdot \right\}$ is a cyclic shift by ${s}$.
\end{Theorem}
The proof is provided in Appendix~\ref{sec:proof of consistency of alg 2}.

In terms of sample complexity, since Algorithm~\ref{alg:MRFA} estimates the phases of $\hat{\theta}$ using all vectors $\wtilde{u}^{(1)},\ldots,\wtilde{u}^{(L-1)}$ simultaneously, we expect it to perform better than Algorithm~\ref{alg:MRFA_formal}, which only uses $\wtilde{u}^{(1)}$. In particular, the alternating minimization in Algorithm~\ref{alg:MRFA} estimates $2L$ unknown parameters ($L$ parameters for $q$, and $L$ for the nuisance phases $\alpha$) from $O(L^2)$ measurements (the matrix $\wtilde{C}_x$), whereas Algorithm~\ref{alg:MRFA_formal} estimates $L$ parameters (the phases of $\hat{\theta}$) from only $L$ measurements (the vector $u^{(1)}$). Even though the error vectors $\epsilon^{(m)}$ (of~\eqref{eq:epsilon error def}) are not independent for different $m$, it is easy to verify that they are uncorrelated, and therefore, it is reasonable to expect that the error in estimating $\theta$ by Algorithm~\ref{alg:MRFA} would improve upon that of Algorithm~\ref{alg:MRFA_formal} (since we use more information). Indeed, we demonstrate numerically in Section~\ref{sec:experimental} that estimating $\theta$ by Algorithm~\ref{alg:MRFA} achieves the sample complexity of~\eqref{eq:utilde_err_propto}, i.e.,
\begin{equation} \label{eq:sampleComplexityU2}
N\propto \frac{1}{L\cdot \operatorname{SNR}^4},
\end{equation}
which improves upon the sample complexity of Algorithm~\ref{alg:MRFA_formal} by a factor of $L$.

\section{Invariant statistics and Trispectrum inversion}\label{sec:moments}
While the previous section provides a self-contained derivation of our algorithms, there are several advantages to present a more systematic approach to the MRFA problem. First, it is worthwhile to relate our algorithms with the method of moments (see \cite{wasserman2013all}, and the generalized method of moments \cite{hansen1982large}) used in previous works (see in particular works related to MRA~\cite{bandeira2014multireference, bendory2017bispectrum, perry2017sample}). Second, this section also explains the methodology behind the derivation of our algorithms, and some of the fundamental limitations related to the sample complexity of these algorithms. Third, the presentation in this section establishes that the presented algorithms can be interpreted as methods for recovering a signal from its second and forth moments (the trispectrum in particular). The trispectrum is currently used, for example, in cosmology \cite{PhysRevD.64.083005, PhysRevD.74.123519} and signal processing \cite{collis1998higher}, but currently there is no known algorithm for its robust inversion.

In what follows, we introduce the concept of shift-invariant statistics, and show its connection to the approach of Section~\ref{sec:method} by demonstrating that both of our algorithms for recovering the phases of the unknown signal makes use of the first two non-vanishing invariant statistic.

Consider the moments of the $L$ dimensional random vector $\hat{y}$ up to order four
\begin{align}
M_1[k] &= \mathbb{E}\left[ \hat{y}[k]\right], \\
M_2[k_1,k_2] &= \mathbb{E}\left[ \hat{y}[k_1]\hat{y}^*[k_2]\right], \\
M_3[k_1,k_2,k_3] &= \mathbb{E}\left[ \hat{y}[k_1]\hat{y}^*[k_2]\hat{y}[k_3]\right],\\
M_4[k_1,k_2,k_3,k_4] &= \mathbb{E}\left[ {\hat{y}}[k_1] {\hat{y}}^*[k_2] {\hat{y}}[k_3] {\hat{y}}^*[k_4] \right]. \label{eq:M_4 def}
\end{align} 
It is easy to verify that for $ y $ satisfying the model~\eqref{eq:MRFA_model_def}, the following holds:
\begin{enumerate}
	\item The first moment $ M_1[k] $ is equal to zero for any $ k $. In particular, the mean of the samples of $ y $ satisfies
	\begin{equation}\label{eq:M_1 def}
	\mu_y = M_1[0] = 0.
	\end{equation}
	\item The second moment $ M_2[k_1,k_2] $ is equal to zero for $ k_1 \neq k_2 $. For $ k_1 = k_2 $, the second moment is just the power spectrum of $ y $, that is
	\begin{equation}\label{eq:M_2 def}
	p_y[k] = M_2[k,k].
	\end{equation}
	\item The third moment $ M_3[k_1,k_2, k_3] $ is equal to zero for any $ k_1,k_2 $ and $ k_3 $. In particular, the \textit{bispectrum} of $ y $ vanishes,
	\begin{equation}\label{eq:M_3 def}
	 B_y[k_1,k_2] = M_3[k_1,k_2,k_2-k_1] = 0.
	\end{equation} 
The bispectrum $B_y$ is a third moment of $\hat{y}$ (the expected value of the product of $\hat{y}$ at three different indices), and it vanishes since $\hat{y}[k]$ admits a symmetric distribution around~$0$, owing to the random factor $a\sim\mathcal{CN}(0,\lambda)$. Specifically, the factor $ a $ in the bispectrum is taken to the third power, which results in an expected value equal to zero.
	\item The forth moment $ M_4[k_1,k_2,k_3,k_4] $ is equal to zero for any $ k_1-k_2+k_3-k_4 \neq 0 $. The value of $ M_{4} $ for the indices satisfying $ k_1-k_2+k_3-k_4 = 0 $ is known as the \textit{trispectrum}~\cite{collis1998higher},
	\begin{equation}\label{eq:defTrispectrum}
	T_y[k_1,k_2,k_3] = M_4[k_1,k_2,k_3,k_1-k_2+k_3].
	\end{equation}
\end{enumerate}
It is easy to verify that the entries of $\mu_y$, $p_y$, $B_y$ and $ T_y $ are invariant to cyclic shifts of $y$ (or equivalently, to modulations of $\hat{y}$), regardless of the distribution of the shifts, and hence serve as shift-invariant statistics of $y$.
When using the model~\eqref{eq:MRFA_model_def_Fourier} to compute~\eqref{eq:M_2 def}, it follows that
\begin{equation}\label{eq:power}
p_y[k] = \lambda \left\vert \hat{\theta}[k]\right\vert^2 + \sigma^2.
\end{equation}
Since $ \mu_y $ and $ B_y $ vanish, among the first three invariant statistics $\mu_y$, $p_y$ and $B_y$, only the power spectrum $p_y$ provides us with useful information, which is used to estimate $\lambda$ and the signal magnitudes $| \hat{\theta}[k] |$ as described in Section~\ref{subsec:magnitude est}. 

While the magnitudes of $\hat{\theta}$ can be estimated from the power spectrum~\eqref{eq:power}, the latter is insufficient to determine $\hat{\theta}$ entirely, as we also need to estimate the phases of $\hat{\theta}$. Therefore, we need to use the fourth-order moment, and particularly, the trispectrum of $ y $.
Thus, we next point out the relation between the trispectrum $T_y$ and the methods described in Section~\ref{sec:method}. Consider the random vector $z\in\mathbb{C}^{L^2}$ and its covariance matrix  $C_{z}\in\mathbb{C}^{L^2\times L^2}$, given by
\begin{equation}
z = 
\begin{pmatrix}
| \\
z^{(0)} \\
| \\
\\
\vdots \\
\\
| \\
z^{(L-1)} \\
| \\
\end{pmatrix}, 
\qquad \qquad C_{z} = \mathbb{E}\left[ z z^*\right],
\end{equation}
where $z^{(m)}\in\mathbb{C}^L, \;m=0,\ldots,L-1,$ are from~\eqref{eq:z^m def}. Note that $z$ can also be viewed as the elements of the matrix $y y^*$ reorganized into a vector (i.e. all products of the elements in $y$). Then, the covariance matrix $C_{z}$ can be expressed as
\begin{equation}
C_{z}[m_1 L +k_1,m_2 L + k_2]=\mathbb{E}\left[ z^{(m_1)}[k_1] \left( z^{(m_2)}[k_2] \right)^*\right] = \mathbb{E}\left[ y[k_1] y^*[k_1+m_1] y[k_2+m_2] y^*[k_2]  \right],
\end{equation}
where $(m_1,m_2)$ enumerates over blocks of size $L\times L$ in $C_z$. It follows that $C_z$ is related to the fourth moment $M_4$ of~\eqref{eq:M_4 def} via
\begin{equation}
C_{z}[m_1 L +k_1,m_2 L + k_2] = M_4[k_1,k_1+m_1,k_2+m_2,k_2].
\end{equation}
Therefore, the matrix $C_z$ encodes the same information as $M_4$. Note that 
\begin{equation}
C_{z}[m_1 L +k_1,m_2 L + k_2] = 0 \qquad \text{whenever} \qquad m_1\neq m_2,
\end{equation}
and thus, $C_{z}$ is a block-diagonal matrix with its non-zero blocks given by
\begin{equation}
C_z^{(m)}[k_1,k_2] = \mathbb{E}\left[ z^{(m)}[k_1] \left( z^{(m)}[k_2] \right)^*\right] =  T_y[k_1,k_1+m,k_2+m],
\end{equation}
for $m=0,\ldots,L-1$ (see~\eqref{eq:z cov def}), thereby establishing that the sequence of matrices $\left\{ C_z^{(m)} \right\}_{m=0}^{L-1}$, which are estimated in Step \ref{algstep:Cz_algAM} of Algorithm \ref{alg:MRFA} ($ C_z^{(1)} $ is also estimated in Step \ref{algstep:Cz_algFM} of Algorithm \ref{alg:MRFA_formal}) is equivalent to the trispectrum $T_y$. This relation between the trispectrum and $ C_z^{(m)} $ shows that Algorithms \ref{alg:MRFA_formal} and \ref{alg:MRFA}, starting from Steps \ref{algstep:tilde_u_algFM} and  \ref{algstep:form_Cx_algAM} respectively, are, in fact, estimating the signal from its trispectrum, or, in other words, Algorithms \ref{alg:MRFA_formal} and \ref{alg:MRFA} preform trispectrum inversion. 

In~\cite{abbe2018estimation} it is shown that the sample complexity of any algorithm that recovers the underlying signal from its noisy measurements on a group action channel (in our case the group action is shifting the signal) is bounded from below by $ \OO(\mbox{SNR}^{-d}) $, where $ d $ is the smallest integer such that all the moments up to order $ d $ define the signal uniquely (up to the ambiguities of the problem). For the case of the MRA problem this moment is the bispectrum ($d= 3 $). For the rank-one MRFA problem, the trispectrum is the first shift-invariant moment carrying sufficient information for recovering $\hat{\theta}$ (and thus $\theta$), and therefore in this case $d=4$.

\section{Numerical Examples}
\label{sec:experimental}

In this section, we demonstrate the performance of Algorithms \ref{alg:MRFA_formal} and \ref{alg:MRFA} (see Section~\ref{sec:method}) by numerical simulations, and show that  their performance agrees with the theoretical results. The error in all the experiments is measured as
\begin{equation}\label{eq:err_theta_def}
\mbox{Err}(\widetilde{\theta}) = \min_{s\in\ZZ_L,\; |\alpha|=1}\|\theta - \alpha \mathcal{R}_s(\widetilde{\theta})\|_2.
\end{equation}
We also compare our algorithms with the Expectation-Maximization (EM) algorithm~\cite{dempster1977maximum}, which is a popular approach for Maximum-Likelihood Estimation (MLE) in the presence of nuisance parameters. We start this section with a short explanation of the EM algorithm adapted to our model~\eqref{eq:MRFA_model_def}.

\subsection{EM algorithm for the MRFA problem}\label{subsec:EM}
The EM algorithm is a classical heuristic approach that tries to optimize model parameters by maximizing the likelihood of the observations given the model parameters. This approach is widely used in many applications \cite{scheres2012relion,van2000fitting}.
The EM algorithm iterates over two steps: The Expectation step (E-step) and the Maximization step (M-step). An elaborated description of the EM algorithm for the MRA problem can be found in~\cite{bendory2017bispectrum}. In the case of the rank-one MRFA~\eqref{eq:MRFA_model_def}, the EM algorithm is initialized with parameter estimates $ \theta_0 $ and $\lambda_0$, and then iterates the E-step and M-step given below until convergence:
\begin{description}
 	\item[E-Step:] At iteration $k$, given current parameter estimates $ \theta_k $ and $\lambda_k$, estimate the likelihood of the observation $ y_j $  assuming shift $ s $. Under the model in \eqref{eq:MRFA_model_def}, the likelihood is
 	\begin{equation}
 	w_k^{j,s} \propto \exp\left(
- \frac{R_s\{y_j^*\} \left( \lambda_k \theta_k  \theta_k^* + \sigma^2 I_L \right)^{-1} R_s\{y_j\}}{2}
 	\right).
 	\end{equation}
 	\item[M-Step:]
 	Find $ \theta_{k+1} $ and $\lambda_{k+1}$ that minimize the expression
 	\begin{equation}\label{eq:EM_Mstep_min}
 	\sum_{j=1}^{N}\sum_{s=0}^{L} w_k^{j,s}
 	{R_s\{y_j^*\} \left( \lambda_{k+1} \theta_{k+1}  \theta_{k+1}^* + \sigma^2 I_L \right)^{-1} R_s\{y_j\}}.
 	\end{equation}
\end{description}
\subsection{Experimental results}
\label{subsec:numerical_a_normal}
We first demonstrate the performance of Algorithms \ref{alg:MRFA_formal} and \ref{alg:MRFA} and the EM algorithm for random complex signals of length 16 ($ L=16 $), for a number of samples ($ N $) between $ 10^1 $ and $ 10^5 $, and for SNR (as defined in \eqref{eq:SNR_def}) values between $ 10^0 $ and $ 10^{-2} $. The signals are generated as follows. First, we generate a random complex Gaussian signal of length $ L $, and normalize its power spectrum to be a vector of ones. The power spectrum normalization ensures that $ \gamma_1 $ and $ \delta_1 $ (as defined in Theorem \ref{thm:alg 1 asym err}) are the same in all experiments. Next, the observations are generated from the clean signal by multiplying it by a normally distributed complex factor with variance~$ 1 $, then shifting it by a uniformly distributed shift and adding to the resulting vector Gaussian noise with variance $ \sigma^2 $ (derived from the SNR). 

The results of the different algorithms are shown in Figure \ref{fig:heat_MRFAs}, where the intensity of each pixel describes the accuracy of the considered algorithm for the corresponding $ N $ and  SNR (blue/dark is high accuracy and yellow/light is low accuracy). In Figures \ref{fig:heat_MRFA_formal} and \ref{fig:heat_MRFA_alt} we demonstrate the performance of Algorithms \ref{alg:MRFA_formal} and \ref{alg:MRFA} respectively. In Figures \ref{fig:heat_EMOracle} and \ref{fig:heat_EM} we demonstrate the performance of the EM algorithm, where in Figure \ref{fig:heat_EMOracle} the EM algorithm is initialized with the correct signal (``Oracle initialization") and in Figure \ref{fig:heat_EM} the EM algorithm is initialized with a random  vector. Although Figures \ref{fig:heat_MRFA_formal} and \ref{fig:heat_MRFA_alt} look similar, Algorithm \ref{alg:MRFA} ``breaks" at lower SNRs for any given sample size (it is mostly visible when comparing Figures~\ref{fig:heat_MRFA_formal} and \ref{fig:heat_MRFA_alt} in the region where $ \log_{10}N = 5 $ and $ \log_{10}\mbox{SNR} = -1.8 $ ).

\begin{figure}
	\centering
	\begin{subfigure}{0.5\linewidth}%[Algorithm \ref{alg:MRFA_formal}]
		\centering
		\includegraphics[scale=.5]{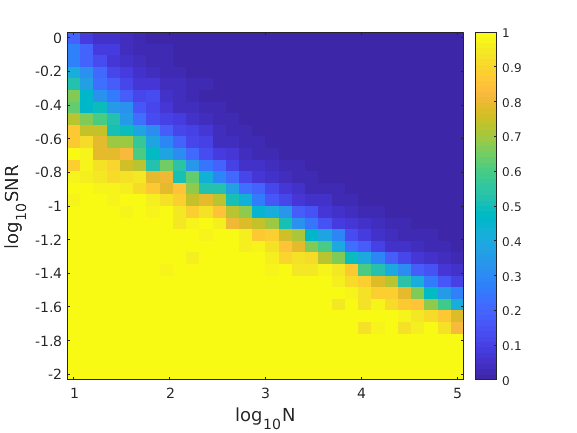}
		\caption{Algorithm \ref{alg:MRFA_formal}}
		\label{fig:heat_MRFA_formal}
	\end{subfigure}%
	\begin{subfigure}{0.5\linewidth}%[Algorithm \ref{alg:MRFA}]
		\centering
		\includegraphics[scale=.5]{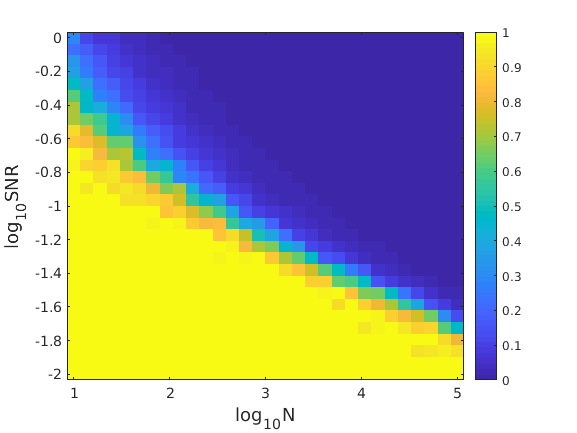}		
		\caption{Algorithm \ref{alg:MRFA}}
		\label{fig:heat_MRFA_alt}
	\end{subfigure}%
	\\[1ex]
	\begin{subfigure}{0.5\linewidth}%[EM algorithm]
		\centering
		\includegraphics[scale=.5]{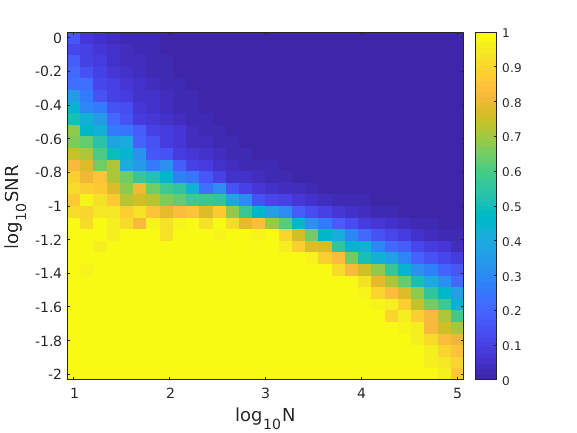}
		\caption{EM algorithm - Oracle initialization}
		\label{fig:heat_EMOracle}
	\end{subfigure}%
	\begin{subfigure}{0.5\linewidth}%[EM algorithm]
		\centering
		\includegraphics[scale=.5]{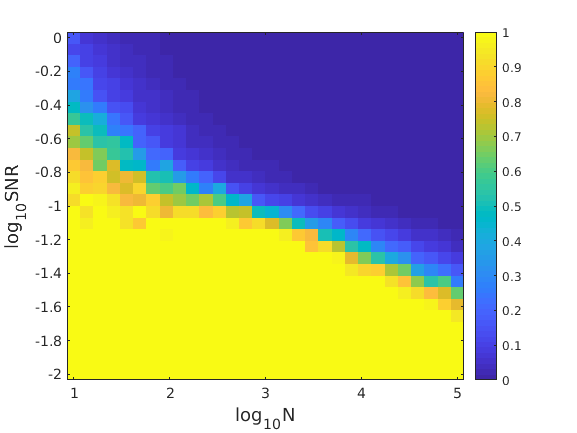}
		\caption{EM algorithm - random initialization}
		\label{fig:heat_EM}
	\end{subfigure}
	% comparisson_MRFA_res_2018_7_6_23_22_25.mat results used
	\caption{Heat map plot of the error rates of Algorithm \ref{alg:MRFA_formal} (in \ref{fig:heat_MRFA_formal}), Algorithm \ref{alg:MRFA} (in \ref{fig:heat_MRFA_alt}), the EM algorithm with oracle initialization (in \ref{fig:heat_EMOracle}) and the EM algorithm with random initialization  (in \ref{fig:heat_EM}) for a signal of length $ L=16 $ with a different number of observations (shown as the x-axis, from $ 10$ to $ 10^5 $ on a log scale), and for different SNR values (shown as the y-axis, from $ 1 $ to $0.01$ on a log scale). The color of each pixel represents the accuracy, measured as in \eqref{eq:err_theta_def}, of the corresponding algorithm, where blue (dark) represents high accuracy (error close to $ 0 $) and yellow (bright) represents low accuracy (error close to 1). Each pixel is an average of 25 runs of the algorithms.}
	\label{fig:heat_MRFAs}
\end{figure}

We have shown in \eqref{eq:N_asymp_L_SNR} that $ N $ should asymptotically be proportional to $ 1/\mbox{SNR}^4 $, or, on a  logarithmic scale, $ \log N $ should depend linearly on $ \log \mbox{SNR} $. It can be noticed that the phase transition in Figures \ref{fig:heat_MRFA_formal} and \ref{fig:heat_MRFA_alt} is indeed linear. In Figure \ref{fig:MRFA_alt_heat_asymp}, we show the same heat-map as in Figure \ref{fig:heat_MRFA_alt} together with a black line that is the solution to the equation $ N = (1/4)L~\mbox{SNR}^4 $. Notice how the transition region is aligned with the solid black line.

\begin{figure}
	\centering
	
	\begin{subfigure}{0.5\linewidth}%[Error]
		\centering
		\includegraphics[scale=.5]{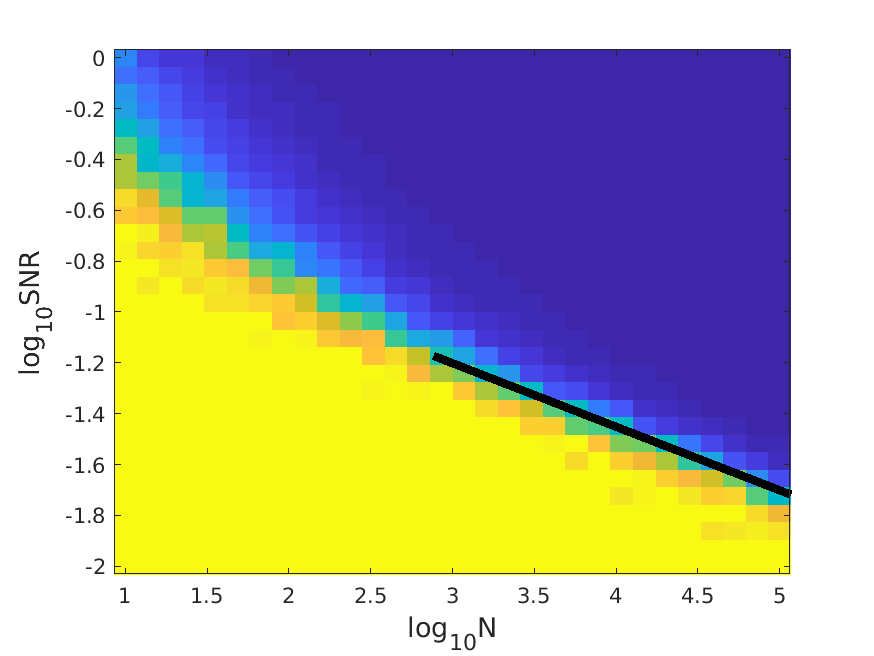} 
		% comparisson_MRFA_res_2018_7_6_23_22_25.mat results used
		\caption{$ L=16 $}
		\label{fig:MRFA_alt_heat_asymp_L16}
	\end{subfigure}%
	\begin{subfigure}{0.5\linewidth}%[Time]
		\centering
		\includegraphics[scale=.5]{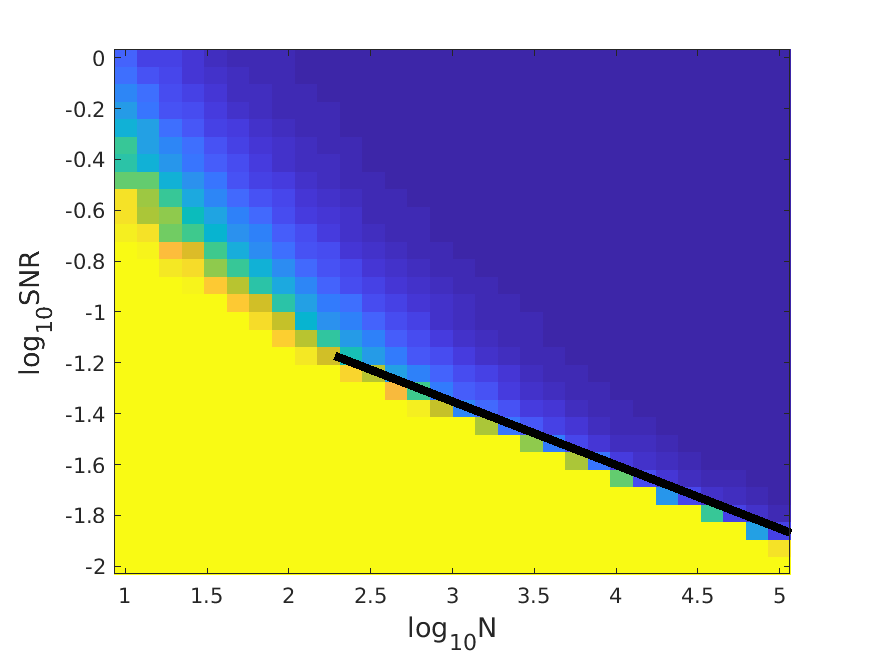}
		% comparisson_MRFA_res_2018_7_17_19_35_13.mat results used
		\caption{$ L=64 $}
		\label{fig:MRFA_alt_heat_asymp_L64}
	\end{subfigure}

	\caption{Heat map plot of the performance of Algorithm~\ref{alg:MRFA} for a random signal of length $ L=16 $ (in~\ref{fig:MRFA_alt_heat_asymp_L16}) and of length $ L=64 $ (in~\ref{fig:MRFA_alt_heat_asymp_L64}) with a different number of realizations (shown as the x-axis, from $ 10$ to $ 10^5 $ on a log scale), and for different SNR values (shown as the y-axis, from $ 1 $ to $0.01$ on a log scale). Each pixel is an average of 25 runs of the algorithms. The solid line (in both plots) corresponds to $ N = (1/4)L\mbox{SNR}^4 $.}
	\label{fig:MRFA_alt_heat_asymp}
\end{figure}

% As a function of SNR
We show the performance of Algorithm \ref{alg:MRFA_formal}, Algorithm \ref{alg:MRFA} and the EM algorithm with oracle/ random initialization for $ N = 10^5 $ and $ L = 16 $ for different SNR values in Figure \ref{fig:MRFA_vs_EM_constN_const_L}. Although the accuracy of the EM algorithm is higher for high SNRs, for low SNRs (starting from around $ 0.17 $) the EM algorithm gives less accurate results. Additionally, as opposed to the constant running time of Algorithms \ref{alg:MRFA_formal} and \ref{alg:MRFA}, the EM algorithm takes much longer for low SNR values.

\begin{figure}
	\centering
	
	\begin{subfigure}{0.5\linewidth}%[Error]
		\centering
		\includegraphics[scale=.5]{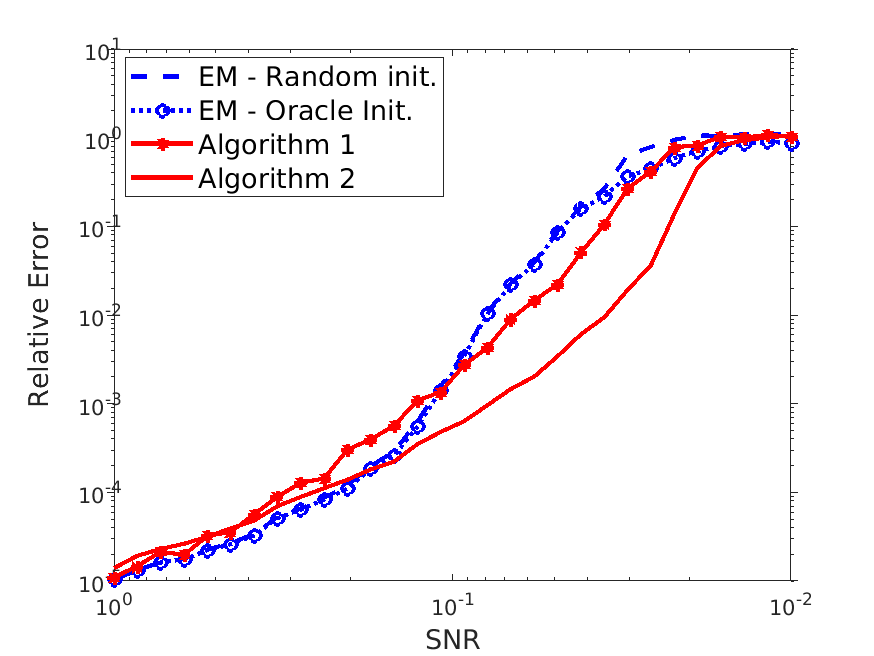}
		\caption{Relative error}
		\label{fig:MRFA_vs_EM_constN_const_L_err}
	\end{subfigure}%
	\begin{subfigure}{0.5\linewidth}%[Time]
		\centering
		\includegraphics[scale=.5]{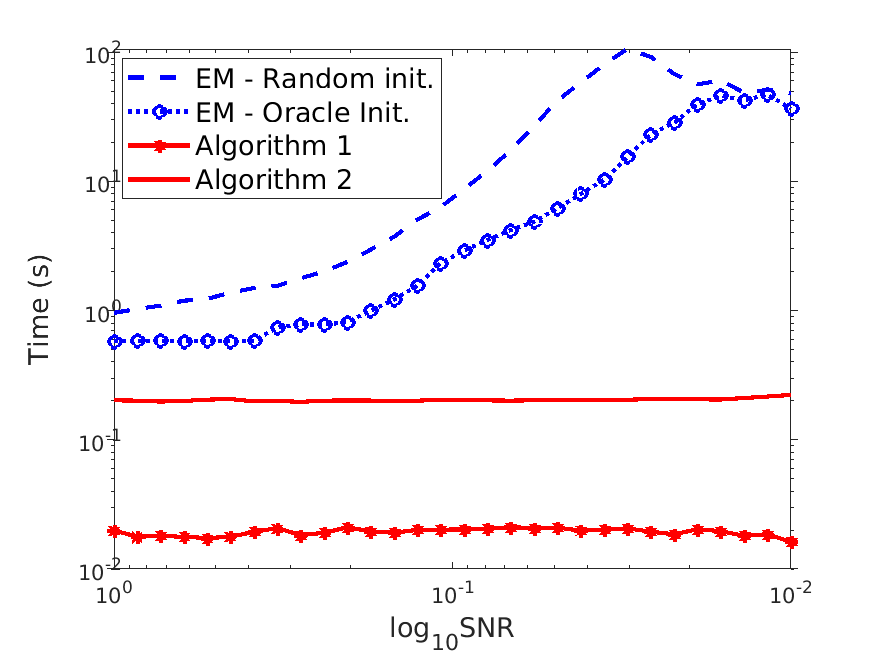}
		\caption{Running time}
		\label{fig:MRFA_vs_EM_constN_const_L_time}
	\end{subfigure}
		
	% comparisson_MRFA_res_2018_7_6_23_22_25.mat results used
	\caption{Performance of Algorithm~\ref{alg:MRFA_formal}, Algorithm~\ref{alg:MRFA} and the EM algorithm with oracle/ random initialization for $ N = 10^5 $ and $ L = 16 $ for different SNR values. }
	\label{fig:MRFA_vs_EM_constN_const_L}
\end{figure}

% As a function of N
We show next, in Figure \ref{fig:MRFA_vs_EM_constSNR_const_L}, the performance of Algorithm \ref{alg:MRFA_formal}, Algorithm \ref{alg:MRFA} and the EM algorithm with oracle/ random initialization for $ \text{SNR} = 0.108 $ and $ L = 16 $ for different values of $ N $. It can be noticed from Figure \ref{fig:MRFA_vs_EM_constSNR_const_L_time} that Algorithm \ref{alg:MRFA} is slower than Algorithm \ref{alg:MRFA_formal} by approximately a constant time for all $ N $. The computationally expensive step of Algorithm \ref{alg:MRFA} is the ``alternating minimization" step, whose complexity does not depend on $ N $. 

\begin{figure}
	\centering
	\begin{subfigure}{0.45\linewidth}%
		\centering
		\includegraphics[scale=0.5]{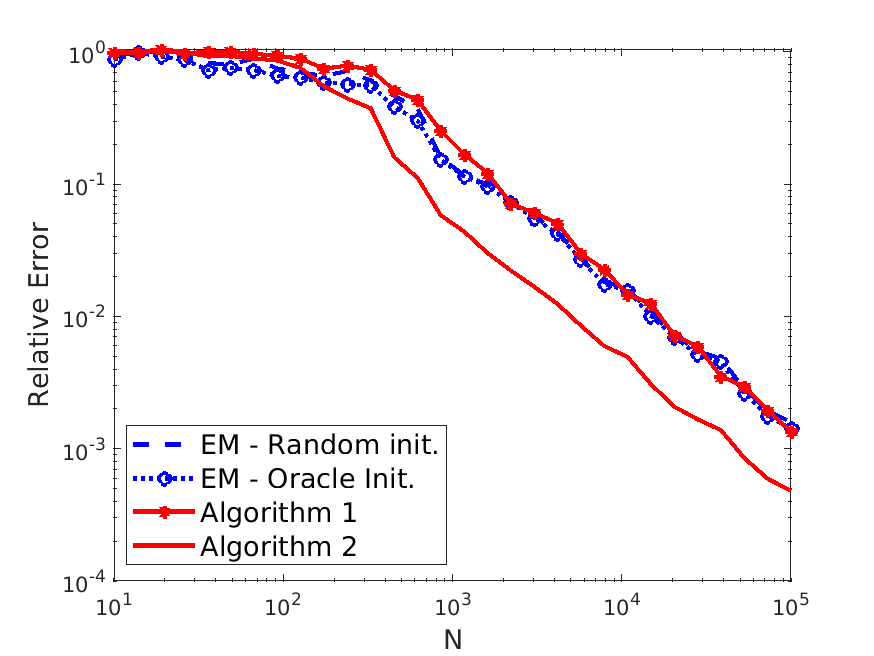}
		\caption{Relative error}
		\label{fig:MRFA_vs_EM_constSNR_const_L_err}
	\end{subfigure}
	\begin{subfigure}{0.45\linewidth}
		\centering
		\includegraphics[scale=0.5]{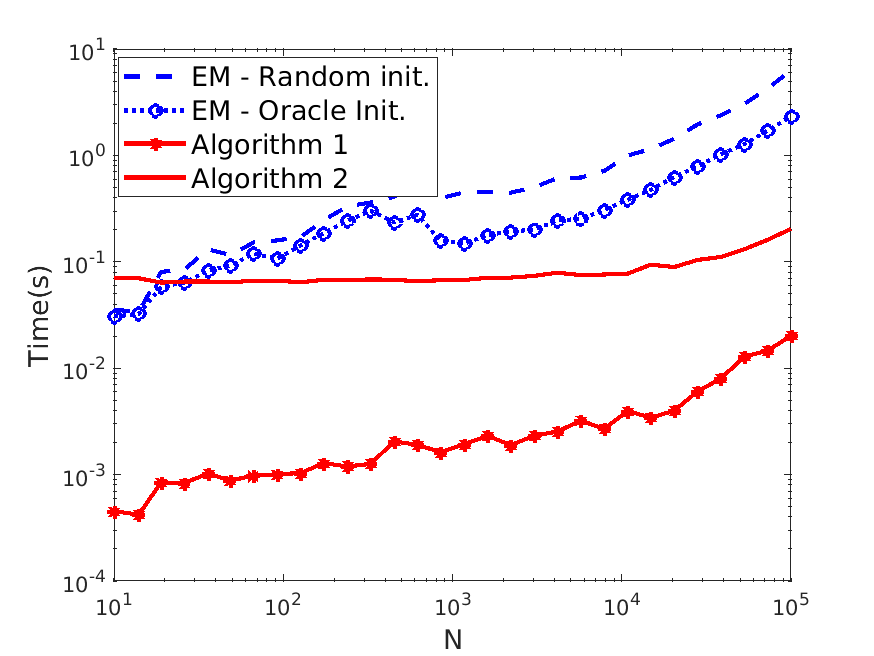}
		\caption{Running time}
		\label{fig:MRFA_vs_EM_constSNR_const_L_time}
	\end{subfigure}
	% comparisson_MRFA_res_2018_7_6_23_22_25.mat results used
	\caption{Performance of Algorithm \ref{alg:MRFA_formal}, Algorithm \ref{alg:MRFA} and the EM algorithm with oracle / random initialization for $ SNR = 0.108 $ and $ L = 16 $ for a different number of observations $ N $.}
	\label{fig:MRFA_vs_EM_constSNR_const_L}
\end{figure}

% As a function of L
Finally, we show in Figure \ref{fig:heatmapL} a heat-map of the accuracy of  Algorithm \ref{alg:MRFA} for $ N = 4000 $, with $ \mbox{SNR} $ values between $ 0.5$ and $ 0.01 $ (on the $ y $-axis, on a log scale) and for values of $ L $ between $ 16 $ and $ 256 $ (on the $ x $-axis, on a log scale). It can be noticed that, as expected, the dependency between $ \log L $ and $ \log$ SNR is linear.

\begin{figure}
	\centering
	\includegraphics[scale=.55]{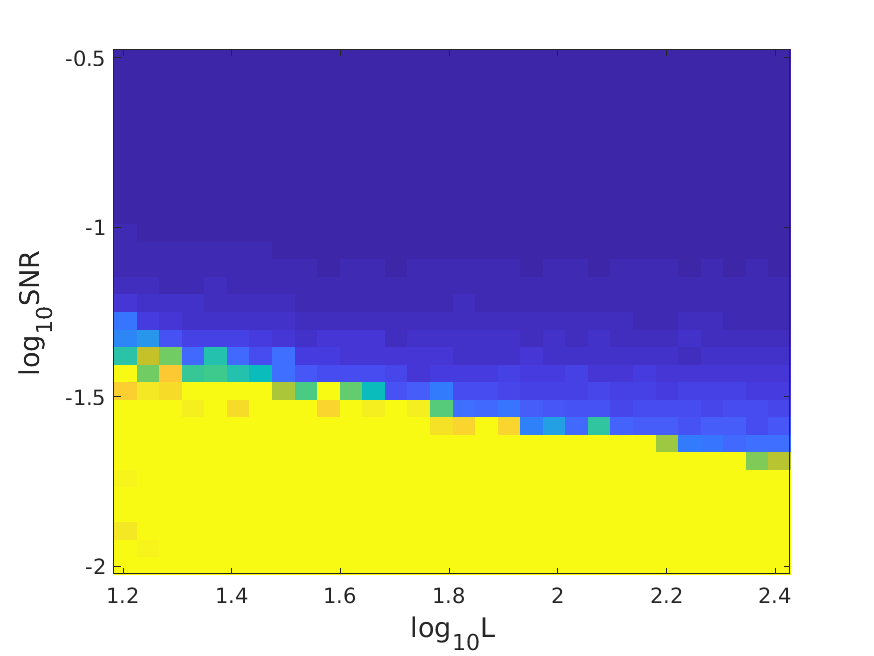}
	% comparisson_MRFA_res_2018_7_15_22_33_5.mat on math014 (or Leia) results used
	\caption{Heat map plot of the accuracy of Algorithm \ref{alg:MRFA} for  $ N = 4000 $, for different signal lengths $ L $ (shown as the x-axis, from $ 16$ to $ 256 $, on a log scale), an for different SNR values (shown as the y-axis, from $ 0.5 $ to $0.01$ on a log scale). The color of each pixel represents the accuracy of the corresponding algorithm, where blue (dark) represents high accuracy (error close to $ 0 $) and yellow (bright) represents large error (close to 1). Each pixel is an average of 25 runs of the algorithms.}
	\label{fig:heatmapL}
\end{figure}

\section{Summary and discussion}	\label{sec:summary}
We presented two statistically consistent algorithms for solving the rank-one MRFA problem. One algorithm (Algorithm \ref{alg:MRFA_formal}) has proven non-asymptotic performance bounds, and the other (Algorithm \ref{alg:MRFA}) has better performance in practice. We compared the performance of the two algorithms  to  the EM algorithm, and showed the superiority of Algorithms \ref{alg:MRFA_formal} and \ref{alg:MRFA} in ``difficult" regimes. 
Intuitively, Algorithm~\ref{alg:MRFA} uses more information than Algorithm~\ref{alg:MRFA_formal}, by taking advantage of the entire trispectrum, and thus gives more accurate results. 
We also note that even though our model~\eqref{eq:MRFA_model_def} considers the case of uniform distribution for the shifts (which is more challenging in terms of estimation accuracy~\cite{abbe2017multireference}), the algorithms presented in this work can handle any distribution.

Although Algorithm \ref{alg:MRFA} uses more information than Algorithm \ref{alg:MRFA_formal}, we did not show  that Algorithm \ref{alg:MRFA} is optimal. Thus, there might be a better way for solving the MRFA problem and, equivalently, for inverting the trispectrum. 

As of future research, an interesting direction is to extend the rank-one MRFA model to a general rank-$ r $ MRFA model.

\begin{appendices}
\section{Covariance matrix of $z^{(m)}$}\label{sec:covariance of z_m}
In this section, we show that $C_{z}^{(m)}$ of~\eqref{eq:z cov def} is given by 
\begin{align}
\begin{aligned}
C_z^{(m)} = 2\lambda^2 u^{(m)} \left( u^{(m)}\right)^* + \Sigma^{(m)}_\epsilon, \label{eq:C_z bias formula}
\end{aligned}
\end{align}
where $\Sigma^{(m)}_\epsilon$ is a diagonal bias term given by 
\begin{equation}\label{eq:Sigma_epsilon def}
\Sigma^{(m)}_\epsilon[k_1,k_2] = 
\begin{dcases}
0, & k_1 \neq k_2, \\
 \sigma^2 \left( p_x[k_1] + p_x[k_1 + m] \right) + \sigma^4, & k_1=k_2, 
\end{dcases}
\end{equation}
for any $m\neq 0$, where $p_x$ is the power spectrum of $x$ (see~\eqref{eq:p_x def}). Thus, we establish that the matrix $C_z^{(m)} - \Sigma^{(m)}_\epsilon$ is of rank one, with its leading eigenvector equal to $u^{(m)}$ up to a constant factor.

Recall that (see \eqref{eq:zm_expression})
\begin{align}
z^{(m)}[k] = |a|^2 \omega^{-s m} u^{(m)}[k] + \epsilon^{(m)}[k],
\label{eq:zm_expression 2}
\end{align}
where $\epsilon^{(m)}$ is given by (see \eqref{eq:epsilon error def})
\begin{equation}
\epsilon^{(m)}[k] = a \omega^{s k} \hat{\theta}[k] \hat{\eta}^*[k+m] + a^*\omega^{-s(k+m)}\hat{\theta}^*[k+m] \hat{\eta}[k] + \hat{\eta}[k]\hat{\eta}^*[k+m], \label{eq:epsilon def 2}
\end{equation}
and $ \hat{\eta} $ is defined in \eqref{eq:MRFA_model_def_Fourier}.
From~\eqref{eq:z cov def},~\eqref{eq:zm_expression 2},~\eqref{eq:epsilon def 2} we get that  
\begin{align}
\begin{aligned}
&C_z^{(m)}[k_1,k_2] = \mathbb{E}\left[ z^{(m)}[k_1] \left( z^{(m)}[k_2]\right)^* \right] = \mathbb{E}\left[ |a|^4 u^{(m)}[k_1] \left( u^{(m)}[k_2]\right)^* \right] \\ &+ \mathbb{E}\left[ |a|^2 \omega^{-s m} u^{(m)}[k_1] \left( \epsilon^{(m)} [k_2] \right)^* \right]  
+ \mathbb{E}\left[ \left( \epsilon^{(m)} [k_1] \right) |a|^2 \omega^{s m} \left( u^{(m)} [k_2]\right)^* \right] 
+ \mathbb{E}\left[ \epsilon^{(m)}[k_1] \left( \epsilon^{(m)}[k_2] \right)^*\right]. \label{eq:C_z equation}
\end{aligned}
\end{align}
Since $a\sim\mathcal{CN}(0,\lambda)$ is complex-valued, we have that $\mathbb{E}\left[ |a|^4\right] = 2\lambda^2$. Additionally, when substituting~\eqref{eq:epsilon def 2} in~\eqref{eq:C_z equation}, together with the fact that  $\mathbb{E}\left[ \hat{\eta}[k]\right] = 0$, it follows that for  $m\neq 0$
\begin{align}
&\mathbb{E}\left[ |a|^2 \omega^{-s m} u^{(m)}[k_1] \left( \epsilon^{(m)} [k_2] \right)^* \right] = 0, \\
&\mathbb{E}\left[ \left( \epsilon^{(m)} [k_1] \right)^* |a|^2 \omega^{s m} \left( u^{(m)} [k_2]\right)^*\ \right] = 0,
\end{align}
for every $k_1$ and $k_2$, since $\hat{\eta}[k]$ and $\hat{\eta}[k+m]$ are uncorrelated for every $k$ (and $m\neq 0$), and independent of the factor $a$.
Next, we have

\begin{equation}\label{eq:Sigma_eps long}
\begin{aligned}
\Sigma^{(m)}_\epsilon[k_1,k_2] \triangleq& \mathbb{E}\left[ \epsilon^{(m)}[k_1] \left( \epsilon^{(m)}[k_2] \right)^*\right] \\=& \mathbb{E} \bigg{[} \left(a \omega^{s k_1} \hat{\theta}[k_1] \hat{\eta}^*[k_1+m] + a^*\omega^{-s(k_1+m)}\hat{\theta}^*[k_1+m] \hat{\eta}[k_1]  \right) \\ &~~~\times \left(a \omega^{s k_2} \hat{\theta}[k_2] \hat{\eta}^*[k_2+m] + a^*\omega^{-s(k_2+m)}\hat{\theta}^*[k_2+m] \hat{\eta}[k_2]  \right)^* \bigg{]} \\
&+ \mathbb{E}\bigg{[} \left(a \omega^{s k_1} \hat{\theta}[k_1] \hat{\eta}^*[k_1+m] + a^*\omega^{-s(k_1+m)}\hat{\theta}^*[k_1+m] \hat{\eta}[k_1]  \right) \hat{\eta}^*[k_2]\hat{\eta}[k_2+m] \\ &~+ \hat{\eta}[k_1]\hat{\eta}^*[k_1+m] \left(a \omega^{s k_1} \hat{\theta}[k_1] \hat{\eta}^*[k_1+m] + a^*\omega^{-s(k_2+m)}\hat{\theta}^*[k_2+m] \hat{\eta}[k_2]  \right)^* \bigg{]} \\
&+ \mathbb{E}\bigg{[} \hat{\eta}[k_1]\hat{\eta}^*[k_1+m] \hat{\eta}^*[k_2]\hat{\eta}[k_2+m] \bigg{]}.
 \end{aligned}
\end{equation}

Further manipulation of \eqref{eq:Sigma_eps long} using the properties of the model~\eqref{eq:MRFA_model_def_Fourier} results in 
\begin{align}
\begin{aligned}
&\Sigma^{(m)}_\epsilon[k_1,k_2] = 
\begin{dcases}
0, & k_1 \neq k_2, \\
\lambda \sigma^2 \left( | \hat{\theta}[k_1] |^2 + | \hat{\theta}[k_1+m] |^2 \right) + \sigma^4, & k_1=k_2.
\end{dcases}
 \end{aligned}
\end{align}
Lastly, formula~\eqref{eq:p_x def} for the power spectrum $p_x$ gives~\eqref{eq:Sigma_epsilon def}.

\section{Error bound for $ \wtilde{u}^{(1)} $}\label{sec:appendix_thmStep1}

\subsection{Preliminary results}
We first recall the definition of a sub-exponential random variable (see~\cite{vershynin2010introduction, vershynin_2012,vershynin2018high}).
\begin{definition}\label{def:sub_exponential}
	A random variable $X$ is called sub-exponential if there exists a constant $\alpha>0$ such that
	\begin{equation}
	\operatorname{Pr}{\left\{ \left\vert X\right\vert >t \right\}} \leq 2e^{-\alpha t}
	\end{equation}
	for all $t\geq 0$.
\end{definition}

\begin{lemma} \label{lem:maxMeanNormal}
	Let $ X_{ik} $, $ i=1,\ldots, N $ and $ k = 0,\ldots, L-1 $, be i.i.d. sub-exponential random variables with zero mean. Then, there exist constants $C^{'},c^{'}>0$ such that
	\begin{equation}
	\max_{k\in \{0,\ldots,L-1\}}\left\vert \frac{1}{N} \sum_{i=1}^N X_{i,k} \right\vert \leq \sqrt{\frac{L}{N}} 
	\end{equation}
	with probability at least
	\begin{equation}
	1 - C^{'} L e^{-c^{'} \sqrt{L}}.
	\end{equation}
\end{lemma}

\begin{proof}
	Fixing $k\in\left\{0,\ldots,L-1\right\}$, $X_{1,k},\ldots,X_{N,k}$ are i.i.d. sub-exponential random variables, and by Proposition 5.16 in~\cite{vershynin2010introduction}, there exist constants $C^{'},c^{'}>0$ such that
	\begin{equation*}
	\operatorname{Pr}\left\{ \left\vert \frac{1}{N} \sum_{i=1}^N X_{i,k} \right\vert > \frac{t}{\sqrt{N}} \right\} \leq C^{'} e^{-c^{'} t}.
	\end{equation*}
	Therefore, by the union bound we have that
	\begin{equation*}
	\operatorname{Pr}\left\{ \max_{k\in \{0,\ldots,L-1\}}\left\vert \frac{1}{N} \sum_{i=1}^N X_{i,k} \right\vert > \frac{t}{\sqrt{N}} \right\} \leq C^{'} L e^{-c^{'} t}.
	\end{equation*}
	Finally, substituting $t=\sqrt{L}$ concludes the proof.
\end{proof}

The following lemma relates the concentration of two non-negative random variables to the concentration of their sum.
\begin{lemma} \label{lem:concentration of sum of two rv}
	Let $a\geq 0$ and $b\geq 0$ be two (possibly dependent) random variables. Then, it holds that
	\begin{equation}
	\Pr{\left\{ a + b > t\right\}} \leq \Pr{\left\{ a > t/2 \right\}} + \Pr{\left\{ b > t/2 \right\}}.
	\end{equation} 
\end{lemma}
\begin{proof}
	Note that 
	\begin{equation*}
	\begin{array}{lll}
	\Pr\{a+b>t\} =\Pr\{b>t-a\}&= \Pr\{a>t/2 ~\wedge~ b>t-a\} &+ \Pr\{a\leq t/2~\wedge~ b>t-a\}\\
	             &\leq \Pr\{a>t/2\}&+\Pr\{a\leq t/2~\wedge~ b>t-a\}.
	\end{array}
	\end{equation*}
	Since in the expression $ \Pr\{a\leq t/2~\wedge~ b>t-a\} $, $ a $ is smaller or equal to $ t/2 $, we have that $ b>t-a$ implies that $b>t/2 $, and thus, $\Pr\{a\leq t/2~\wedge~ b>t-a\}\leq \Pr\{b>t/2\} $.
\end{proof}

To prove Theorem 2.1 we will use the following definition.
\begin{definition}\label{def:zeta_general}
	Let $\wtilde{\zeta} \in \RR^{2L}$ be a vector of identically distributed sub-exponential random variables with zero mean. Suppose that for any $ i=0, \ldots L-1 $, $ \wtilde{\zeta}_i $ depends only on $ \wtilde{\zeta}_{(i+1) \mod L}, \wtilde{\zeta}_{L+i} $ and $ \wtilde{\zeta}_{L+(i+1 \mod L)} $. Then, we call $\wtilde{\zeta}$ ``piecewise-i.i.d."(identically distributed, piecewise independent).
\end{definition}

\subsection{Concentration results for sub-exponential random vectors}

%%%%%%%%%%%%%%%%%%%%%%%
%%%%%%%%%%%%%%%%%%%%%%%%%%%%%%%%%%%
%%%%%%%%%%%%%%%%%%%%%%%%%%%%%%%%%%%%

In this section, we show that a "piecewise i.i.d." vector $\wtilde{\zeta}$, defined in Definition~\ref{def:zeta_general}, admits some concentration properties related to sub-exponential random vectors. The main lemma of this section is Lemma~\ref{lem:tilde_eta sample covariance concentration}.

First, we state the following result from~\cite{vershynin_2012}.
\begin{prop}{[Proposition 5.16 in \cite{vershynin_2012}]} \label{prop:sub-exponential concentration}
	Let $X = \left( X_1,\ldots,X_m \right) \in \mathbb{R}^m$, where $X_i$ are independent sub-exponential random variables with zero mean. Then, there exists a constant $\beta>0$, such that for every $a=\left( a_1,\ldots,a_m\right) \in \mathbb{R}^m$ and every $t\geq 0$
	\begin{equation*}
	\Pr{\left\{ \left\vert \left\langle X,a\right\rangle \right\vert > t \right\}} \leq 2 e^{-\beta {t}/{\left\Vert a\right\Vert_{\infty}}}.
	\end{equation*}
\end{prop}
We mention that we dropped the sub-Gaussian tail from the original statement of the proposition in \cite{vershynin_2012} (and thus the bound above is weaker then the original bound) as it can be bounded by the sub-exponential tail, which is the one of interest for our purposes.

It follows from Proposition~\ref{prop:sub-exponential concentration} that the random variable $\left\langle X,a\right\rangle$ is uniformly sub-exponential (i.e., bounded by the same
decay rate) for all vectors $a$ with $\left\Vert a\right\Vert_{\infty}$ bounded (e.g., for $a\in\mathcal{S}^{2L-1}$). We now prove that a similar property holds for a piecewise i.i.d. vector $\wtilde{\zeta}$, even though its elements are not independent. 
To prove such a property, we use Lemma~\ref{lem:concentration of sum of two rv} to decompose the vector $\wtilde{\zeta}$ into $O(1)$  vectors, each with independent entries. 
 This is stated and proved in Lemma~\ref{lem:tilde_eta inner prod concentration}.
\begin{lemma} \label{lem:tilde_eta inner prod concentration}
	Let $\wtilde{\zeta} \in \RR^{2L}$ be ``piecewise-i.i.d." as in Definition~\ref{def:zeta_general}. Then, there exists a constant $\beta_2>0$, depending on the distribution of the entries of $ \wtilde{\zeta} $, such that for every $a=\left( a[0],\ldots,a[2L-1]\right) \in \mathcal{S}^{2L-1}$ (i.e. $\sum_{k=0}^{2L-1} \left\vert a[k] \right\vert^2 = 1$) and every $t\geq 0$
	\begin{equation*}
	\Pr{\left\{ \left\vert \left\langle \wtilde{\zeta},a\right\rangle \right\vert > t \right\}} \leq 2 e^{-\beta_2 {t}}.
	\end{equation*}  
\end{lemma}
\begin{proof}
	Let us assume for simplicity that $L$ is even, and define the four vectors
	\begin{equation} \label{eq:tilde_eta decomposition}
	\begin{aligned}
	\wtilde{\zeta}_{1}[k] &= \wtilde{\zeta}[2k], \qquad &&k=0,\ldots,L/2-1, \\
	\wtilde{\zeta}_{2}[k] &= \wtilde{\zeta}[2k+1], \qquad &&k=0,\ldots,L/2, \\
	\wtilde{\zeta}_{3}[k] &= \wtilde{\zeta}[2k], \qquad &&k=L/2,\ldots,L-1, \\
	\wtilde{\zeta}_{4}[k] &= \wtilde{\zeta}[2k+1], \qquad &&k=L/2,\ldots,L-1.
	\end{aligned}
	\end{equation}
	By the definition of $\wtilde{\zeta}$ (see Definition~\ref{def:zeta_general}), 
	it follows that the elements in each of $\wtilde{\zeta}_{1}$, $\wtilde{\zeta}_{2}$, $\wtilde{\zeta}_{3}$, $\wtilde{\zeta}_{4}$ are independent and sub-exponential. 
	For every vector $a\in\mathcal{S}^{2L-1}$, we reorganize it into four vectors $\left\{ a_i\right\}_{i=1}^4$ analogously to $\left\{\wtilde{\zeta}_i\right\}_{i=1}^4$, and get using the triangle inequality that
	\begin{equation*}
	\Pr{\left\{ \left\vert \left\langle \wtilde{\zeta}, a\right\rangle\right\vert > t \right\}} \leq \Pr{\left\{ \sum_{i=1}^4\left\vert \left\langle \wtilde{\zeta}_i, a_i\right\rangle\right\vert > t \right\}}.
	\end{equation*}
	Then, using Lemma~\ref{lem:concentration of sum of two rv} twice, we get that
	\begin{equation*}
	\Pr{\left\{ \left\vert \left\langle \wtilde{\zeta}, a\right\rangle\right\vert > t \right\}} \leq  \sum_{i=1}^4 \Pr{\left\{\left\vert \left\langle \wtilde{\zeta}_i, a_i\right\rangle\right\vert > t/4 \right\}}.
	\end{equation*}
	Lastly, we apply Proposition~\ref{prop:sub-exponential concentration} to each of $\left\{\wtilde{\zeta}_i\right\}_{i=1}^4$ and get that
	\begin{equation*}
	\Pr{\left\{ \left\vert \left\langle \wtilde{\zeta}, a\right\rangle\right\vert > t \right\}} \leq 2 \sum_{i=1}^4 e^{-\beta {t}/{\left( 4\left\Vert a_i\right\Vert_{\infty}\right)}} \leq 8 e^{-\beta {t}/{\left( 4\left\Vert a\right\Vert_{\infty}\right)}} \leq 8 e^{-\beta {t}/{4}},
	\end{equation*}
	where we used the fact that $\left\Vert a_i\right\Vert_{\infty} \leq \left\Vert a\right\Vert_{\infty} \leq 1$ for $a\in\mathcal{S}^{2L-1}$. Clearly, the constant $8$ is of no real significance, and can be replaced by $2$ together with an appropriate change in the exponent for any $ t>0 $. Since, in addition, for $ t\geq 0 $ we have that
	\begin{equation*}
	\Pr{\left\{ \left\vert \left\langle \wtilde{\zeta}, a\right\rangle\right\vert > t \right\}} \leq 1,
	\end{equation*}
	we can pick a suitable $\beta_2$ such that
	\begin{equation*}
	\Pr{\left\{ \left\vert \left\langle \wtilde{\zeta}, a\right\rangle\right\vert > t \right\}} \leq 2 e^{-\beta_2 {t}}
	\end{equation*}
	for all $t\geq 0$. This concludes the proof for even values of $L$.
	For odd values of $L$, one can repeat the proof with a slightly different decomposition of the vector $\wtilde{\zeta}$.
\end{proof}
Next, we derive a concentration result for the norm $\left\Vert \wtilde{\zeta} \right\Vert$. 
We first state the following large deviation result for vectors of independent sub-exponential random variables (Lemma 8.3 in~\cite{erdHos2012bulk}, slightly reformulated, and stated for $\alpha=1$, which corresponds to our definition of sub-exponential variables, and with the choice $B_{ii}=1$).
\begin{prop}{\cite{erdHos2012bulk}} \label{prop:sub-exponential vector norm concentration}
	Let $X = \left( X_1,\ldots,X_m \right) \in \mathbb{R}^m$, where $X_i$ are independent sub-exponential random variables with mean zero and variance $1/2$. Then, there exist constants $C,c > 0$ such that
	\begin{equation*}
	\Pr{\left\{ \frac{\left\Vert X \right\Vert^2}{m} > \frac{1}{2} + \frac{t}{2\sqrt{m}} \right\}}  \leq C e^{-c \sqrt{t}}.
	\end{equation*}
\end{prop}
Proposition~\ref{prop:sub-exponential vector norm concentration} essentially states that norms of vectors consisting of independent sub-exponential random variables cannot be too large, as they are well concentrated around their means. 
Now, we shall prove the same result (with different constants) for our  ``piecewise-i.i.d." vectors $\wtilde{\zeta}$, even though their elements are not independent.
\begin{lemma} \label{lem:tilde_eta norm concentration}
	Let $\wtilde{\zeta} \in \RR^{2L}$ be ``piecewise-i.i.d." as in Definition~\ref{def:zeta_general}, with $ \mbox{var}(\wtilde{\zeta}[1]) = 1/2 $ . Then, there exist constants $C_2,c_2>0$, such that 
	\begin{equation*}
	\Pr{\left\{ \frac{\left\Vert \wtilde{\zeta} \right\Vert^2}{2L}  > \frac{1}{2} + \frac{t}{2\sqrt{2L}} \right\}}  \leq C_2 e^{-c_2 \sqrt{t}}.
	\end{equation*}
\end{lemma} 
\begin{proof}
	The proof follows along the same lines as the proof of Lemma~\ref{lem:tilde_eta inner prod concentration}. Suppose that $L$ is even, and define the vectors $\left\{ \wtilde{\zeta}_i \right\}_{i=1}^4$ as in~\eqref{eq:tilde_eta decomposition}. Then,
	\begin{equation*}
	\Pr{\left\{ \frac{\left\Vert \wtilde{\zeta} \right\Vert^2}{{2L}} \geq \tau \right\}} = \Pr{\left\{ \frac{\sum_{i=1}^4 \left\Vert \wtilde{\zeta}_i \right\Vert^2}{{2L}} \geq \tau \right\}} = \Pr{\left\{ \sum_{i=1}^4 \frac{ \left\Vert \wtilde{\zeta}_i \right\Vert^2}{{L/2}} \geq 4\tau \right\}}.
	\end{equation*} 
	By, using Lemma~\ref{lem:concentration of sum of two rv} twice, we get
	\begin{equation*}
	\Pr{\left\{ \frac{\left\Vert \wtilde{\zeta} \right\Vert^2}{{2L}} \geq \tau \right\}} \leq 4 \sum_{i=1}^4 \Pr{\left\{ \frac{ \left\Vert \wtilde{\zeta}_i \right\Vert^2}{{L/2}} \geq {\tau} \right\}},
	\end{equation*}
	and by substituting $\tau= \frac{1}{2} + \frac{t}{2\sqrt{2L}}$ and applying Proposition~\ref{prop:sub-exponential vector norm concentration} (observing that the elements in $\left\{ \wtilde{\zeta}_i\right\}_{i=1}^4$ satisfy the required conditions) we have that
	\begin{equation*}
	\Pr{\left\{ \frac{\left\Vert \wtilde{\zeta} \right\Vert^2}{{2L}} \geq \frac{1}{2} + \frac{t}{2\sqrt{2L}} \right\}} \leq 16 C e^{-c \sqrt{t}/\sqrt{2}} = C_2 e^{-c_2 \sqrt{t}},
	\end{equation*}
	when using $C_2 = 16 C$, $c_2=c/\sqrt{2}$.
	This concludes the proof for even values of $L$. For odd values of $L$, the proof can be repeated with a slightly different partitioning of $\wtilde{\zeta}$.
\end{proof}
Next, using the concentration results of Lemma~\ref{lem:tilde_eta inner prod concentration} and Lemma~\ref{lem:tilde_eta norm concentration}, we are able to use the results of~\cite{adamczak2011sharp} to bound the norm of the sample covariance matrix of $\wtilde{\zeta}$. This is the subject of the next lemma.
\begin{lemma} \label{lem:tilde_eta sample covariance concentration}
	Let $\wtilde{\zeta}_1,\ldots,\wtilde{\zeta}_N$ be i.i.d. samples of the random vector $\wtilde{\zeta} \in \mathbb{R}^{2L}$ which is ``piecewise-i.i.d." as in Definition~\ref{def:zeta_general}, and assume that $N>2L$ is large enough. Assume further that $ \mathbb{E}\left[ \wtilde{\zeta} \wtilde{\zeta}^T\right] =\frac{1}{2} I_{2L} $. Then, there exist constants $C_4,c_3,c_4>0$ such that
	\begin{equation*}
	\Pr\left\{\left\Vert \frac{1}{N} \sum_{i=1}^N \wtilde{\zeta}_i \wtilde{\zeta}_i^T - \frac{1}{2}\cdot I_{2L} \right\Vert > C_4 \sqrt{\frac{2L}{N}}\right\} \leq 2e^{-c_3 \sqrt{2L}} + C_2 N e^{-c_4 N^{1/4}},
	\end{equation*} 
	where the constant $C_2$ is from Lemma~\ref{lem:tilde_eta norm concentration}.
\end{lemma}
\begin{proof}
	At the heart of this proof is Theorem 1 from~\cite{ADAMCZAK2011195} (see also Corollary 1 from~\cite{ADAMCZAK2011195}) which requires two conditions on the random vectors $\wtilde{\zeta}_i$ (equations (2) and (3) in~\cite{ADAMCZAK2011195}). The first condition is that for some $ \psi>0 $,
	\begin{equation} \label{eq:sec cond: uniform sub-exp}
	\max_{i\leq N} \sup_{a\in\mathcal{S}^{2L-1}} \left\Vert \left\langle \wtilde{\zeta}_i, a\right\rangle \right\Vert_{\Psi_1} \leq \psi,
	\end{equation}
	where the $\Psi_1$-norm of a random variable $X\in\mathbb{R}$ is defined as
	\begin{equation*}
	\left\Vert X \right\Vert_{\Psi_1} = \inf\left\{C>0: \; \mathbb{E} e^{\left\vert X \right\vert/C} \leq 2 \right\}.
	\end{equation*}
	The $\Psi_1$-norm is a characterization of a sub-exponential random variable 
	(see~\cite{vershynin2010introduction} for the $\Psi_1$-norm of sub-exponential random variables).
	As we have shown in Lemma~\ref{lem:tilde_eta inner prod concentration}, each random variable $\left\langle \wtilde{\zeta}_i, a\right\rangle$ is sub-exponential uniformly over all $a\in\mathcal{S}^{2L-1}$ (that is, bounded by the same decay rate for all $a$). Therefore, from the definition of a sub-exponential random variable, and as $\wtilde{\zeta}_i$ are identically distributed sub-exponential random variables, we have by Lemma 2.3 of \cite{adamczak2010quantitative} that 
	\begin{equation*}
	\sup_{a\in\mathcal{S}^{2L-1}} \left\Vert \left\langle \wtilde{\zeta}_i, a\right\rangle \right\Vert_{\Psi_1} \leq \psi,
	\end{equation*}
	for all $i=1,\ldots,N$, and for and some $\psi>0$. Hence, the condition~\eqref{eq:sec cond: uniform sub-exp} is satisfied.
	
	The second condition required by Theorem 1 in~\cite{ADAMCZAK2011195} is that there exists a constant $K\geq 1$ such that
	\begin{equation} \label{eq:sec cond: boundedness}
	\Pr \left\{ \max_{i\leq N} \frac{\left\Vert \wtilde{\zeta}_i\right\Vert^2}{2L} > K \left(\frac{N}{2L}\right)^{1/2} \right\} \leq e^{-\sqrt{2L}}.
	\end{equation} 
	Note that \eqref{eq:sec cond: boundedness} does not follow immediately from Lemma~\ref{lem:tilde_eta norm concentration}. Therefore, we instead define restricted random vectors $\bar{\zeta}_i$ which satisfy the boundedness condition~\eqref{eq:sec cond: boundedness} (as well as condition~\eqref{eq:sec cond: uniform sub-exp}), and are equivalent to the original vectors $\wtilde{\zeta}_i$ with high probability. Consider the random vector~$\bar{\zeta}$
	\begin{equation*}
	\bar{\zeta} = 
	\begin{dcases}
	\wtilde{\zeta}, & \frac{\left\Vert \wtilde{\zeta}\right\Vert^2}{2L} \leq K \left(\frac{N}{2L}\right)^{1/2},  \\
	\vec{0}, & \frac{\left\Vert \wtilde{\zeta}\right\Vert^2}{2L} > K \left(\frac{N}{2L}\right)^{1/2},
	\end{dcases}
	\end{equation*}
	where $\vec{0}$ denotes the $2L\times 1$ vector of zeros.
	Note that for two random variables $ X_1 $ and $ X_2 $, if $ |X_1|  \leq |X_2| $ , then $ \|X_1\|_{\psi_1} \leq \|X_2\|_{\psi_1} $.
	Clearly, the first condition~\eqref{eq:sec cond: uniform sub-exp} is satisfied for $\bar{\zeta}$ since $\left\vert \left\langle \bar{\zeta}, a \right\rangle\right\vert \leq \left\vert \left\langle \wtilde{\zeta}, a \right\rangle\right\vert$. The second condition~\eqref{eq:sec cond: boundedness} is also satisfied due to the definition of $\bar{\zeta}$. 
	
	Up to this point, we showed that Theorem 1 in~\cite{ADAMCZAK2011195} holds for the vectors $\bar{\zeta}_1,\ldots,\bar{\zeta}_N$ ($N$ i.i.d. samples of $\bar{\zeta}$). Explicitly, there exist constants $C'_4,c_3$ such that
	
	\begin{equation}\label{eq:thm1 of adamczak for zeta bar}
	\Pr \left\{ \sup_{a\in S^{n-1}} \left| \frac{1}{N} \sum_{i=1}^N(|\langle\bar{\zeta}_i,a\rangle|^2 - \EE|\langle\bar{\zeta}_i,a\rangle|^2) \right|  > C'_4 \sqrt{\frac{2L}{N}}\right\} \leq 2e^{-c_3\sqrt{2L}}.
	\end{equation}
	
	In what follows, we evaluate the quantity $\mathbb{E}\left\vert \left\langle \bar{\zeta}, a \right\rangle \right\vert^2$ for $a\in\mathcal{S}^{2L-1}$ (which takes part in the bound \eqref{eq:thm1 of adamczak for zeta bar}), and show that it is sufficiently close to $\mathbb{E}\left\vert \left\langle \wtilde{\zeta}, a \right\rangle \right\vert^2$. Let us write
	\begin{align}
	\mathbb{E}\left\vert \left\langle \bar{\zeta}, a \right\rangle \right\vert^2 &= \int_{\mathbb{R}^{2L}} \left\vert \left\langle x, a \right\rangle \right\vert^2 dP_{\bar{\zeta}}(x)  \nonumber = \int_{\frac{\left\Vert x\right\Vert^2}{2L} \leq K \left(\frac{N}{2L}\right)^{1/2}} \left\vert \left\langle x, a \right\rangle \right\vert^2 dP_{\wtilde{\zeta}}(x) \\
	&= \int_{\mathbb{R}^{2L}} \left\vert \left\langle x, a \right\rangle \right\vert^2 dP_{\wtilde{\zeta}}(x) - \int_{\frac{\left\Vert x\right\Vert^2}{2L} > K \left(\frac{N}{2L}\right)^{1/2}} \left\vert \left\langle x, a \right\rangle \right\vert^2 dP_{\wtilde{\zeta}}(x) \nonumber \\ 
	&= \mathbb{E}\left\vert \left\langle \wtilde{\zeta}, a \right\rangle \right\vert^2 - \int_{\frac{\left\Vert x\right\Vert^2}{2L} > K \left(\frac{N}{2L}\right)^{1/2}} \left\vert \left\langle x, a \right\rangle \right\vert^2 dP_{\wtilde{\zeta}}(x) \nonumber \\
	&= \frac{1}{2} - \int_{\frac{\left\Vert x\right\Vert^2}{2L} > K \left(\frac{N}{2L}\right)^{1/2}} \left\vert \left\langle x, a \right\rangle \right\vert^2 dP_{\wtilde{\zeta}}(x) 
	\geq \frac{1}{2} - \int_{\frac{\left\Vert x\right\Vert^2}{2L} > K \left(\frac{N}{2L}\right)^{1/2}} \left\Vert x \right\Vert^2 dP_{\wtilde{\zeta}}(x), \label{eq: hat_eta inner prod mean bound}
	\end{align}
	where in the last inequality we used Cauchy-Schwarz inequality and the fact that  $a\in\mathcal{S}^{2L-1}$. 
	Consider the random variable $\chi = \left\Vert \wtilde{\zeta} \right\Vert^2 \bigg{\vert} \left\{\frac{\left\Vert \wtilde{\zeta}\right\Vert^2}{2L} > K \left(\frac{N}{2L}\right)^{1/2}\right\}$, whose probability density function is
	\begin{equation*}
	f_{\chi}(x) = 
	\begin{dcases}
	\frac{f_{\|\wtilde{\zeta}\|^2}(x)}{p_0(N,L)}, & \frac{\left\Vert \wtilde{\zeta}\right\Vert^2}{2L} > K \left(\frac{N}{2L}\right)^{1/2},  \\
	0, & \frac{\left\Vert \wtilde{\zeta}\right\Vert^2}{2L} \leq K \left(\frac{N}{2L}\right)^{1/2},
	\end{dcases}
	\end{equation*}
	where $p_0$ is the normalization for the density of $\chi$
	\begin{equation}\label{eq: p_0 def}
	p_0(N,L) = \Pr \left\{\frac{\left\Vert \wtilde{\zeta}\right\Vert^2}{2L} > K \left(\frac{N}{2L}\right)^{1/2}\right\}.
	\end{equation}
	Then,
	\begin{equation}\label{eq:tilde_eta norm tail int}
	\mathbb{E}\left[ \chi \right] = \frac{1}{p_0(N,L)}\int_{\frac{\left\Vert x\right\Vert^2}{2L} > K \left(\frac{N}{2L}\right)^{1/2}} \left\Vert x \right\Vert^2 dP_{\wtilde{\zeta}}(x).
	\end{equation}
	Note that for any non-negative continuous random variable $\chi$ 
	\begin{equation}\label{eq:survival_expectation}
	\mathbb{E}\left[ \chi \right] = \int_{0}^\infty y dP_\chi(y) = \int_{0}^\infty \left( \int_0^ydx \right) dP_\chi(y) = \int_0^\infty P(\chi>x) dx,
	\end{equation}
	where the last equality is due to interchanging the order of integration.
	Now, using \eqref{eq:survival_expectation}, we have 
	\begin{align}
	\mathbb{E}\left[ \chi \right] &=  \int_{0}^\infty \Pr\left\{\chi > x \right\} dx = \int_{0}^{K\sqrt{2NL}} 1 dx + \int_{K\sqrt{2NL}}^\infty \Pr\left\{\chi > x \right\} dx \\ 
	&\leq K \sqrt{2NL} + \frac{1}{p_0(N,L)} \int_{K\sqrt{2NL}}^\infty C_2 e^{-c_2 (2/L)^{1/4} \sqrt{x-L}} dx,\label{eq:Echi_bound}
	\end{align}
	where we substituted $x = 2L(1/2 + t/(2\sqrt{2L})) = L+t\sqrt{L/2}$ (or $ t = \sqrt{2/L}(x-L)$) in the bound of Lemma~\ref{lem:tilde_eta norm concentration}. 
	Note that according to Lemma~\ref{lem:tilde_eta norm concentration} (by substituting $1/2 + t/(2\sqrt{2L}) = K\sqrt{N/(2L)}$ and after some manipulation)
	\begin{equation}\label{eq:p_0 bound}
	p_0(N,L) \leq C_2 e^{-c_2 (2N)^{1/4} \sqrt{\sqrt{2}K - \sqrt{L/N}}} \leq C_2 e^{-c_2 (2N)^{1/4} \sqrt{\sqrt{2}K - 1}} = C_2 e^{-c_4 N^{1/4} }, 
	\end{equation}
	where $c_4 = c_2 \sqrt{\sqrt{2}K - 1}$. 
	Overall, by~\eqref{eq:tilde_eta norm tail int}, \eqref{eq:Echi_bound}, and \eqref{eq:p_0 bound} we have that
	\begin{equation}\label{eq: tilde_eta norm tail integral two term sum}
	\begin{aligned}
		\int_{\frac{\left\Vert x\right\Vert^2}{2L} > K \left(\frac{N}{2L}\right)^{1/2}} \left\Vert x \right\Vert^2 dP_{\wtilde{\zeta}}(x)
		&= \mathbb{E}\left[ \chi \right] \cdot p_0(N,L) \\
		&\leq K\sqrt{2NL} C_2 e^{-c_4 N^{1/4} } +  \int_{K\sqrt{2NL}}^\infty C_2 e^{-c_2 (2/L)^{1/4} \sqrt{x-L}} dx. 
		\end{aligned}
	\end{equation}
	Note the the term $K\sqrt{2NL} C_2 e^{-c_4 N^{1/4} }$ decays exponentially with $N^{1/4}$, and therefore, it is clear that for a large enough $ N $
	\begin{equation} \label{eq:first term asym decay}
	K\sqrt{2NL} C_2 e^{-c_4 N^{1/4} } \leq C''_4\sqrt{\frac{L}{N}},
	\end{equation}
	for some constant $ C''_4 $, which can be chosen arbitrarily small.
	Also, using a standard integration formula \cite{abramowitz1964handbook}, we have that
	\begin{equation*}
	\int_{K\sqrt{2NL}}^\infty C_2 e^{-c_2 (2/L)^{1/4} \sqrt{x-L}} dx = 2 C_2 e
	^{-\sqrt{y_0}} \left( \sqrt{y_0} + 1\right),
	\end{equation*}
	with $y_0 = \sqrt{2 c_2^4 / L}\cdot (K\sqrt{2NL}-L)$, and thus, it also follows that, for a large enough $ N $, 
	\begin{equation}\label{eq:second term asym decay}
	\int_{K\sqrt{2NL}}^\infty C_2 e^{-c_2 (2/L)^{1/4} \sqrt{x-L}} dx \leq C'''_4\sqrt{\frac{L}{N}},
	\end{equation}
	for some constant $ C'''_4 $, which can be chosen arbitrarily small.
	Next, by plugging~\eqref{eq:second term asym decay} and~\eqref{eq:first term asym decay} in~\eqref{eq: tilde_eta norm tail integral two term sum} and using the result in~\eqref{eq: hat_eta inner prod mean bound} we have 
	\begin{equation} \label{eq:expected zeta times a bound}
	\frac{1}{2} - (C''_4 + C'''_4)\left(\sqrt{\frac{L}{N}}\right) \leq 
	\mathbb{E}\left\vert \left\langle \bar{\zeta}, a \right\rangle \right\vert^2 \leq \mathbb{E}\left\vert \left\langle \wtilde{\zeta}, a \right\rangle \right\vert^2 = \frac{1}{2},
	\end{equation}
	where the second inequality is due to the definition of $\bar{\zeta}$, and the last equality is since~$ \mathbb{E}\left[ \wtilde{\zeta} \wtilde{\zeta}^T\right] =\frac{1}{2} I_{2L} $, from which it follows that
	$$
	\mathbb{E}\left\vert \left\langle \wtilde{\zeta}, a \right\rangle \right\vert^2 =\EE \left(\sum_{k=1}^{2L} \left(\wtilde{\zeta}[k]\right)^2 {a[k]}^2 + 2\sum_{k1\neq k_2} \wtilde{\zeta}[k_1] a[k_1] \wtilde{\zeta}[k_2] a[k_2]\right) = 1/2,
	$$ 
	where $ \wtilde{\zeta}[k] $ is the $ k $th element of $ \wtilde{\zeta} $.

	Note that for any $ a,b,b_0 >0 $ and a small enough $ \eps $, if $ b_0-\eps \leq b\leq b_0 $ then
	\begin{equation}\label{eq:abs_bound}
	|a-b| = \max(a-b,b-a)\geq \max(a-b_0, b_0-\eps-a)\geq \max(a-b_0-\eps, b_0-a - \eps) = |a-b_0|-\eps.
	\end{equation}
	From \eqref{eq:expected zeta times a bound}  and \eqref{eq:abs_bound}, we have that
	\begin{equation}\label{eq:enlarge_sup_zeta_bar}
	\left| \frac{1}{N} \sum_{i=1}^N(|\langle\bar{\zeta}_i,a\rangle|^2 - \EE|\langle\bar{\zeta}_i,a\rangle|^2) \right| 
	\geq
	\left\vert \frac{1}{N} \sum_{i=1}^N \left\vert \left\langle \bar{\zeta}_i, a \right\rangle \right\vert^2 - \frac{1}{2} \right\vert - \left|(C''_4 + C'''_4)\left(\sqrt{\frac{L}{N}}\right)\right|.
	\end{equation}
	Therefore, rewriting \eqref{eq:thm1 of adamczak for zeta bar} using \eqref{eq:enlarge_sup_zeta_bar} we have 
	\begin{equation*}
	\Pr\left\{ \sup_{a\in\mathcal{S}^{2L-1}} \left\vert \frac{1}{N} \sum_{i=1}^N \left\vert \left\langle \bar{\zeta}_i, a \right\rangle \right\vert^2 - \frac{1}{2} \right\vert - \left|(C''_4 + C'''_4)\left(\sqrt{\frac{L}{N}}\right)\right| > C'_4 \sqrt{\frac{2L}{N}} \right\} \leq 2e^{-c_3 \sqrt{2L}}.
	\end{equation*}
	We absorb the term $(C''_4 + C'''_4)\left(\sqrt{\frac{L}{N}}\right)$ into $C'_4\left(\sqrt{\frac{2L}{N}}\right)$, and we have that, for some constant $ C_4 $,
	\begin{equation}\label{eq:thm1 of adamczak for zeta bar_short}
	\Pr\left\{ \sup_{a\in\mathcal{S}^{2L-1}} \left\vert \frac{1}{N} \sum_{i=1}^N \left\vert \left\langle \bar{\zeta}_i, a \right\rangle \right\vert^2 - \frac{1}{2} \right\vert > C_4 \sqrt{\frac{2L}{N}} \right\} \leq 2e^{-c_3 \sqrt{2L}}.
	\end{equation}

	Equation \eqref{eq:thm1 of adamczak for zeta bar_short} is the result of applying  Theorem 1 in~\cite{ADAMCZAK2011195} to the truncated vectors $\bar{\zeta}_1,\ldots,\bar{\zeta}_N$. Next, we adapt this result to the original vectors $\wtilde{\zeta}_1,\ldots,\wtilde{\zeta}_N$. Using \eqref{eq: p_0 def}, \eqref{eq:p_0 bound}, and the union bound, we have that
	\begin{equation*}
	\Pr\left\{ \max_{i\leq N} \frac{\left\Vert \wtilde{\zeta}_i \right\Vert^2}{2L} > K \sqrt{\frac{N}{2L}} \right\} \leq N\cdot p_0(N,n) \leq C_2 N e^{-c_4 N^{1/4}}.
	\end{equation*}
	Denote the event $\left\{ \max_{i\leq N} \frac{\left\Vert \wtilde{\zeta}_i \right\Vert^2}{2L} > K \sqrt{\frac{N}{2L}} \right\}$ by $ A $, and its compliment by $ \bar{A} $.
	We have that
	\begin{equation}\label{eq:cond_prob_sum_A}
	\Pr\left\{\left. \sup\limits_{a\in\mathcal{S}^{2L-1}} \left\vert \frac{1}{N} \sum\limits_{i=1}^N \left\vert \left\langle \wtilde{\zeta}_i, a \right\rangle \right\vert^2 - \frac{1}{2} \right\vert > C_4 \sqrt{\frac{2L}{N}} ~~\right|~~ A \right\} \Pr\left\{ A \right\} \leq \Pr\left\{ A \right\} \leq C_2 N e^{-c_4 N^{1/4}},
	\end{equation}
	and
	\begin{multline}\label{eq:cond_prob_sum_notA}
	\Pr\left\{\left. \sup\limits_{a\in\mathcal{S}^{2L-1}} \left\vert \frac{1}{N} \sum\limits_{i=1}^N \left\vert \left\langle \wtilde{\zeta}_i, a \right\rangle \right\vert^2 - \frac{1}{2} \right\vert > C_4 \sqrt{\frac{2L}{N}} ~\right|~ \bar{A} \right\} \Pr\left\{ \bar{A} \right\} \leq \\
	\Pr\left\{ \sup\limits_{a\in\mathcal{S}^{2L-1}} \left\vert \frac{1}{N} \sum\limits_{i=1}^N \left\vert \left\langle \bar{\zeta}_i, a \right\rangle \right\vert^2 - \frac{1}{2} \right\vert > C_4 \sqrt{\frac{2L}{N}}  \right\} 
	\leq  2e^{-c_3 \sqrt{2L}},
	\end{multline}
	and therefore, using the law of total probability, by combining~\eqref{eq:cond_prob_sum_A} and~\eqref{eq:cond_prob_sum_notA}, we have  
	\begin{equation}\label{eq: concentration inner prod}
	\Pr\left\{ \sup\limits_{a\in\mathcal{S}^{2L-1}} \left\vert \frac{1}{N} \sum\limits_{i=1}^N \left\vert \left\langle \wtilde{\zeta}_i, a \right\rangle \right\vert^2 - \frac{1}{2} \right\vert > C_4 \sqrt{\frac{2L}{N}} \right\} \leq 2e^{-c_3 \sqrt{2L}} + C_2 N e^{-c_4 N^{1/4}}.
	\end{equation}
	Noting that 
	\begin{equation}\label{eq:inner prod as matrix mul}
	\begin{aligned}
	&\sup_{a\in\mathcal{S}^{2L-1}} \left\vert \frac{1}{N} \sum_{i=1}^N \left\vert \left\langle \wtilde{\zeta}_i, a \right\rangle \right\vert^2 - \frac{1}{2} \right\vert 
	= \sup_{a\in\mathcal{S}^{2L-1}} \left\vert a^T \left[\frac{1}{N} \sum_{i=1}^N \wtilde{\zeta}_i \wtilde{\zeta}_i^T\right] a - \frac{1}{2} \right\vert  \\
	&= \sup_{a\in\mathcal{S}^{2L-1}} \left\vert a^T \left[\frac{1}{N} \sum_{i=1}^N \wtilde{\zeta}_i \wtilde{\zeta}_i^T - \frac{1}{2}\cdot I_{2L} \right] a \right\vert = \left\Vert \frac{1}{N} \sum_{i=1}^N \wtilde{\zeta}_i \wtilde{\zeta}_i^T - \frac{1}{2}\cdot I_{2L} \right\Vert,
	\end{aligned}
	\end{equation}
	we rewrite \eqref{eq: concentration inner prod} as
	\begin{equation*}
	\Pr\left\{\left\Vert \frac{1}{N} \sum_{i=1}^N \wtilde{\zeta}_i \wtilde{\zeta}_i^T - \frac{1}{2}\cdot I_{2L} \right\Vert > C_4 \sqrt{\frac{2L}{N}}\right\} \leq 2e^{-c_3 \sqrt{2L}} + C_2 N e^{-c_4 N^{1/4}}
	\end{equation*}
	which concludes the proof.
	
\end{proof}

Essentially, Lemma~\ref{lem:tilde_eta sample covariance concentration} establishes the concentration of the sample covariance matrix of the random vector $\wtilde{\zeta}$ around the population covariance matrix (which is $1/2 \cdot I_{2L}$).

%%%%%%%%%%%%%%%%%%%%%%%%%%%%%%%%%
%%%%%%%%%%%%%%%%%%%%%%%%%%%%%%
%%%%%%%%%%%%%%%

\subsection{Proof of Lemma \ref{lem:u_1 sigma_series_bound}}\label{subsec:proof_lem_u_1 sigma_series_bound}
We begin with Lemma~\ref{lem:Cz1tilde sigma_series_bound} that will help us prove Lemma \ref{lem:u_1 sigma_series_bound}.

\begin{lemma}\label{lem:Cz1tilde sigma_series_bound}
	Let 
	\begin{equation}\label{eq:Sigma1_u def}
	\Sigma_u^{(1)}[k_1,k_2] = 2\lambda^2 u^{(1)} \left( u^{(1)} \right)^*,
	\end{equation}
	where $ u^{(1)} $ is defined in \eqref{eq:um and zm def}, and let $\wtilde{C}_z^{(1)}$ be given by~\eqref{eq:z sample cov def} and~\eqref{eq:z sample cov bias correction}. Then, $ \Sigma_u^{(1)} - \wtilde{C}_z^{(1)} $ can be written as
	\begin{equation*}
	\Sigma_u^{(1)} - \wtilde{C}_z^{(1)} = A'_0 + \sigma^2 A'_2+ \sigma^3 A'_3+ \sigma^4 A'_4,
	\end{equation*}
	where $ A'_0, A'_2, A'_3 $ and $ A'_4 $ are independent of $ \sigma $.
\end{lemma}
\begin{proof}
	We define the random vector $\hat{\xi}\in\mathbb{C}^L$ such that 
	\begin{equation*}
	\hat{\eta} = \sigma \hat{\xi},
	\end{equation*}
	where $\hat{\eta}$ is the noise vector from the  model~\eqref{eq:MRFA_model_def_Fourier}, and therefore 
	\begin{equation}\label{eq:xi_def}
	\hat{\xi} \sim \mathcal{CN}(0,I_{L\times L}).
	\end{equation}
	Then, $z^{(1)}$ of \eqref{eq:zm_expression} can be written as
	\begin{equation}
	z^{(1)}[k] = |a|^2 \omega^{-s} u^{(1)}[k] + \sigma \left( a \omega^{s k} \hat{\theta}[k] \hat{\xi}^*[k+1] + a^*\omega^{-s(k+1)}\hat{\theta}^*[k+1] \hat{\xi}[k]\right) + \sigma^2 \hat{\zeta}, \label{eq:z_1 reformulation}
	\end{equation}
	where we defined the random vector $\hat{\zeta}\in\mathbb{C}^L$ via
	\begin{equation}\label{eq:zeta_def}
	\hat{\zeta}[k] = \hat{\xi}[k]\hat{\xi}^*[k+1].
	\end{equation}
	We have from ~\eqref{eq:z sample cov def} and~\eqref{eq:z sample cov bias correction} that
	\begin{equation}
	\wtilde{C}_z^{(1)} = \frac{1}{N} \sum_{i=1}^N z^{(1)}_i \left( z^{(1)}_i \right)^* - \wtilde{\Sigma}_{\epsilon}^{(1)}, \label{eq:C_z proof def}
	\end{equation}
	where $z^{(1)}_1,\ldots,z^{(1)}_N \in \mathbb{C}^L$ are i.i.d. samples of $z^{(1)}$, and $\wtilde{\Sigma}_{\epsilon}^{(1)}$ is the estimated bias matrix (as in~\eqref{eq:Sigma_epsilon def} except that $p_x$ is replaced with its estimate $\wtilde{p}_x$)
	\begin{equation}
	\wtilde{\Sigma}^{(1)}_\epsilon[k_1,k_2] = 
	\begin{dcases}
	0, & k_1 \neq k_2, \\
	\sigma^2 \left( \wtilde{p}_x[k_1] + \wtilde{p}_x[k_1 + 1] \right) + \sigma^4, & k_1=k_2, \label{eq:Sigma_epsilon_tilde def}
	\end{dcases}
	\end{equation}
	recalling from \eqref{eq:power spectrum estimate} and \eqref{eq:MRFA_model_def_Fourier} that
	\begin{equation}
	\wtilde{p}_x[k] = \frac{1}{N} \sum_{i=1}^N | \hat{y}_i [k]|^2 - \sigma^2 = \frac{1}{N} \sum_{i=1}^N | \omega^{sk}\hat{x}_i[k] + \sigma \hat{\xi}_i[k] |^2 - \sigma^2. \label{eq:p_x est noise}
	\end{equation}

	Since $ |\omega| = 1 $, from \eqref{eq:Sigma_epsilon_tilde def} and \eqref{eq:p_x est noise}, we have
	\begin{equation}\label{eq:Sigma_diag_expantion}
	\begin{array}{ll}
	\wtilde{\Sigma}^{(1)}_\epsilon[k,k]  &=\sigma^2\left(\frac{1}{N} \sum_{i=1}^N | \omega^{sk}\hat{x}_i[k] + \sigma \hat{\xi}_i[k] |^2  + | \omega^{s(k+1)}\hat{x}_i[k+1] + \sigma \hat{\xi}_i[k+1] |^2\right)- \sigma^4\\
	&=\sigma^2\left(\frac{1}{N} \sum_{i=1}^N |\hat{x}_i[k]|^2 +  |\hat{x}_i[k+1]|^2\right)\\
	&+ \sigma^3\left(\frac{1}{N} \sum_{i=1}^N \hat{x}_i[k] \hat{\xi}^*_i[k]+ \hat{x}^*_i[k] \hat{\xi}_i[k]+\hat{x}_i[k+1] \hat{\xi}^*_i[k+1]+\hat{x}^*_i[k+1] \hat{\xi}_i[k+1]\right)\\
	&+ \sigma^4\left(\frac{1}{N} \sum_{i=1}^N|\hat{\xi}_i[k] |^2 + | \hat{\xi}_i[k+1] |^2 - 1\right).
	\end{array}
	\end{equation}
	Next, we substitute \eqref{eq:z_1 reformulation} into \eqref{eq:C_z proof def}
	\begin{equation*}
	\begin{array}{rlll}
	\wtilde{C}_z^{(1)}[k_1,k_2] = & \frac{1}{N}\sum_{i=1}^N &\multicolumn{2}{l}{|a_i|^4 u^{(1)}[k_1]\left(u^{(1)}\right)^*[k_2]}\\
	&& + \sigma^2 &\left( a_i \omega^{s k_1} \hat{\theta}[k_1] \hat{\xi}_i^*[k_1+1] + a_i^*\omega^{-s(k_1+1)}\hat{\theta}^*[k_1+1] \hat{\xi}_i[k_1]\right) \\
	&&&\left( a_i \omega^{s k_2} \hat{\theta}[k_2] \hat{\xi}_i^*[k_2+1] + a_i^*\omega^{-s(k_2+1)}\hat{\theta}^*[k_2+1] \hat{\xi}_i[k_2]\right)^*\\
	&&+ \sigma^4 & \hat{\zeta}_i[k_1]\hat{\zeta}_i^*[k_2]\\
	&\multicolumn{3}{l}{-\wtilde{\Sigma}^{(1)}_\epsilon[k_1,k_2]}\\
	=&\multicolumn{3}{l}{\frac{\sum_{i=1}^N|a_i|^4}{N} u^{(1)}[k_1]\left(u^{(1)}\right)^*[k_2] + \sigma^2 A'_2[k_1,k_2]+ \sigma^3 A'_3[k_1,k_2] + \sigma^4 A'_4[k_1,k_2],}
	\end{array}
	\end{equation*}
	where (recalling that $ \hat{x}_i[k] = a_i\hat{\theta}[k] $)
	\begin{equation}\label{eq:A'2_def}
	A'_2[k_1,k_2] = \begin{dcases}
	A''_2[k_1,k_2] & k_1 \neq k_2,\\
	A''_2[k_1,k_2] - | a_i\hat{\theta}[k_1] |^2 -  | a_i\hat{\theta}[k_1+1] |^2,
	& k_1 = k_2,
	\end{dcases}
	\end{equation}
	
	\begin{equation*}
	\begin{array}{rll}
	A''_2[k_1,k_2] &=\frac{1}{N}\sum\limits_{i=1}^N &\left( a_i \omega^{s k_1} \hat{\theta}[k_1] \hat{\xi}_i^*[k_1+1] + a_i^*\omega^{-s(k_1+1)}\hat{\theta}^*[k_1+1] \hat{\xi}_i[k_1]\right) \\
	&&\times\left( a^*_i \omega^{-s k_2} \hat{\theta}^*[k_2] \hat{\xi}_i[k_2+1] + a_i\omega^{s(k_2+1)}\hat{\theta}[k_2+1] \hat{\xi}^*_i[k_2]\right) \\
	&= \frac{1}{N}\sum\limits_{i=1}^N & |a_i|^2 \omega^{s (k_1-k_2)} \hat{\theta}[k_1]\hat{\theta}^*[k_2] \hat{\xi}_i^*[k_1+1] \hat{\xi}_i[k_2+1]\\
	&&+ a_i a_i \omega^{s(k_1+k_2+1)}\hat{\theta}[k_1]\hat{\theta}[k_2+1] \hat{\xi}^*_i[k_1+1]\hat{\xi}^*_i[k_2]\\
	&&+ a^*_i a^*_i \omega^{-s (k_2+k_1+1)} \hat{\theta}^*[k_1+1] \hat{\theta}^*[k_2]  \hat{\xi}_i[k_1] \hat{\xi}_i[k_2+1] \\
	&&+ |a_i|^2 \omega^{s(k_1-k_2+2)}\hat{\theta}^*[k_1+1] \hat{\theta}[k_2+1] \hat{\xi}_i[k_1] \hat{\xi}^*_i[k_2],
	\end{array}
	\end{equation*}
	\begin{equation}\label{eq:A'3_def}
	A'_3[k_1,k_2] =  
	\begin{dcases}
	0 & k_1 \neq k_2, \\
	\begin{array}{ll}
	-\frac{1}{N} \sum_{i=1}^N& \hat{x}_i[k_1] \hat{\xi}^*_i[k_1]+ \hat{x}^*_i[k_1] \hat{\xi}_i[k_1]\\&+\hat{x}_i[k_1+1] \hat{\xi}^*_i[k_1+1]+\hat{x}^*_i[k_1+1] \hat{\xi}_i[k_1+1]
		\end{array}& k_1 = k_2,
	\end{dcases}
	\end{equation}
	and 
	\begin{equation}\label{eq:A4'_def}
	A'_4 = \frac{1}{N}\sum_{i=1}^N \hat{\zeta}_i \hat{\zeta}_i^* - I_{L\times L} + E^{(0)} + E^{(1)},
	\end{equation}
	with
	\begin{equation} \label{eq:Em_def}
	E^{(m)}[k_1,k_2] = 
	\begin{dcases}
	0, & k_1 \neq k_2, \\
	\frac{1}{N} \sum_{i=1}^N | \hat{\xi}_i[k_1 + m] |^2 - 1, & k_1=k_2,
	\end{dcases}
	\end{equation}
	and $ \hat{\zeta} $ defined in \eqref{eq:zeta_def}. From the definition of $ 	\Sigma_u^{(1)}[k_1,k_2] $ in \eqref{eq:Sigma1_u def}, defining
	\begin{equation}\label{eq:A'0_def}
	A'_0[k_1,k_2] = \left(\frac{\sum_{i=1}^N|a_i|^4}{N}- 2\lambda \right) u^{(1)}[k_1]\left(u^{(1)}\right)^*[k_2]
	\end{equation}
	concludes the proof.
\end{proof}

We will also need the following two supporting lemmas.
\begin{lemma}\label{lem:from_C_to_args}
	Let $u, \wtilde{u}\in \CC^L$ such that $ |u[k]| = |v[k]| = 1 $ for $ k=0,\ldots, L-1 $. Let $ w, \wtilde{w} \in (\RR \mod 2\pi)^L $ such that $ w[k] = \arg(u[k]) $ and $ \wtilde{w}[k] = \arg(\wtilde{u}[k]) $, then 
	\begin{equation}\label{key}
	\|\wtilde{w}-w\| \leq \|\wtilde{u}-u\|.
	\end{equation}
\end{lemma}
\begin{proof}
	Denote, for any $ k $, $ \eta[k] = |\wtilde{u}[k]-u[k]|$. Then $ \wtilde{u}[k] = u[k]+\eta[k] e^{i\phi[k]} $ for some $ \phi[k] $. In this notation, 
	\begin{equation}\label{eq:wtilde_eta_theta}
	\wtilde{w}[k] = w[k] + \arctan\left(\frac{\eta[k] \sin \phi[k]}{1+\eta[k] \cos\phi[k]}\right).  
	\end{equation}
	or,
	\begin{equation}\label{eq:wtilde1 formula}
	|\wtilde{w}[k] - w[k]| = \left|\arctan\left(\frac{\eta[k] \sin \phi[k]}{1+\eta[k] \cos\phi[k]}\right)\right|.  
	\end{equation}

	It is easy to verify that the $ \phi[k] $ that maximizes $ \wtilde{w}[k]$  in \eqref{eq:wtilde_eta_theta} is $ \phi_0[k] $ such that $ \cos\phi_0[k] = \eta[k] $. Thus, since $ \arctan $ is monotonously increasing, $ \sin(\phi[k])\leq 1 $, and $ |\arctan x| <|x| $, we have
	\begin{equation}\label{eq:wtildek-wk formula}
	|\wtilde{w}[k] - w[k]| \leq \left|\arctan\left(\frac{\eta[k]}{1+\eta^2[k] }\right)\right|\leq \left|\arctan\left(\eta[k]\right)\right| \leq |\eta[k]| =  |\wtilde{u}[k]-u[k]|.  
	\end{equation}
	Since \eqref{eq:wtildek-wk formula} holds for all $ k $, we have that 
	\begin{equation}\label{eq:w_diff le u_diff}
	\|\wtilde{w}-w \| \leq \|\wtilde{u}-u\|.
	\end{equation}
\end{proof}
\begin{lemma}\label{lem:norm1_to_element_norm1}
	Let $ u, \wtilde{u} \in \CC^L $, with $ \|u\| = \|\wtilde{u}\| = 1 $, and denote $ u_p, \wtilde{u}_p \in \CC^L $ such that $ u_p[k] = \frac{u[k]}{|u[k]|} $ and $ \wtilde{u}_p[k] = \frac{\wtilde{u}[k]}{|\wtilde{u}[k]|} $.  Denote $ \delta_1 = \min_{k\in\{0,\ldots,L-1\}}|u[k]| $, and assume that $\delta_1>0 $ and $ \|\wtilde{u} - u\|\leq \delta_1/2 $. Then,
	\begin{equation}\label{key}
	\|\wtilde{u}_p - u_p\| \leq \frac{3\|\wtilde{u} - u\|}{\delta_1^2}.
	\end{equation}
\end{lemma}
\begin{proof}
	Denote  $ c,\wtilde{c}\in \RR^L $ such that $ c[k] = \frac{1}{|u[k]|} $ and $ \wtilde{c}[k] = \frac{1}{|\wtilde{u}[k]|} $ for $ k=0,\ldots,L-1 $. 
	From the definition of $ \delta_1 $ and the reverse triangle inequality we have
	
	\begin{equation}\label{eq:utildek bound1}
	|\wtilde{u}[k]| = |u[k] - u[k] + \wtilde{u}[k]|\geq \left|\left|u[k]\right| -\left|\wtilde{u}[k] - u[k]\right|\right|.
	\end{equation}
	Since $ |u[k]| \geq \delta_1  $ and $ \left|\wtilde{u}[k] - u[k]\right| \leq \|\wtilde{u} - u\|\leq \delta_1/2 $ we have that,
	$$ 
	\left|\left|u[k]\right| -\left|\wtilde{u}[k] - u[k]\right|\right|  = \left|u[k]\right| -\left|\wtilde{u}[k] - u[k]\right| \geq \delta_1 -\|\wtilde{u} - u\|,
	$$
	or, together with \eqref{eq:utildek bound1}, we have 
	\begin{equation}\label{eq:utildek bound}
	|\wtilde{u}[k]| \geq \delta_1 -\|\wtilde{u} - u\|,
	\end{equation}
	and thus
	\begin{equation}\label{eq:ck_minus_cktilde}
	\left| \wtilde{c}[k] - c[k] \right| = 	\left| \frac{1}{|\wtilde{u}[k]|} - \frac{1}{|u[k]|} \right| = \left| \frac{|u[k]| - |\wtilde{u}[k]| }{|\wtilde{u}[k]| |u[k]|} \right| \leq \frac{|u[k] - \wtilde{u}[k]| }{|\wtilde{u}[k]| \delta_1} \leq \frac{\|u - \wtilde{u}\| }{ \delta_1 \left(\delta_1 -\|\wtilde{u} - u\|\right)}.
	\end{equation}
	Note that for any $ v\in\CC^L $ and $ a\in \RR^L $ we have that
	\begin{equation}\label{eq:odot norm bound}
	\|v\odot a\| = \sqrt{\sum_{k=0}^{L-1} v[k]^2a[k]^2}\leq \|v\| \max_{k\in \{0,\ldots,L-1\}} |a[k]| .
	\end{equation}
	
	Using the triangle inequality and \eqref{eq:odot norm bound}, we have
	\begin{multline}\label{eq:normlized u minus u_tilde}
	\|\wtilde{u}_p - u_p\| = \|\wtilde{u}\odot\wtilde{c} - u\odot c\| = 	\|\wtilde{u}\odot\wtilde{c} -\wtilde{u}\odot c + \wtilde{u}\odot c - u\odot c\| \\ \leq \|\wtilde{u}\odot\wtilde{c} -\wtilde{u}\odot c \| + \| \wtilde{u}\odot c - u\odot c\| = \|\wtilde{u}\odot (\wtilde{c} - c) \| + \| (\wtilde{u} - u)\odot c\| \\
	\leq 
	\max_{k\in \{0,\ldots,L-1\}}|\wtilde{c}[k] - c[k]| + \| (\wtilde{u} - u)\| \max_{k\in \{0,\ldots,L-1\}} |c[k]|.
	\end{multline}
	Combining \eqref{eq:normlized u minus u_tilde} with \eqref{eq:ck_minus_cktilde} and the fact that $ |c[k]| = 1/|u[k]|\leq 1/\delta_1 $, we get
	\begin{multline*}
	\|\wtilde{u}_p - u_p\| \leq \frac{\|u - \wtilde{u}\| }{ \delta_1 \left(\delta_1 -\|\wtilde{u} - u\|\right)} + \frac{\| (\wtilde{u} - u)\|}{\delta_1} = \frac{\| (\wtilde{u} - u)\|}{\delta_1} \left( 1 + \frac{1}{ \left(\delta_1 -\|\wtilde{u} - u\|\right)} \right) \\
	\leq \frac{\| (\wtilde{u} - u)\|}{\delta_1} \left( 1 + \frac{2}{ \delta_1} \right) \leq \frac{3\| (\wtilde{u} - u)\|}{\delta_1^2},
	\end{multline*}
	where the second inequality is due to the assumption  $ \|\wtilde{u} - u\|\leq \delta_1/2 $ and the last inequality holds since $ \delta_1 < 1 $.

\end{proof}

Next we prove Lemma~\ref{lem:u_1 sigma_series_bound}.
\begin{proof}
We have from Lemma~\ref{lem:Cz1tilde sigma_series_bound} that 
\begin{equation*}
\Sigma_u^{(1)} - \wtilde{C}_z^{(1)} = A'_0 + \sigma^2 A'_2+ \sigma^3 A'_3+ \sigma^4 A'_4,
\end{equation*}
and so, 
\begin{equation} \label{eq:norm_Cz1tilde_err}
\|\Sigma_u^{(1)} - \wtilde{C}_z^{(1)}\| \leq \|A'_0\| + \sigma^2 \|A'_2\|+ \sigma^3\| A'_3\|+ \sigma^4 \|A'_4\|.
\end{equation}
Now, recall from \eqref{eq:Sigma1_u def} that $\Sigma_u^{(1)}$ is a rank one matrix whose leading eigenvalue is $ 2\lambda^2 $. Due to the assumption in Lemma~\ref{lem:u_1 sigma_series_bound},
\begin{equation}\label{eq:2lnorm_u1}
2\lambda^2 \sum_{k=0}^{L-1} \left\vert u^{(1)} [k] \right\vert^2 = \frac{2\lambda^2 \gamma_1}{L}.
\end{equation}

Now, since $\wtilde{u}^{(1)}$ is the leading eigenvector of $\wtilde{C}_{z}^{(1)}$, then by the Davis-Kahan $\sin\Theta$ theorem~\cite{davis1970rotation} (noting that the spectral gap of $\Sigma_u^{(1)}$ is equal to its largest eigenvalue since it is a rank one matrix), using \eqref{eq:2lnorm_u1}, it follows that
\begin{equation}\label{eq:DK_bound u1}
\left\Vert \wtilde{u}^{(1)} - \alpha_1 \frac{u^{(1)}}{\left\Vert u^{(1)} \right\Vert} \right\Vert \leq
\frac{L \left\Vert\Sigma_u^{(1)} - \wtilde{C}_z^{(1)} \right\Vert}{2\lambda^2\gamma_1},
\end{equation}
where $\alpha_1$ is some complex constant satisfying $\left\vert \alpha_1 \right\vert = 1$.

Next we show that there is some $s_1\in \{0,\ldots,L-1\}$ such that 
\begin{equation}\label{key}
\left\Vert \wtilde{u}^{(1)} - e^{-\imath 2\pi s_1 /L } \frac{u^{(1)}}{\left\Vert u^{(1)} \right\Vert} \right\Vert \leq 
\frac{4}{\delta_1^2}\left\Vert \wtilde{u}^{(1)} - \alpha_1 \frac{u^{(1)}}{\left\Vert u^{(1)} \right\Vert} \right\Vert .
\end{equation}
Denote $ \wtilde{u}^{(1)}_p, u^{(1)}_p \in \CC^L $ such that $ \wtilde{u}^{(1)}_p[k] = \frac{\wtilde{u}^{(1)}[k]}{|\wtilde{u}^{(1)}[k]|} $ and $ u^{(1)}_p[k] = \frac{u^{(1)}[k]}{|u^{(1)}[k]|} $. First, suppose that,
$$ 
\left\|\wtilde{u}^{(1)} - \alpha_1 \frac{u^{(1)}}{\left\Vert u^{(1)} \right\Vert}\right\|\leq \delta_1/2,
$$
(the other case will be handled below).
Then, by Lemma \ref{lem:norm1_to_element_norm1} we have,
\begin{equation}\label{eq:u_p bound by u}
\|\wtilde{u}^{(1)}_p - \alpha_1 u^{(1)}_p\| \leq \frac{3}{\delta_1^2}\left\| \wtilde{u}^{(1)} - \alpha_1 \frac{u^{(1)}}{\left\Vert u^{(1)} \right\Vert} \right\|.
\end{equation}
Denote  $ w, \wtilde{w} \in \RR^L $ such that $ w[k] = \arg(u^{(1)}_p[k]) $ and $ \wtilde{w}[k] = \arg(\wtilde{u}^{(1)}_p[k]) $, and note that $ \arg(\alpha_1u^{(1)}_p[k]) = w_k + \beta_1 $, where $ \alpha_1 = e^{\imath \beta_1}$. Then by Lemma~\ref{lem:from_C_to_args} and by~\eqref{eq:u_p bound by u},
\begin{equation}\label{key}
\|w-\wtilde{w} - \beta_1 \mathbbm{1}\| \leq \|u^{(1)}_p - \alpha_1 \wtilde{u}^{(1)}_p\|\leq \frac{3}{\delta_1^2}\left\| \wtilde{u}^{(1)} - \alpha_1 \frac{u^{(1)}}{\left\Vert u^{(1)} \right\Vert} \right\|,
\end{equation}
where $ \mathbbm{1} $ is a vector of ones of length $ L $. Thus, if we write $ \wtilde{w} = w - \mathbbm{1}\beta_1 + \eps$, where $ \eps \in \RR^L $ than,  
\begin{equation}\label{eq:eps_vec_bound}
\|\eps\|\leq  \frac{3}{\delta_1^2}\left\| \wtilde{u}^{(1)} - \alpha_1 \frac{u^{(1)}}{\left\Vert u^{(1)} \right\Vert} \right\|.
\end{equation}
Summing over both sides of $ \mathbbm{1}\beta_1 = w - \wtilde{w} + \eps$ we get,
\begin{equation}\label{eq:L_beta_1}
L\beta_1 = \sum_{k=0}^{L-1}w[k] - \sum_{k=0}^{L-1}\wtilde{w}[k] + \sum_{k=0}^{L-1}\eps[k].
\end{equation}
Recall from \eqref{eq:umphasesum} that $\sum_{k=0}^{L-1} w[k] = 0 \mod 2\pi$ and from \eqref{eq:umupdate} that $\sum_{k=0}^{L-1} \wtilde{w}[k] = 0 \mod 2\pi$, and thus from~\eqref{eq:L_beta_1} we have 
\begin{equation}\label{key}
\beta_1  = \frac{2\pi s_1}{L} + \frac{\sum_{k=0}^{L-1}\eps[k]}{L},
\end{equation}
for some $s_1\in \{0,\ldots,L-1\}$. In other words, we showed that $ \beta_1 $, the phase of $ \alpha_1 $ of~\eqref{eq:DK_bound u1} is "not too far" from an $ L $'th root of one.
Now, note that,
\begin{equation}\label{eq:u_tilde_whole_shift_bound1}
\begin{array}{ll}
\left\Vert \wtilde{u}^{(1)} - 
e^{-\imath 2\pi s_1 /L } \frac{u^{(1)}}{\left\Vert u^{(1)} \right\Vert} \right\Vert  &= 
\left\Vert 
\wtilde{u}^{(1)} - e^{-\imath \beta_1 } \frac{u^{(1)}}{\left\Vert u^{(1)} \right\Vert} + 
e^{-\imath \beta_1 } \frac{u^{(1)}}{\left\Vert u^{(1)} \right\Vert} - e^{-\imath 2\pi s_1 /L } \frac{u^{(1)}}{\left\Vert u^{(1)} \right\Vert} \right\Vert \\
&\leq  \left\Vert 
\wtilde{u}^{(1)} - e^{-\imath \beta_1 } \frac{u^{(1)}}{\left\Vert u^{(1)} \right\Vert} \right\Vert + \left\Vert 
e^{-\imath \beta_1 } \frac{u^{(1)}}{\left\Vert u^{(1)} \right\Vert} - e^{-\imath 2\pi s_1 /L } \frac{u^{(1)}}{\left\Vert u^{(1)} \right\Vert} \right\Vert \\
&=\left\Vert 
\wtilde{u}^{(1)} - \alpha_1 \frac{u^{(1)}}{\left\Vert u^{(1)} \right\Vert} \right\Vert  + |e^{-\imath \beta_1 } - e^{-\imath 2\pi s_1 /L }|.
\end{array}
\end{equation}
Since $ |e^{ix}-1|=2|\sin\frac{x}{2}|\leq |x| $ and from~\eqref{eq:eps_vec_bound}, we have
\begin{multline}\label{eq:u_tilde_whole_shift_bound_half2}
\left|e^{-\imath \beta_1 } - e^{-\imath 2\pi s_1 /L }\right| = \left|e^{-\imath \beta_1  + \imath 2\pi s_1 /L} - 1\right| = \left|e^{-\imath \frac{1}{L}\sum_{k=0}^{L-1}\eps[k]} - 1\right| \\ 
\leq \left|\frac{1}{L}\sum_{k=0}^{L-1}\eps[k]\right|
\leq \left\|\eps\right\| \leq \frac{3}{\delta_1^2}\left\| \wtilde{u}^{(1)} - \alpha_1 \frac{u^{(1)}}{\left\Vert u^{(1)} \right\Vert} \right\|.
\end{multline}
Since $ \delta_1 \leq 1 $, from~\eqref{eq:u_tilde_whole_shift_bound1} and \eqref{eq:u_tilde_whole_shift_bound_half2} we have that 
\begin{equation}\label{u_tilde_bound_small_err}
\begin{array}{ll}
\left\Vert \wtilde{u}^{(1)} - 
e^{-\imath 2\pi s_1 /L } \frac{u^{(1)}}{\left\Vert u^{(1)} \right\Vert} \right\Vert 
&\leq \left\Vert 
\wtilde{u}^{(1)} - \alpha_1 \frac{u^{(1)}}{\left\Vert u^{(1)} \right\Vert} \right\Vert +\frac{3}{\delta_1^2}\left\| \wtilde{u}^{(1)} - \alpha_1 \frac{u^{(1)}}{\left\Vert u^{(1)} \right\Vert} \right\|\\
&\leq \frac{4}{\delta_1^2}\left\| \wtilde{u}^{(1)} - \alpha_1 \frac{u^{(1)}}{\left\Vert u^{(1)} \right\Vert} \right\|,
\end{array}
\end{equation}
in case $ \left\|\wtilde{u}^{(1)} - \alpha_1 \frac{u^{(1)}}{\left\Vert u^{(1)} \right\Vert}\right\| \leq \delta_1/2 $.

Suppose now that $ \left\|\wtilde{u}^{(1)} - \alpha_1 \frac{u^{(1)}}{\left\Vert u^{(1)} \right\Vert}\right\|> \delta_1/2 $. Since $ \|\wtilde{u}^{(1)}\| = 1 $ and $ \left\|e^{-\imath 2\pi s_1 /L } \frac{u^{(1)}}{\left\Vert u^{(1)} \right\Vert} \right\| = 1 $, we have
\begin{equation}\label{u_tilde_bound_large_err}
\left\Vert \wtilde{u}^{(1)} - 
e^{-\imath 2\pi s_1 /L } \frac{u^{(1)}}{\left\Vert u^{(1)} \right\Vert} \right\Vert \leq 2 = \frac{4}{\delta_1^2} \frac{\delta_1}{2} \delta_1\leq \frac{4}{\delta_1^2} \frac{\delta_1}{2} \leq \frac{4}{\delta_1^2} \left\|\wtilde{u}^{(1)} - \alpha_1 \frac{u^{(1)}}{\left\Vert u^{(1)} \right\Vert}\right\|.
\end{equation}

Combining \eqref{u_tilde_bound_small_err} and  \eqref{u_tilde_bound_large_err} we have,
\begin{equation}\label{eq:u_tilde_bound_comb}
\left\Vert \wtilde{u}^{(1)} - 
e^{-\imath 2\pi s_1 /L } \frac{u^{(1)}}{\left\Vert u^{(1)} \right\Vert} \right\Vert  \leq \frac{4}{\delta_1^2} \left\|\wtilde{u}^{(1)} - \alpha_1 \frac{u^{(1)}}{\left\Vert u^{(1)} \right\Vert}\right\|,
\end{equation}
for all values of $ \left\|\wtilde{u}^{(1)} - \alpha_1 \frac{u^{(1)}}{\left\Vert u^{(1)} \right\Vert}\right\|$.
From \eqref{eq:u_tilde_bound_comb} and \eqref{eq:DK_bound u1} we have 
\begin{equation}\label{eq:DK_bound u2}
\left\Vert \wtilde{u}^{(1)} - 
e^{-\imath 2\pi s_1 /L } \frac{u^{(1)}}{\left\Vert u^{(1)} \right\Vert} \right\Vert \leq \frac{2L \left\Vert\Sigma_u^{(1)} - \wtilde{C}_z^{(1)} \right\Vert}{\lambda^2\delta_1^2\gamma_1}.
\end{equation}
Combining \eqref{eq:DK_bound u2} with \eqref{eq:norm_Cz1tilde_err}, and denoting 
\begin{equation}\label{eq:Ai_def}
A_i = \frac{2L}{\lambda^2\delta_1^2\gamma_1} A'_i,~~~~i=0,2,3,4
\end{equation}
concludes the proof.

\end{proof}

%%%%%%%%%%%%%%%%%%%%%%%%%%%%%%%%%
%%%%%%%%%%%%%%%%%%%%%%%%%%%%%%
%%%%%%%%%%%%%%%

\subsection{Proof of Theorem~\ref{thm: step 1 A4 err}}\label{subsec:proof_step1 _A4 err}

\begin{proof}
	
Recall from \eqref{eq:A4'_def} that 
\begin{equation*}
A'_4 = \frac{1}{N}\sum_{i=1}^N \hat{\zeta}_i \hat{\zeta}_i^* - I_{L\times L} + E^{(0)} + E^{(1)},
\end{equation*}
with
\begin{equation*} 
E^{(m)}[k_1,k_2] = 
\begin{dcases}
0, & k_1 \neq k_2, \\
\frac{1}{N} \sum_{i=1}^N | \hat{\xi}_i[k_1 + m] |^2 - 1, & k_1=k_2,
\end{dcases}
\end{equation*}
for $m \in \left\{ 0,1\right\}$, and so 
\begin{equation} \label{eq:asym sample cov error term}
\|A'_4\| \leq \left\Vert \frac{1}{N}\sum_{i=1}^N \hat{\zeta}_i \hat{\zeta}_i^* - I_{L\times L} \right\Vert + \left\Vert E^{(0)} \right\Vert + \left\Vert E^{(1)} \right\Vert.
\end{equation}

As for the terms $\left\Vert E^{(0)} \right\Vert$ and $\left\Vert E^{(1)} \right\Vert$ in~\eqref{eq:asym sample cov error term}, since $E^{(0)}$ and $E^{(1)}$ are diagonal matrices, we have 
\begin{equation*}
\left\Vert E^{(m)} \right\Vert = \max_{k\in\{ 0,\ldots,L-1\}}\left\vert \frac{1}{N} \sum_{i=1}^N | \hat{\xi}_i[k+m] |^2 - 1 \right\vert,
\end{equation*}
for $m\in\left\{ 0,1\right\}$, and therefore, it is also clear that 
\begin{equation*}
\left\Vert E^{(0)} \right\Vert = \left\Vert E^{(1)} \right\Vert.
\end{equation*} 
Since $ \hat{\xi}_i[k] $ are i.i.d. standard Gaussian random variables, the random variable $ | \hat{\xi}_i[k+m] |^2 $ has chi-squared distribution (which is sub-exponential) with expected value of $ 1 $, and so $ | \hat{\xi}_i[k+m] |^2 -1$ is sub-exponential with zero mean.
Thus, we apply Lemma~\ref{lem:maxMeanNormal}, and get that
\begin{equation} \label{eq:E01_Bound}
\operatorname{Pr}\left\{\left\Vert E^{(0)} \right\Vert \leq \sqrt{\frac{L}{N}} \right\} \geq 1 - C^{'} L e^{-c^{'} \sqrt{L}},
\end{equation}
for some constants $C^{'},c^{'}>0$. Since $ \left\Vert E^{(0)} \right\Vert = \left\Vert E^{(1)} \right\Vert $ we have that
\begin{equation} \label{eq:E0E1_Bound}
\operatorname{Pr}\left\{\left(\left\Vert E^{(0)} \right\Vert \leq \sqrt{\frac{L}{N}} \right) \wedge \left( \left\Vert E^{(1)} \right\Vert \leq \sqrt{\frac{L}{N}} \right)\right\} \geq 1 - C^{'} L e^{-c^{'} \sqrt{L}}.
\end{equation}

Next, we consider
\begin{equation}\label{eq:zeta sample cov err norm}
\left\Vert \frac{1}{N}\sum_{i=1}^N \hat{\zeta}_i \hat{\zeta}_i^* - I_{L\times L} \right\Vert
\end{equation}
from~\eqref{eq:asym sample cov error term}. Requiring that \eqref{eq:zeta sample cov err norm} is small is essentially a concentration result for the sample covariance of the random vector $\hat{\zeta}$.
We mention that since each element of $\hat{\zeta}$ is the product of two Gaussian random variables (see \eqref{eq:zeta_def}), the elements of $\hat{\zeta}$ admit a sub-exponential distribution (see Lemma 2.7.7 in~\cite{vershynin2010introduction}).

We start by characterizing the real and imaginary parts of $\hat{\zeta}$ from~\eqref{eq:zeta_def},
\begin{align}
\hat{\zeta}_R[k]\triangleq\operatorname{Re}\left\{\hat{\zeta}[k]\right\} &= \hat{\xi}_R[k]\hat{\xi}_R[k+1] + \hat{\xi}_I[k]\hat{\xi}_I[k+1], \label{eq:eta real part} \\
\hat{\zeta}_I[k]\triangleq \operatorname{Im}\left\{\hat{\zeta}[k]\right\} &= \hat{\xi}_I[k]\hat{\xi}_R[k+1] - \hat{\xi}_R[k]\hat{\xi}_I[k+1], \label{eq:eta im part}
\end{align}
where $\hat{\xi}_R = \operatorname{Re}\left\{\hat{\xi}\right\}$, $\hat{\xi}_I = \operatorname{Im}\left\{\hat{\xi}\right\}$. 
Let us now define the random vector $\wtilde{\zeta}\in\mathbb{R}^{2L}$
\begin{equation} \label{eq:tilde_eta def}
\wtilde{\zeta} \triangleq 
\begin{bmatrix}
\hat{\zeta}_R \\
\hat{\zeta}_I
\end{bmatrix}.
\end{equation}
Since  $\hat{\xi}_R[k]$, $\hat{\xi}_I[k]$, $\hat{\xi}_R[k+1]$, $\hat{\xi}_I[k+1]$ are mutually independent Gaussian random variables with mean zero and variance $1/2$ (from \eqref{eq:xi_def}), by~\eqref{eq:eta real part} and~\eqref{eq:eta im part} we have that $ \EE (\hat{\zeta}^2_R[k_1])= 1/2,~\EE (\hat{\zeta}^2_I[k_1])= 1/2  $, $ \EE (\hat{\zeta}_R[k_1]\hat{\zeta}_R[k_2])= 0,~\EE (\hat{\zeta}_I[k_1]\hat{\zeta}_I[k_2])= 0$, and $ \EE (\hat{\zeta}_R[k_1]\hat{\zeta}_I[k_2])= 0 $ for $ k_1 \neq k_2 $. Additionally,
\begin{equation*}
\EE\left[ \wtilde{\zeta} \right] =0.
\end{equation*}
Note that the covariance matrix of $\sqrt{2}\wtilde{\zeta}$  is the identity matrix, i.e.
\begin{equation}\label{eq:Ezeta tilde}
\mathbb{E}\left[ \wtilde{\zeta} \wtilde{\zeta}^T\right] =\frac{1}{2} I_{2L},
\end{equation}
where $I_{2L}$ is the $2L\times 2L$ identity matrix.

Since $ \wtilde{\zeta} $ satisfies the requirements of Definition~\ref{def:zeta_general} and \eqref{eq:Ezeta tilde}, we apply Lemma~\ref{lem:tilde_eta sample covariance concentration} and get
\begin{equation}\label{eq:tildeZeta_sampleCov_from_THM}
\Pr\left\{\left\Vert \frac{1}{N} \sum_{i=1}^N \wtilde{\zeta}_i \wtilde{\zeta}_i^T - \frac{1}{2}\cdot I_{2L} \right\Vert > C_4 \sqrt{\frac{2L}{N}}\right\} \leq 2e^{-c_3 \sqrt{2L}} + C_2 N e^{-c_4 N^{1/4}},
\end{equation} 
or, equivalently,
\begin{equation}\label{eq:tildeZeta_sampleCov_from_THM_exp}
\left\Vert \frac{1}{N} \sum_{i=1}^N \wtilde{\zeta}_i \wtilde{\zeta}_i^T - \frac{1}{2}\cdot I_{2L} \right\Vert \leq C_4 \sqrt{\frac{2L}{N}}
\end{equation}
holds with probability at least
\begin{equation}\label{eq:tildeZeta_sampleCov_from_THM_prob}
1 - 2e^{-c_3 \sqrt{2L}} - C_2 N e^{-c_4 N^{1/4}}.
\end{equation}

Next, we adapt this result to the complex valued vector  $\hat{\zeta}$ of \eqref{eq:zeta_def}. 

From \eqref{eq:tilde_eta def} we have
\begin{equation*}%\label{eq:tildeZeta_sampleCov_from_THM}
\frac{1}{N} \sum_{i=1}^N \wtilde{\zeta}_i \wtilde{\zeta}_i^T - \frac{1}{2}I_{2L}= \left[\begin{matrix}
\frac{1}{N} \sum_{i=1}^N \hat{\zeta}_{Ri} \hat{\zeta}_{Ri}^T - \frac{1}{2}I_{L} &\frac{1}{N} \sum_{i=1}^N\hat{\zeta}_{Ri} \hat{\zeta}_{Ii}^T \\ 
\frac{1}{N} \sum_{i=1}^N \hat{\zeta}_{Ii} \hat{\zeta}_{Ri}^T & \frac{1}{N} \sum_{i=1}^N \hat{\zeta}_{Ii}  \hat{\zeta}_{Ii}^T - \frac{1}{2}I_{L}
\end{matrix}
\right].
\end{equation*} 
Note that, if 
\begin{equation}\label{eq:norm zeta_tilde zeta_tilde* small}
\left\|\left[\begin{matrix}
\frac{1}{N} \sum_{i=1}^N \hat{\zeta}_{Ri} \hat{\zeta}_{Ri}^T - \frac{1}{2}I_{L} &\frac{1}{N} \sum_{i=1}^N\hat{\zeta}_{Ri} \hat{\zeta}_{Ii}^T \\ 
\frac{1}{N} \sum_{i=1}^N \hat{\zeta}_{Ii} \hat{\zeta}_{Ri}^T & \frac{1}{N} \sum_{i=1}^N \hat{\zeta}_{Ii}  \hat{\zeta}_{Ii}^T - \frac{1}{2}I_{L}
\end{matrix}
\right]\right\| \leq   C_4 \sqrt{\frac{2L}{N}}
\end{equation}
then,
\begin{equation}\label{eq:norm zeta_tilde zeta_tilde* small parts}
\begin{array}{ll}
\left\|\frac{1}{N} \sum_{i=1}^N \hat{\zeta}_{Ri} \hat{\zeta}_{Ri}^T - \frac{1}{2}I_{L}\right\| &\leq   C_4 \sqrt{\frac{2L}{N}},\\ 
\left\|\frac{1}{N} \sum_{i=1}^N\hat{\zeta}_{Ri} \hat{\zeta}_{Ii}^T \right\| &\leq   C_4 \sqrt{\frac{2L}{N}},\\ 
\left\|\frac{1}{N} \sum_{i=1}^N \hat{\zeta}_{Ii} \hat{\zeta}_{Ri}^T \right\| &\leq   C_4 \sqrt{\frac{2L}{N}},\\
\left\|\frac{1}{N} \sum_{i=1}^N \hat{\zeta}_{Ii}  \hat{\zeta}_{Ii}^T - \frac{1}{2}I_{L}\right\| &\leq   C_4 \sqrt{\frac{2L}{N}}.

\end{array}
\end{equation}
Since 
\begin{equation}\label{eq}
\frac{1}{N} \sum_{i=1}^N \hat{\zeta}_i \hat{\zeta}_i^* - I_{L}= 
\frac{1}{N} \sum_{i=1}^N 
\hat{\zeta}_{Ri} \hat{\zeta}_{Ri}^T + 
\hat{\zeta}_{Ii} \hat{\zeta}_{Ii}^T  +
i \left(
\hat{\zeta}_{Ii} \hat{\zeta}_{Ri}^T  -
\hat{\zeta}_{Ii}  \hat{\zeta}_{Ii}^T 
\right)
- I_{L},
\end{equation}
it follows from \eqref{eq:tildeZeta_sampleCov_from_THM_exp}  ,\eqref{eq:tildeZeta_sampleCov_from_THM_prob} and \eqref{eq:norm zeta_tilde zeta_tilde* small parts} that, with probability at least
\begin{equation*}
1 - 2e^{-c_3 \sqrt{2L}} - C_2 N e^{-c_4 N^{1/4}},
\end{equation*}
it holds that
\begin{equation}\label{eq:zetahat_cov_bound}
\begin{array}{lll}
\left\|
\frac{1}{N} \sum_{i=1}^N \hat{\zeta}_i \hat{\zeta}_i^* - I_{L}
\right\| &\leq& 
\left\|\frac{1}{N} \sum_{i=1}^N 
\hat{\zeta}_{Ri} \hat{\zeta}_{Ri}^T + 
\hat{\zeta}_{Ii} \hat{\zeta}_{Ii}^T - I_{L}\right\| + 
\left\|\frac{1}{N} \sum_{i=1}^N \left(
\hat{\zeta}_{Ii} \hat{\zeta}_{Ri}^T  -
\hat{\zeta}_{Ii}  \hat{\zeta}_{Ii}^T 
\right)
\right\|\\
&\leq& \left\|\frac{1}{N} \sum_{i=1}^N \hat{\zeta}_{Ri} \hat{\zeta}_{Ri}^T - \frac{1}{2}I_{L}\right\| +
\left\|\frac{1}{N} \sum_{i=1}^N \hat{\zeta}_{Ii}  \hat{\zeta}_{Ii}^T - \frac{1}{2}I_{L}\right\| \\
&&+\left\|\frac{1}{N} \sum_{i=1}^N\hat{\zeta}_{Ri} \hat{\zeta}_{Ii}^T \right\| +
\left\|\frac{1}{N} \sum_{i=1}^N \hat{\zeta}_{Ii} \hat{\zeta}_{Ri}^T \right\| \leq\\
&\leq&
4C_4 \sqrt{\frac{2L}{N}}.
\end{array}
\end{equation}
Thus, from \eqref{eq:zetahat_cov_bound}  it follows that
\begin{equation}\label{eq:zetahat_cov_bound_final}
\left\Vert \frac{1}{N} \sum_{i=1}^N {\hat{\zeta}}_i {\hat{\zeta}}_i^* -  I_{L} \right\Vert \leq 4 C_4 \sqrt{\frac{2L}{N}} 
\end{equation}
with probability at least
\begin{equation}\label{eq:zetahat_cov_bound_final_prob}
1 - 2e^{-c_3 \sqrt{2L}} - C_2 N e^{-c_4 N^{1/4}}.
\end{equation}

Plugging \eqref{eq:zetahat_cov_bound_final}, \eqref{eq:zetahat_cov_bound_final_prob} and~\eqref{eq:E0E1_Bound} in~\eqref{eq:asym sample cov error term} gives that
\begin{equation}\label{eq:lim norm bound}
\|A'_4\| \leq 4 C_4 \sqrt{\frac{2L}{N}} + \sqrt{\frac{L}{N}}  + \sqrt{\frac{L}{N}} = (C_5 + 2) \sqrt{\frac{L}{N}} 
\end{equation}
with probability at least 
\begin{equation}
1 - C^{'}Le^{-c^{'} \sqrt{L}} -  2e^{-c_3 \sqrt{2L}} - C_2 N e^{-c_4 N^{1/4}} \geq 1-C_1Le^{-c_1L^{1/2}}-C_2Ne^{-c_2N^{1/4}}, \label{eq:bound prob exact}
\end{equation}
where we defined $C_5 = \sqrt{32}C_4$.

Now, recall from \eqref{eq:Ai_def} that $ A_4 = \frac{2L}{\lambda^2\delta_1^2\gamma_1} A'_4 $, and thus
\begin{equation}\label{key}
\|A_4\| \leq
\frac{C_3}{\gamma_1\delta_1^2} \sqrt{\frac{L^3}{\lambda^4 N}} ,
\end{equation}
for $ C_3 = 2(C_5 + 2) $, which concludes the proof.
\end{proof}

\subsection{Non-Asymptotic error bound for $ \wtilde{u}^{(1)} $} \label{sec:non-asymp_u}
In this section, we would like to get a better understanding of the error 
$$ \left\Vert \wtilde{u}^{(1)} - e^{-\imath 2\pi s_1 /L } \frac{u^{(1)}}{\left\Vert u^{(1)} \right\Vert} \right\Vert,
$$
in the regime of large $ N $ and large $ \sigma $. 
From Lemma~\ref{lem:u_1 sigma_series_bound} we have that 
\begin{equation}\label{eq:utilde_sigma_poly_bound}
\left\Vert \wtilde{u}^{(1)} - e^{-\imath 2\pi s_1 /L } \frac{u^{(1)}}{\left\Vert u^{(1)} \right\Vert} \right\Vert \leq \|A_0\| + \sigma^2 \|A_2\|+ \sigma^3 \|A_3\|+ \sigma^4 \|A_4\|,
\end{equation}
and from Theorem~\ref{thm: step 1 A4 err} we have 
\begin{equation}\label{eq:A4_bound_for non-asymp}
\|A_4\| \leq \frac{C_3}{\gamma_1} \sqrt{\frac{L^3}{\lambda^4 N}},
\end{equation}
with high probability.

Note that  $ A_0[k_1,k_2], A_2[k_1,k_2] $, and $ A_3[k_1,k_2] $ are all averages of $ N $ samples of random variables with zero expected value (see \eqref{eq:A'2_def}, \eqref{eq:A'3_def}, \eqref{eq:A'0_def}, and \eqref{eq:Ai_def}). Thus, from the central limit theorem we have that 
\begin{equation}\label{eq:A0A2A3_k1k2_simN12}
A_0[k_1,k_2], A_2[k_1,k_2], A_3[k_1,k_2] \sim \frac{1}{\sqrt{N}},
\end{equation}
for large $ N $ for all $ k_1 $ and $ k_2 $. Since we are interested in the regime of constant $ L $ and large $ N $, we can extend \eqref{eq:A0A2A3_k1k2_simN12} to bound all the entries of $ A_0, A_2$, and $ A_3$ simultaneously (by the maximum of $ L^2 $ random variables in each matrix), and get that 
\begin{equation}\label{eq:A0A2A3_simN12}
\|A_0\|, \|A_2\|, \|A_3\| \sim \frac{1}{\sqrt{N}}.
\end{equation}

The result in \eqref{eq:A0A2A3_simN12} together with \eqref{eq:A4_bound_for non-asymp} means that all ``parts" of the bound \eqref{eq:utilde_sigma_poly_bound} of $ \left\Vert \wtilde{u}^{(1)} - e^{-\imath 2\pi s_1 /L } \frac{u^{(1)}}{\left\Vert u^{(1)} \right\Vert} \right\Vert$ (namely, $ \|A_0\| , \sigma^2 \|A_2\|, \sigma^3 \|A_3\|$ and $\sigma^4 \|A_4\|$) go to zero at the same rate when $ N\rightarrow \infty $, and thus, in the case of large $ N $ and large $ \sigma $, the dominant part of the bound is $ \sigma^4 \|A_4\| $ (since $ \sigma^4 $ dominates $ 1,~\sigma^2$, and $ \sigma^3 $ for large $ \sigma $). This, together with \eqref{eq:A4_bound_for non-asymp} means that in the regime of large $ \sigma $ and large $ N $,
\begin{equation*}
\left\Vert \wtilde{u}^{(m)} - e^{-\imath 2\pi s_m /L } \frac{u^{(m)}}{\left\Vert u^{(m)} \right\Vert} \right\Vert \propto  \sigma^4 \frac{1}{\gamma_m} \sqrt{\frac{ L^3}{ \lambda^4 N }} = \sqrt{\frac{1}{ N L \cdot\operatorname{SNR}^4}}.
\end{equation*}

\section{Proof of Theorem~\ref{thm:alg 1 consistency}} \label{sec:proof of consistency of alg 1}
In Appendix~\ref{sec:covariance of z_m}, we have shown that (see \eqref{eq:C_z bias formula})
\begin{equation*}
C_z^{(m)} = \mathbb{E}\left[ z^{(m)} \left( z^{(m)} \right)^* \right] = 2\lambda u^{(m)} \left( u^{(m)} \right)^* + \Sigma_\epsilon^{(m)},
\end{equation*}
where $\Sigma^{(m)}_\epsilon$ is a bias term given by
\begin{equation*}
\Sigma^{(m)}_\epsilon[k_1,k_2] = 
\begin{dcases}
0, & k_1 \neq k_2, \\
 \sigma^2 \left( p_x[k_1] + p_x[k_1 + m] \right) + \sigma^4, & k_1=k_2.
\end{dcases}
\end{equation*}
By~\eqref{eq:z sample cov def} and~\eqref{eq:z sample cov bias correction}, $\wtilde{C}_z^{(m)}$ can be expressed as
\begin{equation*}
\wtilde{C}_z^{(m)} = \frac{1}{N} \sum_{i=1}^N z_i^{(m)} \left( z_i^{(m)} \right)^* - \wtilde{\Sigma}_\epsilon^{(m)},
\end{equation*}
where $\wtilde{\Sigma}_\epsilon^{(m)}$ is the estimated bias term given by
\begin{equation*}
\wtilde{\Sigma}^{(m)}_\epsilon[k_1,k_2] = 
\begin{dcases}
0, & k_1 \neq k_2, \\
 \sigma^2 \left( \wtilde{p}_x[k_1] + \wtilde{p}_x[k_1 + m] \right) + \sigma^4, & k_1=k_2.
\end{dcases}
\end{equation*}
Since $\wtilde{p}_x$ of~\eqref{eq:power spectrum estimate} is a consistent estimator for $p_x$, we have that 
\begin{equation*}
\wtilde{\Sigma}^{(m)}_\epsilon[k_1,k_2]  \underset{\text{a.s.,\;} N\rightarrow \infty}{\longrightarrow} \Sigma^{(m)}_\epsilon[k_1,k_2],
\end{equation*}  
and thus, 
\begin{equation*}
\wtilde{C}_z^{(m)} \underset{\text{a.s.,\;} N\rightarrow \infty}{\longrightarrow} 2\lambda u^{(m)} \left( u^{(m)} \right)^*,
\end{equation*}  
for every $m=1,\ldots,L-1$.
Correspondingly, from the Davis-Kahan $\sin{\Theta}$ theorem~\cite{davis1970rotation}, the eigenspace $\operatorname{span}\{\wtilde{u}^{(m)}\} $ of the leading eigenvalue of $\wtilde{C}_z^{(m)}$ converges almost surely to $\operatorname{span}\{u^{(m)}\}$. Thus, because of the updates made in \eqref{eq:umupdate} and \eqref{eq:umupdate2}, we have that 
\begin{equation}
\wtilde{u}^{(m)} \underset{\text{a.s.,\;} N\rightarrow \infty}{\longrightarrow} \alpha_m \frac{u^{(m)}}{\left\Vert u^{(m)} \right\Vert}, \label{eq:u_m est convergence}
\end{equation}
where $\alpha_m$ is unknown with $| \alpha_m |=1$.
Recall that $u^{(m)}[k] = \hat{\theta}[k]\hat{\theta}^*[k+m]$, and
\begin{equation}
\operatorname{arg}\left\{ u^{(m)}[k] \right\} = \operatorname{arg}\left\{ \hat{\theta}[k] \right\} - \operatorname{arg}\left\{ \hat{\theta}[k+m] \right\}.
\end{equation}
Therefore, if $| \hat{\theta}[k] |>0$ for every $k=0,\ldots,L-1$, then $| u^{(m)}[k] |>0$ for every $k=0,\ldots,L-1$, and we have
\begin{equation}\label{eq:argu_converges}
\operatorname{arg}\left\{ \wtilde{u}^{(m)}[k] \right\} \underset{\text{a.s.,\;} N\rightarrow \infty}{\longrightarrow} \operatorname{arg}\left\{ \hat{\theta}[k] \right\} - \operatorname{arg}\left\{ \hat{\theta}[k+m] \right\} + \varphi_m,
\end{equation}
where $\varphi_m = \operatorname{arg}\left\{ \alpha_m \right\}$. 
Lastly, by the formulas for the frequency marching estimator~\eqref{eq:frequency marching theta est formula_amp} and~\eqref{eq:frequency marching theta est formula_phase}, using~\eqref{eq:argu_converges} with $m=1$, we have
\begin{equation}
\begin{array}{lcl}
\wtilde{\theta}[k] &\underset{\text{a.s.,\;} N\rightarrow \infty}{\longrightarrow}& \sqrt{{p_x[k]}/{\lambda}} \cdot \exp\left\{ -\imath \sum_{\ell=0}^{k-1}  \operatorname{arg}\left\{ \hat{\theta}[\ell] \right\} - \operatorname{arg}\left\{ \hat{\theta}[\ell+1] \right\} + \varphi_1 \right\}
\\
& = & \sqrt{{p_x[k]}/{\lambda}} \cdot \exp\left\{ -\imath \left(\operatorname{arg}\left\{ \hat{\theta}[0] \right\} - \operatorname{arg}\left\{ \hat{\theta}[k] \right\} + k\varphi_1 \right)\right\}
\\
& = & \hat{\theta}[k] e^{-\imath \left(\operatorname{arg}\left\{ \hat{\theta}[0] \right\}  + k\varphi_1 \right)},
\end{array}
\end{equation}
where the last equality is due to \eqref{eq:power spectrum relations}.
Therefore 
\begin{equation}
F^{-1} \wtilde{\theta} \underset{\text{a.s.,\;} N\rightarrow \infty}{\longrightarrow} e^{-\imath \operatorname{arg}\left\{ \hat{\theta}[0] \right\}} \mathcal{R}_{s_1} \left\{ \theta \right\},
\end{equation}
where 
\begin{equation}
s_1 = L \frac{\varphi_1}{2\pi}.
\end{equation}
Note that due to the update made in~\eqref{eq:umupdate} we have that $\sum_{k = 0}^{L-1}\operatorname{arg}\left\{ \wtilde{u}^{(m)}[k] \right\} \equiv 0~(\operatorname{mod} 2\pi)$ and from \eqref{eq:argu_converges} we have that $\sum_{k = 0}^{L-1}\operatorname{arg}\left\{ \wtilde{u}^{(m)}[k] \right\} \underset{\text{a.s.,\;} N\rightarrow \infty}{\longrightarrow} L\varphi_m $ for all $ m $. Thus $ L\varphi_1 \equiv 0 ~(\operatorname{mod} 2\pi)$ and thus $s_1$ is indeed an integer.

\section{Proof of Theorem~\ref{thm:alg 1 asym err}} \label{sec:proof of alg 1 asym err}
\subsection{Supporting lemmas}
We begin with Lemma~\ref{lem:bound_step2_FM}, which bounds the error of the ``frequency marching" in Step~\ref{algstep:freqMarch} of Algorithm~\ref{alg:MRFA_formal}. For future reference we denote the ``frequency marching"  procedure as Algorithm~\ref{alg:freqmarch}.
\begin{algorithm}[H]
	\caption{Frequency marching}
	\label{alg:freqmarch}
	\begin{algorithmic}[1]
		\Statex {\bfseries Inputs:} $\wtilde{u}^{(1)}[k] \in \CC^L$.
		\Statex {\bfseries Outputs:} $ \wtilde{\theta}_p$.
		\State $\wtilde{\theta}_p[0]\gets 1$
		\For {$k=0,\ldots,L-2$} 
		\State $\wtilde{\theta}_p[k+1]\gets \wtilde{\theta}_p[k] \cdot \exp\left\{-\imath\cdot \operatorname{arg}\left\{ \wtilde{u}^{(1)}[k]\right\} \right\}$
		\EndFor
		\State\Return $\wtilde{\theta}_p$
	\end{algorithmic}
\end{algorithm}

\begin{lemma}\label{lem:bound_step2_FM}
	Let $ v\in \CC^L $ s.t. $ |v[k]|=1 $ for $ k=0,\ldots,L-1 $, and let $ u \in \CC^L $ s.t. $ u[k] = v[k]v^*[k+1], k=0,\ldots,L-1 $. Let $ \wtilde{u}\in \CC^L $ be an estimate of $ u $, with $ \sum_{k=0}^{L-1}\arg\{\wtilde{u}[k]\} = 0 $. Then, applying Algorithm~\ref{alg:freqmarch} on $ \wtilde{u}  $ returns $ \wtilde{v}\in \CC^L $ s.t.  $ \|v-\alpha \wtilde{v}\| \leq \|u-\wtilde{u}\| \cdot \wtilde{C}' L $ for some $ \alpha \in \CC $ with $ |\alpha| = 1 $ and some constant $ \wtilde{C}' $ independent of $ L $.
\end{lemma}

\begin{proof}
		
	Since $ v,u $ are vectors with entries of norm $ 1 $, we switch to working with angles.
	Denote by $ \bar{v}, w, \wtilde{w} \in \RR^L $ the vectors $ \bar{v}[k] = \arg(v[k]) $, $ w[k] = \arg(u[k]) $ and $ \wtilde{w}[k] = \arg(\wtilde{u}[k])$.

	By Lemma~\ref{lem:from_C_to_args}, we have
	\begin{equation}\label{eq:w_diff le u_diff_2}
		 \|\wtilde{w}-w \| \leq \|\wtilde{u}-u\|.
	\end{equation}
	
	From the definition of $ \bar{v} $ and $ w $, we have that $ w[k] = \bar{v}[k] - \bar{v}[k+1] $. In other words, we can write $ A\bar{v} = w $ where
	
	\begin{equation}\label{eq:A_def}
	A = \left( 
	\begin{array}{ccccc}
	1      & -1    & \hfill & \hfill & \hfill \\
	\hfill & 1     & -1     & \hfill & \hfill \\
	\hfill & \hfill& \ddots & \ddots & \hfill\\
	\hfill & \hfill& \hfill &  1     & -1\\
	-1     & \hfill& \hfill & \hfill & 1
	\end{array} \right).
	\end{equation}  
	
	Note that we can also write $ A = I-\mathcal{R}_1 $, where $ \mathcal{R}_1 $ is the linear operator of a cyclic shift by 1. The eigen-decomposition of $ \mathcal{R}_1 $ is 
	\[
	\mathcal{R}_1 = F\left( 
	\begin{array}{cccc}
	1      &    & \hfill & \hfill\\
	\hfill & e^{2\pi\imath\cdot 1/L}     &     & \hfill \\
	\hfill & \hfill& \ddots & \\
	& \hfill& \hfill & e^{2\pi\imath\cdot(L-1)/L}
	\end{array} \right)F^{-1},
	\]  
	%$ S_1 = F diag(e^{2\pi\cdot 0/L},e^{2\pi\cdot 1/L}\ldots,e^{2\pi\cdot(L-1)/L} ) F^{-1} $
	where $F$ is the $ L\times L $ discrete Fourier transform matrix.
	Thus,
	\[
	A = F\left( 
	\begin{array}{cccc}
	1-1      &    & \hfill & \hfill\\
	\hfill & 1-e^{2\pi\imath\cdot 1/L}     &     & \hfill \\
	\hfill & \hfill& \ddots & \\
	& \hfill& \hfill & 1-e^{2\pi\imath\cdot(L-1)/L}
	\end{array} \right)F^{-1}.
	\]   
	The smallest non-zero singular value of $ A $ is $| 1-e^{2\pi\imath /L}| \geq \sin(\frac{2\pi}{L}) \geq \frac{2\pi}{L}-\frac{(2\pi)^3}{6L^3}$, thus, $ \|A^\dagger\|\leq \wtilde{C}' L $, where $ A^\dagger $ is the Moore–Penrose pseudo-inverse (satisfying that $ x = A^\dagger y $ is a solution of $ y = Ax $), and $ \wtilde{C}' $ is some constant independent of $ L $. Note that the null space of $ A $ ($ \ker A $) is one-dimensional as $ A $ has a single zero singular value. It is also evident from \eqref{eq:A_def} that for any constant vector $ \bar{\alpha} \in \CC^L $ we have that $ A\bar{\alpha} = 0 $, and thus the null space of $ A $ consists solely of constant vectors. Additionally, also from \eqref{eq:A_def}, if $ y=Ax $ then $ \sum_{k=0}^{L-1}y[k] = 0 $. Since the constraint $ \sum_{k=0}^{L-1}y[k] = 0 $ defines a linear subspace of dimension $ L-1 $, and since $ \dim (\Ima(A)) = L-1 $, any vector $ y $ with $ \sum_{k=0}^{L-1}y[k] = 0 $ is in the image of $ A $.
	
	The ``frequency marching" procedure described in Algorithm~\ref{alg:freqmarch} (and in Step~\ref{algstep:freqMarch} of Algorithm~\ref{alg:MRFA_formal}), when written in the form of \eqref{eq:frequency marching}, solves 
	\begin{equation}\label{eq:FM_eqsystem}
	\operatorname{arg}\left\{ u^{(1)}[k] \right\} = \operatorname{arg}\left\{ \hat{\theta}[k] \right\} - \operatorname{arg}\left\{ \hat{\theta}[k+1] \right\}
	\end{equation}
	for $ \hat{\theta}$, for all $ k $. Rewriting \eqref{eq:FM_eqsystem} in matrix form, Algorithm~\ref{alg:freqmarch} solves $\operatorname{arg}\{u\} = A \operatorname{arg}\{\hat{\theta}\}$. Since $ \sum_{k=0}^{L-1}\wtilde{w}[k] = 0 $ we have that $ \wtilde{w}\in \Ima A $, and therefore, there exists a constant vector $\bar{\alpha} \in \CC^L$ such that $\wtilde{\bar{v}} + \bar{\alpha} =  A^\dagger \wtilde{w} $, where $ \wtilde{\bar{v}}[k] = \arg\{\wtilde{v}[k]\} $ (since constant vectors are the null space of $ A $ and $ \wtilde{\bar{v}} $ is a solution to $ \wtilde{w} = A\wtilde{\bar{v}} $). Additionally, from the definition of $ w $ ($w=A\bar{v}$) there is a constant vector $\bar{\beta} \in \CC^L$ such that $\bar{v} + \bar{\beta} =  A^\dagger w $.
	Now, we have, 
	\begin{equation}\label{eq:angle_u bound}
	 \|\bar{v}+\bar{\beta}-(\wtilde{\bar{v}} + \bar{\alpha})\| = \|A^\dagger(w-\wtilde{w})\| \leq \|A^\dagger\| \|w-\wtilde{w}\| \leq \|u-\wtilde{u}\| \wtilde{C}' L,
	\end{equation}
	where the last inequality is due to \eqref{eq:w_diff le u_diff_2}. Since $ |e^{ix}-1|=2|\sin\frac{x}{2}|\leq |x| $, denoting $ \alpha = \exp\{\imath(\bar{\beta} - \bar{\alpha})\} $, we have from \eqref{eq:angle_u bound} that
	\begin{multline*}
	\|\alpha v - \wtilde{v}\| = \|\exp\{\imath (\bar{v} + \bar{\beta} - \bar{\alpha})\}-\exp\{\imath\wtilde{\bar{v}}\}\| =\\  \|\exp\{\imath (\bar{v} + \bar{\beta} - \wtilde{\bar{v}} - \bar{\alpha})\}-1\| \leq \|\bar{v} + \bar{\beta} - (\wtilde{\bar{v}} + \bar{\alpha})\| \leq \|u-\wtilde{u}\| \wtilde{C}' L.
	\end{multline*}
\end{proof}

Next, in Lemma~\ref{lem:phase recovery}, we extend the result of Lemma~\ref{lem:bound_step2_FM} by removing the requirement that $ |v[k]| = 1 $ for $ k=0,\ldots,L-1 $.
\begin{lemma}\label{lem:phase recovery}
	Let $ v\in \CC^L $, and let $v_p\in\CC^L$ be the vector of phases of $ v $. Let $ u \in \CC^L $ such that $ u[k] = v[k]v^*[k+1], k=0,\ldots,L-1 $. Assume that $ \|u\| = 1 $. Denote $ \delta_1 = \min_{k\in\{0,\ldots,L-1\}}|u[k]| $, and assume that $\delta_1>0 $. Let $ \wtilde{u}\in \CC^L $ with $ \|\wtilde{u}\|=1 $ and $ \sum_{k=0}^{L-1}\arg\{\wtilde{u}[k]\} = 0 $ be an estimate of $ u $. Then, applying Algorithm~\ref{alg:freqmarch} on  $ \wtilde{u}  $ returns $ \wtilde{v}\in \CC^L $ s.t.  $ \|v_p-\alpha \wtilde{v}\| \leq \frac{\|u-\wtilde{u}\|}{\delta_1^2} \cdot \wtilde{C} L $ for some $ \alpha \in \CC $ with $ |\alpha| = 1 $ and some constant $ \wtilde{C} $ independent of $ L $.
\end{lemma}

\begin{proof}
	 Denote $ u_p, \wtilde{u}_p \in \CC^L $ such that $ u_p[k] = \frac{u[k]}{|u[k]|} $ and $ \wtilde{u}_p[k] = \frac{\wtilde{u}[k]}{|\wtilde{u}[k]|} $.
	 Assume that $ \|\wtilde{u} - u\|\leq \delta_1/2 $ (we will later consider the alternative). Then, by Lemma~\ref{lem:norm1_to_element_norm1} we have, 
	\begin{equation*}
	\|\wtilde{u}_p - u_p\| \leq \frac{3\| (\wtilde{u} - u)\|}{\delta_1^2}.
	\end{equation*}
	
	Since $ u_p[k] = v_p[k]v_p^*[k+1], k=0,\ldots,L-1 $, Lemma~\ref{lem:bound_step2_FM} guarantees that applying Algorithm~\ref{alg:freqmarch} on $ \wtilde{u}_p $ will result in $ \wtilde{v} $ such that 
	
\begin{equation}\label{eq:v_p-vtilde small u diff}
	\|v_p-\alpha \wtilde{v}\| \leq \|u_p -\wtilde{u}_p\| \cdot \wtilde{C}' L  \leq 3\frac{\| (\wtilde{u} - u)\|}{\delta_1^2} \wtilde{C}' L,	
\end{equation}
	for some $ \alpha \in \CC $ with $ |\alpha|=1 $. Thus, We showed that applying the frequency marching procedure on $ \wtilde{u}_p $ results in $ \wtilde{v} $ as required. Since the frequency marching procedure dose not depend on the magnitude of the input vector, applying it on $ \wtilde{u} $ results in the same vector as applying it on  $ \wtilde{u}_p $.
	
	Until now, we have shown that $ 	\|v_p-\alpha \wtilde{v}\| \leq 3\frac{\| (\wtilde{u} - u)\|}{\delta_1^2} \wtilde{C}' L,$ in the case when  $ \|\wtilde{u} - u\|\leq \delta_1/2 $. Assume now that $ \|\wtilde{u} - u\| > \delta_1/2 $. 	Since $ |v_p[k]| = |\wtilde{v}[k]| = 1 $ for all $ k $ ($ v_p $ is a vector of phases, and since $ \wtilde{v} $ is the output of Algorithm~\ref{alg:freqmarch} it is also a vector of phases), we have that $ (v_p[k]-\alpha \wtilde{v}[k])^2\leq 4 $, and thus
	\begin{equation}\label{eq:v_p-vtilde large u diff}
	\|v_p-\alpha \wtilde{v}\| =\sqrt{\sum_{k=0}^{L-1} (v_p[k]-\alpha \wtilde{v}[k])^2} \leq 2\sqrt{L}\leq 2L \leq 4\frac{\delta_1/2}{\delta_1^2}L \leq 4\frac{\|\wtilde{u} - u\|}{\delta_1^2}L,
	\end{equation}
	where the third inequality holds since $ \delta_1 < 1 $ and the last inequality is due to the assumption  $ \|\wtilde{u} - u\|> \delta_1/2 $.
	Finally, denoting $ \wtilde{C} = \max(4,3\wtilde{C}') $, we have from \eqref{eq:v_p-vtilde small u diff} and \eqref{eq:v_p-vtilde large u diff} that $ 	\|v_p-\alpha \wtilde{v}\| \leq \frac{\| (\wtilde{u} - u)\|}{\delta_1^2} \wtilde{C} L $.

\end{proof}

The following Lemma shows that if the magnitudes of a signal are estimated accurately, and the phases are estimated accurately, than the signal is estimated accurately.

\begin{lemma}\label{lem:bound_on combinning_powerErr_phaseErr}
	Let $ \theta\in \CC^L $ with $ \|\theta\| = 1 $. 
	Denote by $ \theta_m \in \RR^L $ the magnitudes of $ \theta $, and by $ \theta_p $ the phases of $ \theta $ (i.e. = $\theta_p \odot {\theta_m}$). Suppose that $ \wtilde{\theta}_m $ and $ \wtilde{\theta}_p $ are approximations of $ \theta_m $ and $ \theta_p $ such that $ \|\wtilde{\theta}_m-\theta_m\| \leq \eps_1  $ and $ \|\wtilde{\theta}_p-\theta_p\| \leq \eps_2  $. Then, for $ \wtilde{\theta} = \wtilde{\theta}_p \odot {\wtilde{\theta}_m}$, it holds that  $ \| \theta - \wtilde{\theta}\| \leq \eps_1 + \eps_2 $.
\end{lemma}

\begin{proof}
	\begin{align*}
	\| \wtilde{\theta} - \theta\| & = \|\wtilde{\theta}_p\odot{\wtilde{\theta}_m} - \theta_p\odot{\theta_m}\| \\ & = 
	\|\wtilde{\theta}_p\odot{\wtilde{\theta}_m} -\wtilde{\theta}_p\odot{\theta_m}+\wtilde{\theta}_p\odot{\theta_m}- \theta_p\odot{\theta_m}\| \\ 
	& \leq 
	\|\wtilde{\theta}_p\odot{\wtilde{\theta}_m} -\wtilde{\theta}_p\odot{\theta_m}\| + \|\wtilde{\theta}_p\odot{\theta_m}- \theta_p\odot{\theta_m}\| \\& = 
	\|\wtilde{\theta}_p\odot[{\wtilde{\theta}_m} -{\theta_m}]\| + \|[\wtilde{\theta}_p- \theta_p]\odot{\theta_m}\|.
	\end{align*}
	Since $ |\wtilde{\theta}_p[i]| = 1 $ and $ |{\theta_m}[i]| \leq 1 $ for $ i=0,\ldots,L-1 $, we get from \eqref{eq:odot norm bound} that
	\begin{equation}
	\| \wtilde{\theta} - \theta\| \leq 	\|{\wtilde{\theta}_m} -{\theta_m}\| + \|\wtilde{\theta}_p- \theta_p\| \leq \eps_1 + \eps_2.
	\end{equation}
\end{proof}

\subsection{Proof of Theorem~\ref{thm:alg 1 asym err}}
\begin{proof}
The outline of the proof is as follows. We show that the frequency marching procedure results in a ``good" estimate of the phases of $ \hat{\theta} $ of \eqref{eq:FT_params_def}. Then we show that we have a ``good" estimate of the magnitudes of $ \hat{\theta} $   as well. Finally, we use Lemma~\ref{lem:bound_on combinning_powerErr_phaseErr} to combine the two and conclude the proof.
 
Denote 
\begin{equation}\label{eq:s_0 def}
s_0 = \argmin_{s \in \{0,\ldots, L-1\}}  \left\| \wtilde{u}^{(1)} - e^{-\imath 2\pi s /L } \frac{u^{(1)}}{\|u^{(1)}\|}\right\|.
\end{equation}
Denote $ u_{s_0}^{(1)} = e^{-\imath 2\pi s_0 /L } \frac{u^{(1)}}{\|u^{(1)}\|} $, and denote by $ u_{ps_0}^{(1)} $ the vector of phases of $ u_{s_0}^{(1)} $,that is, 
$$ 
u_{ps_0}^{(1)}[k] = \frac{u_{s_0}^{(1)}[k]}{|u_{s_0}^{(1)}[k]|} = e^{-\imath 2\pi s_0 /L }u_p^{(1)}[k],
$$
where $ u_p^{(1)} $ is the phases part of $ u^{(1)} $ ($ u_p^{(1)} = u^{(1)}[k]/|u^{(1)}[k]| $). Denote $ \hat{\theta}_{ps_0}[k] =e^{\imath 2\pi ks_0 /L } \hat{\theta}_p[k] $  where $ \hat{\theta}_p $ is the vector of phases of $ \hat{\theta} $ of~\eqref{eq:FT_params_def} ($ \hat{\theta}_{p}[k] = \hat{\theta}[k]/|\hat{\theta}[k]| $). Note that
\begin{multline*}
u_{ps_0}^{(1)}[k] = e^{-\imath 2\pi s_0 /L }u_p^{(1)}[k] = e^{-\imath 2\pi s_0 /L }\hat{\theta}_p[k] \hat{\theta}_p^*[k+1] \\= e^{\imath 2\pi s_0 k /L } \hat{\theta}_p[k] e^{-\imath 2\pi s_0 (k+1) /L } \hat{\theta}_p^*[k+1] =\hat{\theta}_{ps_0}[k]\hat{\theta}^*_{ps_0}[k+1].
\end{multline*}
Since $ \|\wtilde{u}^{(1)}\| = 1$ and from~\eqref{eq:umupdate}, $ \sum_{k=0}^{L-1}\arg\{\wtilde{u}^{(1)}[k]\} = 0 $, we have that $ \wtilde{u}^{(1)}$ and $u_{s_0}^{(1)}$ satisfy the requirements of Lemma~\ref{lem:phase recovery}  (as $ \wtilde{u} $ and $ u $ correspondingly). Therefore, it follows that applying Algorithm~\ref{alg:freqmarch} on $ \wtilde{u}^{(1)}  $ will result in $ \wtilde{\theta}_p $ such that 
\begin{equation}\label{eq:theta_hat_phase_approx}
\left\| \alpha \wtilde{\theta}_p - \hat{\theta}_{ps_0}\right\| \leq \wtilde{C} \frac{L}{\delta_1^2} \left\| \wtilde{u}^{(1)} - u_{s_0}^{(1)}\right\|,
\end{equation}
for some $ \alpha \in \CC $ with $ |\alpha| = 1 $.

Next, we bound the error in estimating $ \hat{\theta}_m $, where $ \hat{\theta}_m[k] = |\hat{\theta}[k]| $. From \eqref{eq:power spectrum est error} follows that, for large $N$, 
\begin{equation}\label{eq:power_diff}
\left|\wtilde{p}_x[k] - p_x[k]\right| \leq
C'_2\frac{\sigma^2}{\sqrt{N}},
\end{equation}
for some constant $ C'_2 $. Since for any $ v\geq -1 $ we have 
\begin{equation}\label{key}
|1-\sqrt{1+v}|\leq |v|,
\end{equation}
we have for any positive $ v $ and $ \wtilde{v} $
\begin{equation}\label{eq:bound_sqrt_diff_general}
\left|\sqrt{v} - \sqrt{\wtilde{v}}\right| = \left|\sqrt{v} - \sqrt{v}\sqrt{1+\frac{\wtilde{v} - v}{v}}\right| = \sqrt{v} \left|1 - \sqrt{1+\frac{\wtilde{v} - v}{v}}\right| \leq \sqrt{v} \left|\frac{\wtilde{v} - v}{v}\right| = \frac{|\wtilde{v} - v|}{\sqrt{v}}.
\end{equation}
From \eqref{eq:power_diff} and \eqref{eq:bound_sqrt_diff_general}, we have
\begin{equation}\label{eq:bound_sqrt_diff1}
\left| \sqrt{\wtilde{p}_x[k]} - \sqrt{{p}_x[k]}\right| \leq C'_2\frac{\sigma^2}{\sqrt{p_x[k]}\sqrt{N}}.
\end{equation}
Note that from the definition of $ \delta_1 $ and from~\eqref{eq:um and zm def} we have that $ |\hat{\theta}[k]| \geq \delta_1 $, which together with~\eqref{eq:power spectrum relations} gives $ \sqrt{p_x[k]} \geq \delta_1 \sqrt{\lambda}$.
Thus we have from ~\eqref{eq:bound_sqrt_diff1},
\begin{equation}\label{eq:power_approx}
\sqrt{{p}_x[k]} - C'_2\frac{\sigma^2}{\delta_1 \sqrt{\lambda}\sqrt{N}} \leq \sqrt{\wtilde{p}_x[k]} \leq \sqrt{{p}_x[k]} + C'_2\frac{\sigma^2}{\delta_1 \sqrt{\lambda}\sqrt{N}}.
\end{equation}

From \eqref{eq:lambda_est_error}  we have that, for large $N$, 
\begin{equation}\label{eq:lambda_approx}
\left|\wtilde{\lambda} - \lambda\right| \leq
C'''_2\frac{L\sigma^2}{\sqrt{N}},
\end{equation}
for some constant $ C'''_2 $.
From Taylor expansion of $ \frac{1}{1+\eps} $ around $ 0 $ we have that for $\left|\eps\right|\leq 1/2 $ it holds that 
$$
\left|\frac{1}{\lambda+\lambda \eps} - \frac{1}{\lambda}\right|  = \frac{1}{\lambda} \left|\frac{1}{1+\eps} - 1\right| \leq \frac{c\eps}{\lambda}
$$
or,
$$
\frac{1}{\lambda} -\eps\frac{c}{\lambda} \leq \frac{1}{\lambda+\lambda\eps}\leq 
\frac{1}{\lambda} + \eps\frac{c}{\lambda},
$$
which, together with \eqref{eq:lambda_approx}, assuming $ C'''_2\frac{L\sigma^2}{\lambda\sqrt{N}} \leq \frac{1}{2} $, gives
\begin{equation}\label{eq:one_over_lambda_approx}
\frac{1}{\lambda} - C''_2\frac{L\sigma^2}{\lambda^2\sqrt{N}} \leq \frac{1}{\wtilde{\lambda}} \leq
\frac{1}{\lambda} + C''_2\frac{L\sigma^2}{\lambda^2\sqrt{N}},
\end{equation}
where $ C''_2 = cC'''_2 $. Similarly to the derivation of \eqref{eq:power_approx}, using \eqref{eq:bound_sqrt_diff_general}, we have 
\begin{equation}\label{eq:lambda_approx_sqrt}
\frac{1}{\sqrt{\lambda}} - C''_2\frac{L\sigma^2}{\lambda^{3/2}\sqrt{N}} \leq \frac{1}{\sqrt{\wtilde{\lambda}}} \leq
\frac{1}{\sqrt{\lambda}} + C''_2\frac{L\sigma^2}{\lambda^{3/2}\sqrt{N}}.
\end{equation}
Thus, combining \eqref{eq:power_approx} with  \eqref{eq:lambda_approx_sqrt}, we have
\begin{multline}\label{eq:amplitude_bound1}
\frac{\sqrt{p_x[k]}}{\sqrt{\lambda}} - \left(
C''_2\frac{L\sigma^2\sqrt{{p}_x[k]}}{\lambda^{2/3} \sqrt{N}} + 
C'_2 \frac{\sigma^2}{\delta_1 \lambda  \sqrt{N}} + 
C'_2 C''_2 \frac{L\sigma^4}{\delta_1\lambda ^2N} \right)\\ \leq
\frac{\sqrt{\wtilde{p}_x[k]}}{\sqrt{\wtilde{\lambda}}} \leq \frac{\sqrt{p_x[k]}}{\sqrt{\lambda}} + \left(
C''_2\frac{L\sigma^2\sqrt{{p}_x[k]}}{\lambda^{2/3} \sqrt{N}} + 
C'_2 \frac{\sigma^2}{\delta_1 \lambda  \sqrt{N}} + 
C'_2 C''_2 \frac{L\sigma^4}{\delta_1\lambda ^2N} \right),
\end{multline}
or,
\begin{equation*}
\left|\frac{\sqrt{\wtilde{p}_x[k]}}{\sqrt{\wtilde{\lambda}}} - \frac{\sqrt{p_x[k]}}{\sqrt{\lambda}}  \right|\leq \left(
C''_2\frac{L\sigma^2\sqrt{{p}_x[k]}}{\lambda^{2/3} \sqrt{N}} + 
C'_2 \frac{\sigma^2}{\delta_1 \lambda  \sqrt{N}} + 
C'_2 C''_2 \frac{L\sigma^4}{\delta_1\lambda ^2N} \right).
\end{equation*}

Recall From \eqref{eq:power spectrum relations} that $ \frac{\sqrt{{p}_x[k]}}{\lambda} = \hat{\theta}_m[k] \leq \|\hat{\theta}\| = 1 $. Thus we have 
\begin{equation}\label{eq:theta_tilde_m_bound}
\left|\frac{\sqrt{\wtilde{p}_x[k]}}{\sqrt{\wtilde{\lambda}}} - \hat{\theta}_m[k]  \right|\leq 
\left(
C''_2\frac{L\sigma^2}{\sqrt{\lambda N}} + 
C'_2 \frac{\sigma^2}{\delta_1 \lambda \sqrt{N}} + 
C'_2 C''_2 \frac{L\sigma^4}{\delta_1\lambda ^2N} 
\right),
\end{equation}
assuming $ C''_2\frac{L\sigma^2}{\lambda\sqrt{N}} \leq \frac{1}{2} $.

In case $ C''_2\frac{L\sigma^2}{\lambda\sqrt{N}} > \frac{1}{2} $, we note that from \eqref{eq:power spectrum estimate} we have that $ \frac{\sqrt{\wtilde{p}_x[k]}}{\sqrt{\wtilde{\lambda}}} \leq 1 $ and since $ \hat{\theta}_m[k] \leq 1 $ we have,
\begin{equation}\label{eq:theta_tilde_m_bound2}
\left|\frac{\sqrt{\wtilde{p}_x[k]}}{\sqrt{\wtilde{\lambda}}} - \hat{\theta}_m[k]  \right|\leq 2 < 4 C''_2\frac{L\sigma^2}{\lambda\sqrt{N}}.
\end{equation}
From \eqref{eq:theta_tilde_m_bound} and \eqref{eq:theta_tilde_m_bound2} we have that there is a constant $ \dbtilde{C} $ such that,
\begin{equation}\label{eq:amplitude_bound2}
 \left\|\frac{\sqrt{\wtilde{p_x}}}{\sqrt{\wtilde{\lambda}}} - \hat{\theta}_m\right\| \leq \dbtilde{C}\sqrt{L}\left(\frac{L\sigma^2}{\lambda\sqrt{N}} +\frac{L\sigma^2}{\sqrt{\lambda N}} + \frac{\sigma^2}{\delta_1\sqrt{\lambda N}} + \frac{L\sigma^4}{\delta_1\lambda^2N}\right),
\end{equation}
for a large enough $ N $. Thus, by Step~\ref{algstep:algFM_fixPower} of Algorithm~\ref{alg:MRFA_formal} we have 
\begin{equation}\label{eq:theta_hat_power_approx}
\left\|\wtilde{\theta}_m - \hat{\theta}_m\right\| \leq \dbtilde{C}\sqrt{L}\left(\frac{\sigma^2}{\sqrt{\lambda N}}\left(\frac{L}{\sqrt{\lambda}} +L + \frac{1}{\delta_1}\right) + \frac{L\sigma^4}{\delta_1\lambda^2N}\right).
\end{equation}
By \eqref{eq:theta_hat_phase_approx}, \eqref{eq:theta_hat_power_approx} and Lemma~\ref{lem:bound_on combinning_powerErr_phaseErr}, we have
\begin{equation*}
\left\| \alpha\wtilde{\theta}_p \odot \wtilde{\theta}_m - \hat{\theta}_{ps_0} \odot \hat{\theta}_m  \right\| 
\leq 
\dbtilde{C}\sqrt{L}\left(\frac{\sigma^2}{\sqrt{\lambda N}}\left(\frac{L}{\sqrt{\lambda}} +L + \frac{1}{\delta_1}\right) + \frac{L\sigma^4}{\delta_1\lambda^2N}\right)+
\wtilde{C} \frac{L}{\delta_1^2} \left\| \wtilde{u}^{(1)} - u_{s_0}^{(1)}\right\|.
\end{equation*}
Since  $ \hat{\theta}_{ps_0}[k] =e^{\imath 2\pi ks_0 /L } \hat{\theta}_p[k] $, and that the inverse Fourier transform is an orthogonal transformation, we have from Step~\ref{algstep:final theta_tilde} on Algorithm~\ref{alg:MRFA_formal},
\begin{equation}\label{eq:thetatilde_Rtheta_bound}
\left\| \alpha \wtilde{\theta} - \mathcal{R}_{s_0}\{\theta\}\right\| \leq \dbtilde{C}\sqrt{L}\left(\frac{\sigma^2}{\sqrt{\lambda N}}\left(\frac{L}{\sqrt{\lambda}} +L + \frac{1}{\delta_1}\right) + \frac{L\sigma^4}{\delta_1\lambda^2N}\right)+
\wtilde{C} \frac{L}{\delta_1^2} \left\| \wtilde{u}^{(1)} - u_{s_0}^{(1)}\right\|.
\end{equation}
Denoting $ s = -s_0 $, and noting that applying $ \mathcal{R}_s $ on a signal, does not not change its norm, we have that 
\begin{equation}\label{eq:change_R_side}
\left\| \alpha \wtilde{\theta} - \mathcal{R}_{s_0}\{\theta\}\right\|   = 
\left\| \mathcal{R}_{s}\left\{\alpha \wtilde{\theta} - \mathcal{R}_{s_0}\{\theta\}\right\}\right\|  = \left\| \alpha \mathcal{R}_{s}\{\wtilde{\theta}\} - \theta\right\|.
\end{equation}
From \eqref{eq:thetatilde_Rtheta_bound}, \eqref{eq:change_R_side}, and the fact that  $ u_{s_0}^{(1)} = e^{-\imath 2\pi s_0 /L } \frac{u^{(1)}}{\|u^{(1)}\|} $, we have

\begin{align}\label{eq:lim_theta_thetawitlde_1}
\left\Vert \alpha  \mathcal{R}_{{s}} \left\{\theta\right\}  - \wtilde{\theta}  \right\Vert & \leq  
\dbtilde{C}\sqrt{L}\left(\frac{\sigma^2}{\sqrt{\lambda N}}\left(\frac{L}{\sqrt{\lambda}} +L + \frac{1}{\delta_1}\right) + \frac{L\sigma^4}{\delta_1\lambda^2N}\right) + \wtilde{C} \frac{L}{\delta_1^2} \left\| \wtilde{u}^{(1)} - u_{s_0}^{(1)}\right\|\nonumber\\& =  
\dbtilde{C}\sqrt{L}\left(\frac{\sigma^2}{\sqrt{\lambda N}}\left(\frac{L}{\sqrt{\lambda}} +L + \frac{1}{\delta_1}\right) + \frac{L\sigma^4}{\delta_1\lambda^2N}\right) + \wtilde{C} \frac{L}{\delta_1^2} \left\| \wtilde{u}^{(1)} - e^{-\imath 2\pi s_0 /L } \frac{u^{(1)}}{\|u^{(1)}\|}\right\|.
\end{align}
Since $ s_0 $ minimizes the expression in \eqref{eq:s_0 def}, by Lemma~\ref{lem:u_1 sigma_series_bound} we have that 
$$
\left\| \wtilde{u}^{(1)} - e^{-\imath 2\pi s_0 /L } \frac{u^{(1)}}{\|u^{(1)}\|}\right\| \leq \left\| \wtilde{u}^{(1)} - \alpha \frac{u^{(1)}}{\|u^{(1)}\|}\right\| \leq \|A_0\| + \sigma^2 \|A_2\|+ \sigma^3 \|A_3\|+ \sigma^4 \|A_4\|,
$$
and thus, from \eqref{eq:lim_theta_thetawitlde_1} we have,
\begin{equation*}
\left\Vert \alpha  \mathcal{R}_{{s}} \left\{\theta\right\}  - \wtilde{\theta}  \right\Vert \leq b_0 + \sigma^2 b_2 + \sigma^3 b_3 + \sigma^4 b_4,
\end{equation*}
where 
\begin{equation}\label{eq:b4_def}
 b_4 = \frac{\dbtilde{C}L\sqrt{L}}{\delta_1\lambda^2 N} + \wtilde{C} \frac{L}{\delta_1^2} \|A_4\|,
\end{equation} 
where $ A_4 $ is from \eqref{eq:A4'_def} and \eqref{eq:Ai_def}.

By Theorem~\ref{thm: step 1 A4 err} we have that 
\begin{equation}\label{eq:limutilde_us_0_2}
\frac{\wtilde{C}L}{\delta_1^2}  \|A_4\|  \leq \frac{\wtilde{C}L}{\delta_1^2}~ \frac{C_3}{\gamma_1\delta_1^2} \sqrt{\frac{L^3}{\lambda^4 N}},
\end{equation}
with probability at least 
\begin{equation*}
1 - C_1 L e^{-c_1 {L}^{1/2}} - C_2 N e^{-c_2 N^{1/4}}. 
\end{equation*}
Combining \eqref{eq:b4_def} and \eqref{eq:limutilde_us_0_2} we have that there is some constant $ C_4 $, such that
\begin{equation}
b_4 \leq \frac{C_4}{\gamma_1 \delta_1^4} \sqrt{\frac{L^5}{\lambda^4 N}}, 
\end{equation}
with probability at least 
\begin{equation*}
1 - C_1 L e^{-c_1 {L}^{1/2}} - C_2 N e^{-c_2 N^{1/4}}. 
\end{equation*}

\end{proof}

\section{Proof of Theorem~\ref{thm:alg 2 consistency}} \label{sec:proof of consistency of alg 2}
In order to show that $ \wtilde{\theta} $ converges to  $ \theta $ (up to the inherent ambiguities), we will show that $ \tilde{\theta}_m[k] $ (from Step~\ref{algstep:algAM_fixPower} of Algorithm~\ref{alg:MRFA}) converges to $ \hat{\theta}_m[k] = |\hat{\theta}[k]| $ and $ \wtilde{\theta}_p[k] $ (from Step~\ref{algstep:algAM_fixPhases} of Algorithm~\ref{alg:MRFA}) converges to $ \alpha \hat{\theta}_p[k] e^{2\pi \imath kj/L}$ where $ \hat{\theta}_p[k] = {\hat{\theta}[k]} / {| \hat{\theta}[k] |}$, for $k=0,\ldots,L-1 $, and for some $ j \in \{0,\ldots, L-1\} $ and $ \alpha\in \CC $ such that $ |\alpha|=1 $. Since $ \wtilde{\theta}_m $ is computed in the same way in Algorithm~\ref{alg:MRFA} and Algorithm~\ref{alg:MRFA_formal}, we already proved in the analysis of Algorithm~\ref{alg:MRFA_formal} that (see~\eqref{eq:theta_hat_power_approx})
\begin{equation} \label{eq:theta_tilde_m converges}
\wtilde{\theta}_m \underset{\text{a.s.,\;} N\rightarrow \infty}{\longrightarrow} \hat{\theta}_m.
\end{equation}

We now show that $ \wtilde{\theta}_p[k] 
\underset{\text{a.s.,\;} N\rightarrow \infty}{\longrightarrow}
\alpha  \hat{\theta}_p[k] e^{2\pi \imath kj/L} $.
Define the ``clean version" of $ \wtilde{C}_x $ (from~\eqref{eq:C_x def}), as
\begin{equation}
{C}_x[k_1,k_2] = 
\begin{dcases}
1 &, k_1=k_2, \\
\frac{{u}^{(k_2-k_1)\operatorname{mod} L}[k_1]}{\left\vert {u}^{(k_2-k_1)\operatorname{mod} L}[k_1] \right\vert } &, k_1 \neq k_2.	
\end{dcases}
\label{eq:C_x clean def}
\end{equation}
In~\eqref{eq:C_x observation no noise} we showed that 
\begin{equation} \label{eq:C_x eq}
{C}_x = \hat{\theta}_p \hat{\theta}_p^* \odot \operatorname{Circul}\left\{ \alpha \right\}.
\end{equation}
As was shown in Appendix~\ref{sec:proof of consistency of alg 1} (equation~\eqref{eq:u_m est convergence}) 
\begin{equation} \label{eq:u_m_tilde_lim}
\wtilde{u}^{(m)} \underset{\text{a.s.,\;} N\rightarrow \infty}{\longrightarrow} \alpha_m \frac{u^{(m)}}{\left\Vert u^{(m)} \right\Vert}.
\end{equation}
Thus, from~\eqref{eq:C_x eq},~\eqref{eq:u_m_tilde_lim}, and~\eqref{eq:C_x def}, it is easy to see that
\begin{equation}
\wtilde{C}_x \underset{\text{a.s.,\;} N\rightarrow \infty}{\longrightarrow} \hat{\theta}_p \hat{\theta}_p^* \odot \operatorname{Circul}\left\{ \alpha \right\}.
\end{equation}
Recall that in Step~\ref{algstep:altmin_step1} of Algorithm~\ref{alg:MRFA}, we solve
\begin{equation}
\wtilde{q}_{1}\gets \argmin_{\wtilde{q}\in\mathbb{C}^{L}}{\left\Vert {\wtilde{q}}{\wtilde{q}}^* - \wtilde{C}_x \odot \operatorname{Circul}\left\{\wtilde{\alpha}_0^* \right\} \right\Vert_F}, \label{eq:am step 1 proof}
\end{equation}
where 
\begin{equation}
\wtilde{\alpha}_0[k] = \exp \left\{\imath 2\pi {\phi}[k] \right\}, \quad {\phi}[k]\sim U[0,1),\quad k=0,\ldots,L-1.
\end{equation}
Note that, 
\begin{equation} \label{eq:C_x lim}
\wtilde{C}_x \odot \operatorname{Circul}\left\{\wtilde{\alpha}_0^* \right\} \underset{\text{a.s.,\;} N\rightarrow \infty}{\longrightarrow} \hat{\theta}_p \hat{\theta}_p^* \odot \operatorname{Circul}\left\{ \beta_0,\beta_1,\ldots,\beta_{L-1} \right\},
\end{equation}
where
\begin{equation}
\beta_k = \alpha_k \wtilde{\alpha}_0[k] = \exp \left\{\imath 2\pi {\phi^{'}_k} \right\}, \quad \phi^{'}_k = \log \alpha[k] + \phi[k] .
\end{equation}
Note also that  $ \beta_k$ is uniformly distributed on the unit circle for $ k=0,\ldots,L-1 $. Denote  
\begin{equation*}
H \triangleq \hat{\theta}_p \hat{\theta}_p^* \odot \operatorname{Circul}\left\{ \beta_0,\beta_1,\ldots,\beta_{L-1} \right\}.
\end{equation*}
Since the solution of~\eqref{eq:am step 1 proof} is the eigenvector corresponding to the leading eigenvalue of $ \wtilde{C}_x \odot \operatorname{Circul}\left\{\wtilde{\alpha}_0^* \right\} $, from~\eqref{eq:C_x lim} by the Davis-Kahan $\sin{\Theta}$ theorem~\cite{davis1970rotation}, we have that $ \wtilde{q}_1 $ converges to an eigenvector of $ H $. 

Next, since the discrete Fourier transform diagonalizes circulant matrices \cite{davis2013circulant}, and by simple algebra, it can be shown that
\begin{equation}
H =
 \operatorname{Diag}\left\{ \hat{\theta}_p \right\} \cdot F \cdot \operatorname{Diag}\left\{ F \cdot \left[  \beta_0,\beta_1,\ldots,\beta_{L-1} \right]^T \right\} \cdot F^* \cdot \operatorname{Diag}\left\{ \hat{\theta}_p^* \right\},
\end{equation}
where $ F $  is the $ L\times L $ discrete Fourier transform matrix defined in~\eqref{eq:DFT_def}. Then, it is evident that
\begin{equation}
{V} \triangleq \operatorname{Diag}\left\{ \hat{\theta}_p \right\} \cdot F
\end{equation}
is a unitary matrix containing the eigenvectors of $H$, and $ F \cdot \left[  \beta_0,\beta_1,\ldots,\beta_{L-1} \right]^T$ are the eigenvalues of $H$. Since $F$ is a unitary matrix, and $\beta_0,\ldots,\beta_{L-1}$ are strictly continuous i.i.d. on the unit circle, the eigenvalues of $H$ are distinct with probability $1$. Then, since $\wtilde{q}_{1}$ from~\eqref{eq:am step 1 proof} converges to one of the eigenvectors of $ H $, it converges to one of the columns of $V$ up to a constant factor of $ \alpha \in \CC $, where $ |\alpha|=1 $. 

Denote by $ V_j $ the $ j $'s column of V. Note the special structure of $ V_j $,
\begin{equation*}
V_j[k] = \hat{\theta}_p[k] e^{2\pi \imath kj/L} ~~~ k=0,\ldots,L-1.
\end{equation*}
Thus, $\wtilde{q}_{1}$ converges to $\hat{\theta}_p$ almost surely, up to a constant factor and an unknown modulation, or, explicitly, 
\begin{equation}\label{eq:theta_tilde_p converges}
\wtilde{q}_1[k] 
\underset{\text{a.s.,\;} N\rightarrow \infty}{\longrightarrow}
\alpha  \hat{\theta}_p[k] e^{2\pi \imath kj/L},
\end{equation} 
for some $ \alpha\in \CC $, $ |\alpha| = 1 $, and some $ j\in\{0,\ldots,L-1\} $.
Since $\wtilde{\theta}_m$ is a consistent estimator for the magnitudes of $\hat{\theta}$ (see~\eqref{eq:theta_tilde_m converges}), and $\wtilde{q}_1$ is a consistent estimator for the phases (see~\eqref{eq:theta_tilde_p converges}), by Lemma~\ref{lem:bound_on combinning_powerErr_phaseErr}, the combination of them in Step~\ref{algstep:est theta hat} provides a consistent estimate for $\hat{\theta}$. Thus, computing the inverse Fourier transform in Step~\ref{algstep:invFT_algAM} of  Algorithm~\ref{alg:MRFA} provides a consistent estimate for $\theta$ up to an unknown factor and a cyclic shift.

\end{appendices}

\section{Acknowledgments}
\noindent We would like to thank Prof. Boaz Nadler and Dr. Tamir Bendory for their remarks and comments.
This research was partially supported by the European Research Council (ERC) under the
European Unions Horizon 2020 research and innovation programme (grant agreement
723991 - CRYOMATH), by Award Number R01GM090200 from the NIGMS, and by a Fellowship from Jyv\"{a}skyl\"{a} University and the Clore Foundation.

\bibliography{mybib}{}
\bibliographystyle{plain}

\end{document}